\documentclass{amsart}
\usepackage{graphicx}
\usepackage{amssymb}
\usepackage{amsfonts}
\usepackage{todonotes}
\usepackage{dsfont}
\usepackage{amsmath,amsfonts,amsthm}
\usepackage{mathrsfs}
\swapnumbers
\newtheorem{theorem}{Theorem}[section]
\newtheorem{lemma}[theorem]{Lemma}
\newtheorem{corollary}[theorem]{Corollary}
\newtheorem{proposition}[theorem]{Proposition}
\theoremstyle{definition}

\newtheorem{remark}[theorem]{Remark}
\newtheorem{example}[theorem]{Example}

\numberwithin{equation}{section}
 \theoremstyle{plain}
 \newtheorem{thm}{Theorem}[section]
 \numberwithin{equation}{section} 
 \numberwithin{figure}{section} 
 \theoremstyle{plain}
 \theoremstyle{remark}
 \newtheorem*{acknowledgement*}{Acknowledgement}

\newcommand{\cB}{{\mathcal B}}

\newcommand{\cF}{{\mathcal F}}
\newcommand{\cG}{{\mathcal G}}
\newcommand{\cH}{{\mathcal H}}
\newcommand{\cI}{{\mathcal I}}
\newcommand{\cJ}{{\mathcal J}}

\newcommand{\cL}{{\mathcal L}}

\newcommand{\cP}{{\mathcal P}}

\newcommand{\cW}{{\mathcal W}}

\newcommand{\te}{{\theta}}
\newcommand{\vt}{{\vartheta}}
\newcommand{\Om}{{\Omega}}
\newcommand{\om}{{\omega}}
\newcommand{\ve}{{\varepsilon}}
\newcommand{\del}{{\delta}}
\newcommand{\Del}{{\Delta}}
\newcommand{\gam}{{\gamma}}
\newcommand{\Gam}{{\Gamma}}

\newcommand{\vr}{{\varrho}}
\newcommand{\Sig}{{\Sigma}}
\newcommand{\sig}{{\sigma}}
\newcommand{\al}{{\alpha}}
\newcommand{\be}{{\beta}}
\newcommand{\ka}{{\kappa}}

\newcommand{\vrho}{{\varrho}}

\newcommand{\vs}{{\varsigma}}

\newcommand{\bbB}{{\mathbb B}}

\newcommand{\bbR}{{\mathbb R}}
\newcommand{\bbS}{{\mathbb S}}

\newcommand{\bbZ}{{\mathbb Z}}
\newcommand{\bbI}{{\mathbb I}}
\newcommand{\bbW}{{\mathbb W}}

\newcommand{\bbQ}{{\mathbb Q}}
\newcommand{\bbU}{{\mathbb U}}
\newcommand{\bbV}{{\mathbb V\underline{}}}

\newcommand{\R}{\mathbb{R}}

\newcommand{\op}{\mathcal{P}}

\newcommand{\eps}{\varepsilon}

\begin{document}
\title[]{Almost sure diffusion approximation in averaging\\
via rough paths theory}%
 \vskip 0.1cm
 \author{Peter Friz\quad\quad\quad and \quad\quad\quad Yuri Kifer\\
\vskip 0.1cm
Institut f\" ur Mathematik \quad\quad\quad\quad\quad\quad  Institute  of Mathematics\\
 Technische Universit\" at\quad\quad\quad\quad\quad\quad\quad\quad Hebrew University\\
\& WIAS Berlin - Berlin, Germany.\quad\quad\quad\quad Jerusalem, Israel}
\address{
Institut f\" ur Mathematik, Technische Universit\" at, Berlin, Germany}
\email{friz@math.tu-berlin.de}
\address{
WIAS Berlin, Berlin, Germany}
\email{friz@wias-berlin.de}
\address{
Institute of Mathematics, The Hebrew University, Jerusalem 91904, Israel}
\email{ kifer@math.huji.ac.il}%

\thanks{ }
\dedicatory{  }
 \date{\today}
\begin{abstract}\noindent
The paper deals with the fast-slow motions setups in the continuous time
$\frac {dX^\ve(t)}{dt}=\frac 1\ve\sig(X^\ve(t))\xi(t/\ve^2)+b(X^\ve(t)),\, t\in [0,T]$ and the discrete time $X_N((n+1)/N)=X_N(n/N)+N^{-1/2}\sig(X_N(n/N))\xi(n)+N^{-1}b(X_N(n/N))\xi(n)$, $n=0,1,...,[TN]$
 where $\sig$ and $b$ are smooth matrix and vector functions, respectively, $\xi$ is a centered stationary
 vector stochastic process and $\ve, 1/N$ are small parameters. We derive, first, estimates in the strong
 invariance principles for sums
 $S_{N}(t)=N^{-1/2}\sum_{0\leq k< [Nt]}\xi(k)$ and iterated sums $\bbS^{ij}_{N}(t)=N^{-1}\sum_{0\leq k<l<[Nt]}\xi_i(k)\xi_j(l)$ together with the corresponding results for integrals in the continuous time case
  which, in fact, yields almost sure invariance principles for iterated sums and integrals of any order and, moreover,
 implies laws of iterated logarithm for these objects.
 Then, relying on the rough paths theory, we obtain strong almost sure approximations of processes $X^\ve$ and
 $X_N$ by corresponding diffusion processes $\Xi^\ve$ and $\Xi_N$, respectively. Previous results for the above
  setup dealt either with weak or moment diffusion approximations and not with almost sure approximation
  which is the new and natural generalization of well known works on strong invariance principles for sums of
  weakly dependent random variables.
\end{abstract}
\maketitle
\tableofcontents
\markboth{P.Friz and Yu.Kifer}{Diffusion approximation}
\renewcommand{\theequation}{\arabic{section}.\arabic{equation}}
\pagenumbering{arabic}

\section{Introduction}\label{roughsec1}\setcounter{equation}{0}
The study of the asymptotic behavior as $\ve\to 0$ of solutions $X^\ve$ of systems of ordinary differential equations having the form
\begin{equation}\label{1.1}
\frac {dX^\ve(t)}{dt}=\frac 1\ve B(X^\ve(t),\xi(t/\ve^2))+b(X^\ve(t),\,\xi(t/\ve^2)),\,\, t\in[0,T]
\end{equation}
has already more than a half century history. Here $B(\cdot,\xi(s))$ and $b(\cdot,\xi(s))$ are (random) smooth vector fields on $\bbR^d$ and $\xi$ is a stationary process which is viewed as a fast motion while $X^\ve$ is considered as a slow motion. Assuming that
\begin{equation}\label{1.2}
EB(x,\xi(0))\equiv 0
\end{equation}
for all $x$ it was shown in a series of papers \cite{Kha}, \cite{PK} and \cite{Bor} that $X^\ve$ converges weakly as $\ve\to 0$ to a diffusion process provided $\xi$ is sufficiently fast mixing with respect to $\sig$-algebras generated by $\xi$ itself. The latter condition is quite restrictive when $\xi$ is generated by a dynamical system, i.e. when $\xi(t)=g\circ F^t$ where $g$ is a vector function and $F^t$ is a flow (continuous time dynamical system) preserving certain measure which makes $\xi$
a stationary process. In order to derive weak convergence of $X^\ve$ to a diffusion for $\xi$ built by a sufficiently large class of observables $g$ and dynamical systems other approaches were developed recently based mainly on the rough paths
theory (see e.g. \cite{KM,KM2,BC,CFKMZorg,CFKMZ, DOP}). All above mentioned results can be obtained both in the continuous time setup (\ref{1.1}) and in the discrete time setup given by the following recurrence relation
\begin{equation}\label{1.3}
X^\ve((n+1)\ve^2)=X^\ve(n\ve^2)+\ve B(X^\ve(n\ve^2),\,\xi(n))+\ve^2b(X^\ve(n\ve^2),\,\xi(n))
\end{equation}
where $0\leq n<[T/\ve^2]$ and $\xi(n),\, n\geq 0$ is a stationary sequence of random
vectors. We will study (\ref{1.3}) for $X_N=X^{1/\sqrt N}$ as $N\to\infty$.
The above results can be viewed as a substantial generalization of the functional central limit theorem since when $B(x,\zeta)$ does not depend on $x$ and $b\equiv 0$ the process $X^\ve$ weakly converges to the Brownian motion (with a covariance matrix).

Motivated by strong invariance principle results for sums (see, for instance, \cite{BP}, \cite{KP}, \cite{DP} and more recent
\cite{MN} and \cite{Go}) the second author obtained in \cite{Ki20} and \cite{Ki21} certain results on strong $L^p$ diffusion
 approximations of solutions of (\ref{1.1}) and (\ref{1.3}) under the condition (\ref{1.2}). In this paper we are interested
 in strong diffusion approximation results in a more traditional form saying that $X^\ve$ (or $X_N$) can be redefined on
 a richer probability space where there exists a diffusion process $\Xi^\ve$ (or $\Xi_N$) such that $\ve^{-\del}\sup_{0\leq t
 \leq T}|X^\ve(t)-\Xi^\ve(t)|$ (or $N^\del\max_{0\leq n\leq TN}|X_N(n/N)-\Xi_N(n/N)|$) remains bounded almost surely (a.s.) for
all $\ve> 0$ (or $N\geq 1$) where $\del>0$ does not depend on $\ve$ or $N$. All papers on almost sure approximations (strong
 invariance principles) dealt before with sums (or integrals in time) of random variables (or vectors), and so the limiting
  process was always a Brownian motion while almost sure diffusion approximation results do not seem to appear before in the
   literature. These results do not follow from \cite{Ki20} and \cite{Ki21} and, on the other hand, our methods do not yield
   moment estimates from these papers.

 Our methods rely on the rough paths theory \cite{Lyo}, as exposed in \cite{FV,FH}, and in order to adapt the equations (\ref{1.1}) and (\ref{1.3}) to this setup
 we consider (as in \cite{KM, CFKMZorg}, but see Remark \ref{rem2.7}) a more restricted situation assuming $B(x,\xi(\cdot))=\sig(x)\xi(\cdot)$ where $\xi(k),\, k\geq 0$ is a
 $\phi$-mixing sequence of random vectors, $\sig(x)$ is a smooth  matrix function
 and $b(x,\xi(\cdot))=b(x)$ does not depend on the second variable. We obtain, first, strong invariance principles for sums
 $S_{N}(t)=N^{-1/2}\sum_{0\leq k< [Nt]}\xi(k)$ and iterated sums $\bbS^{ij}_{N}(t)=N^{-1}\sum_{0\leq k<l<[Nt]}\xi_i(k)\xi_j(l)$
 in $p$-variation rough path sense, 
  which is a stronger result than the standard strong invariance principle,
 and then, relying on a quantitative form of the local {Lipschitz} property of the It\^ o-Lyons map for c\`adl\`ag rough differential equations \cite{FZ},
 in conjunction with a law of iterated logarithm type growth control for Brownian rough paths, we obtain our estimates for strong diffusion approximations of processes $X^\ve$ and $X_N$. This is in contrast to \cite{KM,CFKMZorg} and subsequent works which ``only'' rely on continuity of the It\^ o-Lyons map.  Similar results will be
 derived here in the continuous time setup (\ref{1.1}) with $B(x,\xi(\cdot))=\sig(x)\xi(\cdot)$ where $\xi(t),\, t\geq 0$ is a
 vector stochastic process constructed by a suspension procedure over a $\phi$-mixing discrete time process. Again, we obtain
 first strong invariance principles for integrals $\ve\int_0^{t\ve^{-2}}\xi(s)ds$ and iterated integrals
 $\ve^2\int_0^{t\ve^{-2}}\xi_j(s)ds\int_0^s\xi_i(u)du$ and then rely on the rough paths machinery.

  In fact, rough paths theory
 yields an extension to almost sure invariance principles for iterated sums and integrals from second to any order and
  as a byproduct of these results for weakly dependent random vectors we obtain for these objects corresponding
 laws of iterated logarithm. Almost sure diffusion approximations, strong invariance principles for iterated sums and laws
 of iterated logarithms for them never appeared before in the literature. We also note the current interest of iterated integrals and their discrete counterparts from a data science / time-series perspective \cite{BKPASL, DEFT, CGGOT}.


 The structure of this paper is as follows. In the next section we describe our precise setup and main results. In Sections
 \ref{roughsec3} and \ref{roughsec4} we prove the strong invariance principles in appropriate variational norms for sums
 and iterated sums, respectively. In Section \ref{roughsec5} we obtain the strong invariance principles in the continuous
 time setup for integrals and iterated integrals. In Section 6 we provide a brief introduction
 to rough paths theory and show how together with the above strong invariance principles it yields our strong diffusion
 approximations results.  In Section 7 we obtain extensions to strong invariance principles for iterated sums and integrals
 of any order which yields also laws of iterated logarithm for these objects.

\subsection*{Acknowledgement} PKF's research is supported by DFG Excellence Cluster MATH+ through Projekt AA4-2 and a MATH+ Distinguished Fellowship.

\section{Preliminaries and main results}\label{roughsec2}\setcounter{equation}{0}
\subsection{Discrete time case}
We start with the discrete time setup which consists of a complete probability space
$(\Om,\cF,P)$, a stationary sequence of  $e$-dimensional random vectors $\xi(n)$, $-\infty<n<\infty$ and
a two parameter family of countably generated $\sig$-algebras
$\cF_{m,n}\subset\cF,\,-\infty\leq m\leq n\leq\infty$ such that
$\cF_{mn}\subset\cF_{m'n'}\subset\cF$ if $m'\leq m\leq n
\leq n'$ where $\cF_{m\infty}=\cup_{n:\, n\geq m}\cF_{mn}$ and
$\cF_{-\infty n}=\cup_{m:\, m\leq n}\cF_{mn}$.
We will measure the dependence between $\sig$-algebras $\cG$ and $\cH$ by the $\phi$-coefficient
 defined by
\begin{eqnarray}\label{2.1}
&\phi(\cG,\cH)=\sup\{\vert\frac {P(\Gam\cap\Del)}{P(\Gam)}-P(\Del)\vert:\, P(\Gam)\ne
0,\,\Gam\in\cG,\,\Del\in\cH\}\\
&=\frac 12\sup\{\| E(g|\cG)-Eg\|_\infty:\, g\,\,\mbox{is $\cH$-measurable and }\,\|g\|_\infty=1\}\nonumber
\end{eqnarray}
(see \cite{Bra}) where $\|\cdot\|_\infty$ is the $L^\infty$-norm. For each $n\geq 0$ we set also
\begin{equation}\label{2.2}
\phi(n)=\sup_m\phi(\cF_{-\infty,m},\cF_{m+n,\infty}).
\end{equation}
If $\phi(n)\to 0$ as $n\to\infty$ then the probability measure $P$ is called $\phi$-mixing with respect to the family $\{\cF_{mn}\}$. Unlike \cite{Ki20}, in order
to ensure more applicability of our results to dynamical systems, we do not assume that
$\xi(n)$ is $\cF_{nn}$-measurable and instead we will work with the approximation
 coefficient
\begin{equation}\label{2.3}
\rho(n)=\sup_m\|\xi(m)-E(\xi(m)|\cF_{m-n,m+n})\|_\infty.
\end{equation}
To save notations we will still write $\cF_{mn}$, $\phi(n)$ and $\rho(n)$ for $\cF_{[m][n]}$, $\phi([n])$ and $\rho([n])$, respectively, if $m$ and $n$ are not integers, where $[\cdot]$ denotes the integral part.

We will deal with the recurrence relation
\begin{equation}\label{2.4}
X_N(n+1/N)=X_N(n/N)+\frac 1{\sqrt N}\sig(X_N(n/N))\xi(n)+\frac 1Nb(X_N(n/N))
\end{equation}
and this definition is extended to all $t\in[0,T]$ by setting $X_N(t)=X_N(n/N)$ whenever
$n/N\leq t<(n+1)/N$. We will assume that
\begin{equation}\label{2.5}
E\xi(0)=0
\end{equation}
and that $\sig$ is a $d \times e$ matrix function and $b$ is a $d$-dimensional vector function both defined on $\bbR^d$.
To avoid excessive technicalities these coefficients
and the process $\xi$ are supposed to satisfy the following uniform bounds
\begin{equation}\label{2.6}
\|\sig\|_{C^3},\,\|b\|_{C^3}\leq L
\end{equation}
and
\begin{equation}\label{2.7}
\|\xi(0)\|_\infty\leq L
\end{equation}
for some $L\geq 1$, where $\|\cdot\|_{C^3}$ is a matrix or a vector function $C^3$ norm and $\|\cdot\|_\infty$ is the $L^\infty$ norm.

Introduce the $e\times e$ matrix $\vs=(\vs_{ij})$ where
\begin{equation}\label{2.8}
\vs_{ij}=\lim_{k\to\infty}\frac 1k\sum_{m=0}^k\sum_{n=0}^k\vs_{ij}(n-m),\,\,
\mbox{and}\,\, \vs_{il}(n-m)=E(\xi_i(m)\xi_l(n))
\end{equation}
taking into account that the limit above will exist under conditions of our theorem below (see
(\ref{4.21})--(\ref{4.22}) in Section \ref{roughsec4} below). Define also
\begin{equation}\label{2.9}
c_i(x)=\sum_{j,l=1}^e\sum_{k=1}^d\frac {\partial\sig_{ij}(x)}{\partial x_k}\hat\vs_{lj}\sig_{kl}(x),\, i=1,...,d
\end{equation}
where
\begin{equation}\label{2.10}
\hat\vs_{jl}=\lim_{k\to\infty}\frac 1k\sum_{n=0}^k\sum_{m=-k}^{n-1}\vs_{jl}(n-m)=\sum_{n=1}^\infty E(\xi_j(0)\xi_l(n))
\end{equation}
and the latter limit exists under our conditions (which follows from (\ref{4.21})--(\ref{4.22}) below).
Let $\Xi$ be the unique solution of the stochastic differential equation
\begin{equation}\label{2.11}
d\Xi(t)=\sig(\Xi(t))dW(t)+(b(\Xi(t))+c(\Xi(t)))dt
\end{equation}
where $W$ is the $e$-dimensional Brownian motion with the covariance matrix $\vs$ (at the time 1). In what follows we will write
$A(u)=O(\al(u))$ a.s. for a family of random variables $A(u),\, u\in U\subset\bbR$ and a nonrandom sequence $\al(u),\,
u\in U\subset\bbR$ if $A(u)/\al(u)$ is bounded for all $u\in U$ by an almost surely finite random variable.

\begin{theorem}\label{thm2.1} Suppose that $X_N$ is defined by (\ref{2.4}), (\ref{2.5})--(\ref{2.7}) hold true and assume that
\begin{equation}\label{2.12}
\sup_{n\geq 0}n^4(\phi(n)+\rho(n))<\infty.
\end{equation}
Then the stationary sequence of random vectors
 $\xi(n),\,-\infty<n<\infty$ can be redefined preserving its distributions on a sufficiently rich probability space which contains also a $e$-dimensional Brownian motion $\cW$ with the covariance matrix $\vs$ (at the time 1) so that for each $T>0$ and $\Xi_N$
 solving (\ref{2.11}) with $W(t)=W_N(t)=N^{-1/2}\cW(Nt),\, 0\leq t\leq T$,
\begin{equation}\label{2.13}
\sup_{0\leq t\leq T}|X_N(t)-\Xi_N(t)|=O(N^{-\del})\quad\mbox{a.s. for all}\quad N\geq 1
\end{equation}
 where $X_N(0)=\Xi_N(0)$ and $\del>0$ does not depend on $N\geq 1$.
\end{theorem}

We stress that a sufficiently rich probability space has here a quite precise meaning that one can define on it a sequence of 
independent uniformly didtributed on $[0.1]$ random variables which are independent of the sequence $\xi(n),\, n\in\bbZ$ (see Theorem \ref{thm3.8} below).

Under certain additional conditions, which are always satisfied when $d=1$ and $\sig$ is bounded away from zero, the problem
can be reduced to the strong invariance principle for sums (cf. \cite{Ki21}) which is well known and this would imply then
Theorem \ref{thm2.1}. Nevertheless, in the multidimensional case such reduction requires substantial additional assumptions
 on $\sig$ (see \cite{Ki21}).
For the proof of Theorem \ref{thm2.1} we will follow, first, the strategy appeared more than forty years ago in \cite{BP} leading 
to the strong (almost sure) invariance principle in the supremum norm. Moreover, we extend this to  
 estimates of the almost sure invariance principles in a $p$-variation norm both for sums $S_N$ and iterated sums $\bbS_N$ given by
\begin{equation}\label{2.14}
S_{N}(t)=N^{-1/2}\sum_{0\leq k< [Nt]}\xi(k)\quad\mbox{and}\quad \bbS^{ij}_{N}(t)=N^{-1}\sum_{0\leq k<l<[Nt]}\xi_i(k)\xi_j(l),
\end{equation}
where $S_N(0)=\bbS^{ij}_N(0)=0$, and then rely on the local Lipschitz property of the It\^ o-Lyons map for rough differential equations (see, for instance, \cite{FH}) together with the growth estimates of the corresponding Lipschitz constant
(see Theorem 6.1 in Section 6). Observe also that by (\ref{2.10}),
\[
\lim_{N\to\infty}E\bbS^{ij}_{N}(t)=t\hat\vs_{ij}.
\]

As it is customary in the rough paths theory we use slightly different definitions for processes denoted by usual letters and
the ones denoted by blackboard letters. Namely, for a process $Q(t),\, t\geq 0$ we write $Q(s,t)=Q(t)-Q(s)$ when $t\geq s$.
On the other hand, if $\bbQ(t)=\sum_{0\leq k<l<[tN]}\eta(k)\zeta(l)$ or $\bbQ(t)=\int_0^t\eta(u)d\zeta(u)$ (where $\eta$ and
$\zeta$ are one-dimensional processes and when the latter integral makes sense) then
\[
\bbQ(s,t)=\sum_{[sN]\leq k<l<[tN]}\eta(k)\zeta(l)\,\,\mbox{and}\,\,\bbQ(s,t)=\int_s^t(\eta(u)-\eta(s))d\zeta(u),
\]
respectively (see Section \ref{sec:RP} for some recalls). These notations affect the following definition of $p$-variation norms. For any path $\gam(t),\, t\in[0,T]$ in a Euclidean space having left and right limits and $p\geq 1$ the $p$-variation norm of
$\gam$ on  an interval $[U,V],\, U<V$ is given by
 \begin{equation}\label{2.15}
 \|\gam\|_{p,[U,V]}=\big(\sup_\cP\sum_{[s,t]\in\cP}|\gam(s,t)|^p\big)^{1/p}
 \end{equation}
 where the supremum is taken over all partitions $\cP=\{ U=t_0<t_1<...<t_n=V\}$ of $[U,V]$ and the sum is taken over
 the corresponding subintervals $[t_i,t_{i+1}],\, i=0,1,...,n-1$ of the partition while $\gam(s,t)$ is taken according
 to the definitions above depending on the process under consideration. The main step in the proof of Theorem \ref{thm2.1} is the following result.
 \begin{theorem}\label{thm2.2}
 Suppose that (\ref{2.12}) holds true. Then the stationary sequence of random vectors
 $\xi(n),\,-\infty<n<\infty$ can be redefined preserving its distributions on a sufficiently rich probability space which contains also a $e$-dimensional Brownian motion $\cW$ with the covariance matrix $\vs$ (at the time 1) so that for each $T>0$ and
 $W_N(t)=N^{-1/2}\cW(Nt),\, 0\leq t\leq T$,
 \begin{equation}\label{2.16}
 \| S_N-W_N\|_{p,[0,T]}=O(N^{-\del})\quad\mbox{a.s.}
 \end{equation}
 and
  \begin{equation}\label{2.17}
 \max_{1\leq i,j\leq d}\|\bbS^{ij}_N-\bbW^{ij}_N\|_{\frac p2,[0,T]}=O(N^{-\del})\quad\mbox{a.s.}
 \end{equation}
 where $p\in(2,3)$, $\del>0$ does not depend on $N\geq 1$ (and may be different from $\del$ in (\ref{2.13})), $S_N,\,\bbS_N$
  are given  by (\ref{2.14}) and
 \begin{equation}\label{2.18}
 \bbW^{ij}_N(t)=\int_0^tW^i_N(s)dW^j_N(s)+t\sum_{l=1}^\infty E(\xi_i(0)\xi_j(l)).
 \end{equation}
\end{theorem}
Observe that the estimate (\ref{2.16}) obtained in the $p$-variation norm is a stronger result than the standard strong
invariance principle considered in other papers such as \cite{BP}, \cite{KP}, \cite{DP}, \cite{MN} and \cite{Go}.
The strong invariance principle for $S_N$, i.e. an a.s. eventual estimate of $|S_N-W_N|$ in the supremum norm, is
essentially well known but in \cite{BP} and \cite{KP} it is proved assuming that the $\sig$-algebras $\cF_{mn}$ are
 generated by the random vectors $\xi(m),...,\xi(n)$ themselves (which restricts applications) while in \cite{MN} and
 \cite{Go} this is proved also under different from ours conditions. Thus, we cannot rely here on a direct reference and will
 provide the proof of this part, as well. In fact, we will need not only this result itself but also the specific construction
 of the Brownian motions $W_N$ emerging there which is used in the proof of (\ref{2.17}).
 The strong invariance principle for iterated sums in the form (\ref{2.17}) does not seem to appear in the literature before.
 Moreover, we will see in Section 7 that the strong invariance principle for second order iterated sums obtained in Theorem
 \ref{thm2.2} implies strong invariance principles for multiple iterated sums of any order which, in turn, yields laws of
 iterated logarithm for them relying on the laws of iterated logarithm for multiple stochastic integrals
 $\int_0^td\cW^{i_\ell}(t_\ell)\int_0^{t_\ell}...\int_0^{t_2}d\cW^{i_1}(t_1)$  (see \cite{Bal}).

  \begin{corollary}\label{cor2.3}
  Under the conditions of Theorem \ref{thm2.2} for all $\ell\geq 1$ and $0\leq i_1,...,i_\ell\leq d$,
   \begin{equation*}
  P\big\{\limsup_{t\to\infty}\frac {\sum_{0\leq k_1<k_2<...<k_\ell<t}\xi_{i_1}(k_1)\xi_{i_2}(k_2)\cdots\xi_{i_\ell}(k_\ell)}
  {(2t\log\log t)^{\ell/2}}=d_{i_1,i_2,...,i_\ell}\big\}=1
  \end{equation*}
  where $1\le i_1\le i_2\leq ...\leq i_\ell\le e$ and the constant $d_{i_1,i_2,...,i_\ell}$ can be computed from
  the formula given in Corollary \ref{cor:BaldiGen}.
  \end{corollary}

Important classes of processes satisfying our conditions come from
dynamical systems. Let $F$ be a $C^2$ Axiom A diffeomorphism (in
particular, Anosov) in a neighborhood of an attractor or let $F$ be
an expanding $C^2$ endomorphism of a Riemannian manifold $\Om$ (see
\cite{Bow}), $g$ be either a H\" older continuous vector function or a
vector function which is constant on elements of a Markov partition and let $\xi(n)=
\xi(n,\om)=g(F^n\om)$. Here the probability space is $(\Om,\cB,P)$ where $P$ is a Gibbs
 invariant measure corresponding to some H\"older continuous function and $\cB$ is the Borel $\sig$-field.
 In this setup the assumption (\ref{2.6}) on boundedness of $\xi$ turns out to be quite natural.
  Let $\zeta$ be a finite Markov partition for $F$ then we can take $\cF_{kl}$
 to be the finite $\sig$-algebra generated by the partition $\cap_{i=k}^lF^i\zeta$.
 In fact, we can take here not only H\" older continuous $g$'s but also indicators
of sets from $\cF_{kl}$. The conditions of Theorems \ref{thm2.1} and \ref{thm2.2} allow all such functions
since the dependence of H\" older continuous functions on $m$-tails, i.e. on events measurable
with respect to $\cF_{-\infty,-m}$ or $\cF_{m,\infty}$, decays exponentially fast in $m$ and
the condition (\ref{2.12}) is even weaker than that. A related class of dynamical systems
corresponds to $F$ being a topologically mixing subshift of finite type which means that $F$
is the left shift on a subspace $\Om$ of the space of one (or two) sided
sequences $\om=(\om_i,\, i\geq 0), \om_i=1,...,l_0$ such that $\om\in\Om$
if $\pi_{\om_i\om_{i+1}}=1$ for all $i\geq 0$ where $\Pi=(\pi_{ij})$
is an $l_0\times l_0$ matrix with $0$ and $1$ entries and such that $\Pi^n$
for some $n$ is a matrix with positive entries. Again, we have to take in this
case $g$ to be H\" older continuous bounded functions on the sequence space above,
 $P$ to be a Gibbs invariant measure corresponding to some H\" older continuous function and to define
$\cF_{kl}$ as the finite $\sig$-algebra generated by cylinder sets
with fixed coordinates having numbers from $k$ to $l$. The
exponentially fast $\psi$-mixing, which is stronger than $\phi$-mixing
required here, is well known in the above cases (see \cite{Bow}). Among other
dynamical systems with exponentially fast $\psi$-mixing we can mention also the Gauss map
$Fx=\{1/x\}$ (where $\{\cdot\}$ denotes the fractional part) of the
unit interval with respect to the Gauss measure $G$ and more general transformations generated
by $f$-expansions (see \cite{Hei}). Gibbs-Markov maps which are known to be exponentially fast
$\phi$-mixing (see, for instance, \cite{MN}) can be also taken as $F$ in Theorem \ref{thm2.1}
with $\xi(n)=g\circ F^n$ as above.

\subsection{Continuous time case}

Here we start with a complete probability space $(\Om,\cF,P)$, a
$P$-preserving invertible transformation $\vt:\,\Om\to\Om$ and
a two parameter family of countably generated $\sig$-algebras
$\cF_{m,n}\subset\cF,\,-\infty\leq m\leq n\leq\infty$ such that
$\cF_{mn}\subset\cF_{m'n'}\subset\cF$ if $m'\leq m\leq n
\leq n'$ where $\cF_{m\infty}=\cup_{n:\, n\geq m}\cF_{mn}$ and
$\cF_{-\infty n}=\cup_{m:\, m\leq n}\cF_{mn}$. The setup includes
also a (roof or ceiling) function $\tau:\,\Om\to (0,\infty)$ such that
for some $\hat L>0$,
\begin{equation}\label{2.19}
\hat L^{-1}\leq\tau\leq\hat L.
\end{equation}
Next, we consider the probability space $(\hat\Om,\hat\cF,\hat P)$ such that $\hat\Om=\{\hat\om=
(\om,t):\,\om\in\Om,\, 0\leq t\leq\tau(\om)\},\, (\om,\tau(\om))=(\vt\om,0)\}$, $\hat\cF$ is the
restriction to $\hat\Om$ of $\cF\times\cB_{[0,\hat L]}$, where $\cB_{[0,\hat L]}$ is the Borel
$\sig$-algebra on $[0,\hat L]$ completed by the Lebesgue zero sets, and for any $\Gam\in\hat\cF$,
\[
\hat P(\Gam)=\bar\tau^{-1}\int\bbI_\Gam(\om,t)dtdP(\om)\,\,\mbox{where}\,\,\bar\tau=\int\tau dP=E\tau
\]
and $E$ denotes the expectation on the space $(\Om,\cF,P)$.

Finally, we introduce a vector valued stochastic process $\xi(t)=\xi(t,(\om,s))$, $-\infty<t<\infty,\, 0\leq
s\leq\tau(\om)$ on $\hat\Om$ satisfying
\begin{eqnarray*}
&\xi(t,(\om,s))=\xi(t+s,(\om,0))=\xi(0,(\om,t+s))\,\,\mbox{if}\,\, 0\leq t+s<\tau(\om)\,\,\mbox{and}\\
&\xi(t,(\om,s))=\xi(0,(\vt^k\om,u))\,\,\mbox{if}\,\, t+s=u+\sum_{j=0}^k\tau(\vt^j\om)\,\,\mbox{and}\,\,
0\leq u<\tau(\vt^k\om).
\end{eqnarray*}
This construction is called in dynamical systems a suspension and it is a standard fact that $\xi$ is a
stationary process on the probability space $(\hat\Om,\hat\cF,\hat P)$ and in what follows we will write
also $\xi(t,\om)$ for $\xi(t,(\om,0))$.

We will assume that $X^\ve(t)=X^\ve(t,\om)$ from (\ref{1.1}) considered as a process on $(\Om,\cF,P)$
solves the equation
\begin{equation}\label{2.20}
\frac {dX^\ve(t)}{dt}=\frac 1\ve \sig(X^\ve(t))\xi(t/\ve^2)+b(X^\ve(t)),\,\, t\in[0,T]
\end{equation}
where the matrix function $\sig$ and the process $\xi$ satisfy (\ref{2.6}). Set $\eta(\om)=\int_0^{\tau(\om)}\xi(s,\om)ds$ and
\begin{eqnarray}\label{2.21}
&\rho(n)=\sup_m\max\big(\|\tau\circ\vt^m-E(\tau\circ\vt^m|\cF_{m-n,m+n})\|_\infty,\\
&\| ess\sup_{0\leq s\leq\tau(\te^m\om)}|\xi(0,(\vt^m\om,s))-E(\xi(0,(\vt^m\om,s))|\cF_{m-n,m+n})|\|_\infty\big).\nonumber
\end{eqnarray}
Observe also that $\eta(k)=\eta\circ\vt^k$ is a stationary sequence of random vectors.

Next, we consider a diffusion process $\Xi$ solving the stochastic differential equation
\begin{equation}\label{2.11bis}
d\Xi(t)=\sig(\Xi(t))dW(t)+(\bar\tau b(\Xi(t))+c(\Xi(t)))dt,
\end{equation}
with $d$-dimensional Brownian motion $W$ having the covariance matrix $\vs=(\vs_{ij})$ at the time 1 given by
\begin{equation}\label{2.22}
\vs_{ij}=\lim_{n\to\infty}\frac 1n\sum_{k,l=0}^nE(\eta_i(k)\eta_j(l))
\end{equation}
and with
\begin{eqnarray}\label{2.23}
&
\quad c_{i}(x)=\sum_{j,l=1}^e\sum_{k=1}^d\frac {\partial\sig_{ij}(x)}{\partial x_k}\Big(\hat\vs_{lj}+
\int_0^{\tau(\om)}
 \xi_j(s,\om) ds
 \int_0^s
\xi_l(u,\om) du
\Big)\sig_{kl}(x) , \\
&\hat\vs_{ij}=\lim_{n\to\infty}\frac 1n\sum_{k=0}^n\sum_{l=-n}^{k-1}E(\eta_i(l)\eta_j(k))=\sum_{n=1}^\infty
E(\eta_i(0)\eta_j(n)).\nonumber
\end{eqnarray}
The limits in (\ref{2.22}) and (\ref{2.23}) exist in view of estimates similar to (\ref{4.21})--(\ref{4.22}) below.
Notice the difference in the definitions of $c(x)$ in (\ref{2.9}) and in (\ref{2.23}) which is due to the fact that $c(x)$
is defined here in terms of the process $\eta$ and not $\xi$. The following result is a continuous time version of
Theorem \ref{thm2.1} which can be viewed also as a substantial extension of \cite{DP}.

\begin{theorem}\label{thm2.4} Assume that $E\eta=E\int_0^\tau\xi(t)dt=0$ and (\ref{2.6}), (\ref{2.7}) and (\ref{2.12}) (for $\rho$ defined in (\ref{2.21})) hold true, as well.
Then the stationary vector process
 $\xi(t),\,-\infty<t<\infty$ can be redefined preserving its distributions on a sufficiently rich probability space which contains also a $e$-dimensional Brownian motion $\cW$ with the covariance matrix $\vs$ (at the time 1) so that for each $T>0$ and $\Xi^\ve$ solving (\ref{2.11bis}) with
 $W(t)=W^\ve(t)=\ve\cW(\ve^{-2}t),\, 0\leq t\leq T$,
\begin{equation}\label{2.24}
\sup_{0\leq t\leq T}|X^\ve(t)-\Xi^\ve(t/\bar\tau)|=O(\ve^\delta)\,\,\,\mbox{a.s.}
\end{equation}
where $X^\ve(0)=\Xi^\ve(0)$ and $\del>0$ does not depend on $\ve>0$.
\end{theorem}

We observe that if the stationary process $\xi$ on the probability space $(\hat\Om,\hat\cF,\hat P)$ would be sufficiently
fast mixing then the proof of Theorem \ref{thm2.4} could proceed essentially in the same way as in Theorem \ref{thm2.1} and,
in fact, the former could be derived from the latter by discretizing time. But, in general, this is not the case in applications
to dynamical systems no matter what speed of decay of the coefficients $\phi$ and $\rho$ on the base space $(\Om,\cF,P)$ is
 assumed, and so substantial additional work is required here.
The key step in the proof of Theorem \ref{thm2.4} is to obtain a continuous time version of Theorem \ref{thm2.2}
which is the following result proved in Section \ref{roughsec5}. After that we will rely on rough paths theory arguments of Section \ref{roughsec6}. Set
\[
V^\ve(t)=\ve\int_0^{t\bar\tau\ve^{-2}}\xi(s)ds,\,\,\,\bbV_{ij}^\ve(t)=\ve^2\int_0^{t\bar\tau\ve^{-2}}
\xi_j(s) ds \int_0^s\xi_i(u)du,\,\, i,j=1,...,d
\]
and observe that in view of estimates in Section \ref{subsec5.3}, more precisely \eqref{5.16},
\[
\lim_{\ve\to 0}E\bbV_{ij}^\ve(t)=t\sum_{l=1}^\infty E(\eta_i(0)\eta_j(l))+
 tE \Big( \int_0^{\tau(\om)}\xi_j(s,\om)ds\int_0^s\xi_i(u,\om)du \Big).
\]
\begin{theorem}\label{thm2.5}
Suppose that (\ref{2.12}) and (\ref{2.19}) hold true with $\rho$ defined in (\ref{2.21}).
 Then the stationary vector process
 $\xi(t),\,-\infty<t<\infty$ can be redefined preserving its distributions on a sufficiently rich probability space which contains also a $e$-dimensional Brownian motion $\cW$ with the covariance matrix $\vs$ at the time 1 given by (\ref{2.22}) so that for each $T>0$ and $W^\ve(t)=\ve\cW(\ve^{-2}t),\, 0\leq t\leq T$,
\begin{equation}\label{2.25}
\|V^\ve-W^\ve\|_{p,[0,T]}=O(\ve^\del)\quad\mbox{a.s.}
\end{equation}
and
\begin{equation}\label{2.26}
\|\bbV^\ve_{ij}-\bbW^\ve_{ij}\|_{p/2,[0,T]}=O(\ve^\del)\quad\mbox{a.s.}
\end{equation}
where $p\in(2,3)$, $\del>0$ does not depend on $\ve$, a.s. is taken simultaneously over $\ve\in(0,1)$ and
 \[
 \bbW_{ij}^\ve(t)=\int_0^tW_i^\ve(s)dW^\ve_j(s)+t\sum_{l=1}^\infty E(\eta_i(0)\eta_j(l))+
 tE \Big( \int_0^{\tau(\om)}\xi_j(s,\om)ds\int_0^s\xi_i(u,\om)du \Big).
 \]
\end{theorem}

Again, as explained in Section 7 the strong invariance principle (\ref{2.26}) for second order iterated integrals implies strong invariance principles for iterated integrals of all orders. Taking this into account we can
 use Theorem \ref{thm2.5} not only for the proof of Theorem \ref{thm2.4} but also can apply (\ref{2.26}) and its extension
 to multiple iterated integrals in order to obtain laws of iterated logarithm for iterated
  integrals of any order relying on the laws of iterated logarithm for multiple stochastic integrals (see, for instance,
  \cite{Bal}).
  \begin{corollary}\label{cor2.6}
  Under the conditions of Theorem \ref{thm2.5}, for all $\ell\geq 1$, 
   \begin{equation*}
  P\big\{\limsup_{t\to\infty}\frac {\int_0^t\xi_{i_\ell}(t_\ell)dt_\ell\int_0^{t_\ell}\xi_{i_{\ell-1}}(t_{\ell-1})dt_{\ell-1}\cdots\int_0^{t_2}\xi_{i_1}(t_1)dt_1}
  {(2t\log\log t)^{\ell/2}}=d_{i_1,i_2,...,i_\ell}\big\}=1
  \end{equation*}
  where $1\le i_1\le i_2\leq ...\leq i_\ell\le e$ and the constant $d_{i_1,i_2,...,i_\ell}$
   can be computed from the formula given in Corollary \ref{cor:BaldiGenInt}.

   \end{corollary}

The main application to dynamical systems we have here in mind is a $C^2$ Axiom A flow $F^t$ near an attractor which
using Markov partitions can be represented as a suspension over an exponentially fast $\psi$-mixing transformation so that
we can take $\xi(t)=g\circ F^t$ for a H\" older continuous function $g$ and the probability $P$ being a Gibbs invariant
measure constructed by a H\" older continuous potential on the base of the Markov partition (see, for instance, \cite{BR}).

\begin{remark}\label{rem2.7}
It is not clear how to extend our methods to derive Theorems \ref{thm2.1} and \ref{thm2.4} for the processes $X_N$ and $X^\ve$
given by the general equations (\ref{1.3}) and (\ref{1.1}), respectively. This would require, say in the discrete time case,
 to obtain estimates for strong (almost sure) approximations for the sums $S_N(x,t)=N^{-1/2}\sum_{0\leq k< [Nt]}B(x,\xi(k))$
 and iterated sums $\bbS^{ij}_N(x,t)=N^{-1}\sum_{0\leq k<l<[Nt]}B_i(x,\xi(k))B_j(x,\xi(l))$ in the supremum (even a H\" older)
 in $x$ norm. This, essentially, amounts to a strong invariance principle in a Banach space which was proved for sums in
  \cite{KP} but only with logarithmic estimates which does not seem to allow to extend it to iterated sums and to
  $p$-variation norms. If this were
  possible then we could proceed relying on local Lipschitz continuity of the Banach space version of the  It\^ o-Lyons map
 for rough differential equations (see, for instance, Theorem 3.6 in \cite{CFKMZ}; also \cite{KM2, BC}).
\end{remark}
\begin{remark}\label{rem2.8} Our results can be obtained assuming moment
rather than uniform bounds, namely, in place of $\|\xi\|_\infty<\infty$
we can assume that $\|\xi\|_m<\infty$ for some $m$ big enough. To do this
it is helpful to replace the $\phi$-mixing coefficient by more general dependence coefficients between pairs of $\sig$-algebras
$\cG,\cH\subset\cF$ defined by
\[
\varpi_{q,p}(\cG,\cH)=\sup\{\|E(g|\cG)-Eg\|_p:\, g\,\,\mbox{is}\,\,\cH-\mbox{measurable and}\,\,\|g\|_q\leq 1\}.
\]
The proofs proceed then essentially in the same way estimating moments of
conditional expectations not by Lemma \ref{lem3.1} below, as we do here, but
by Corollary 3.6 from \cite{KV} and relying on some related estimates from
Section 3 there.
\end{remark}
\begin{remark}\label{rem2.9}
Our method in the proof of Theorem \ref{thm2.5} can be slightly modified to extend Theorem \ref{thm2.2} to random vectors built by nonuniformly hyperbolic dynamical systems modelled
by Young towers which are discrete time suspensions assuming boundedness of appropriate
moments of the return time function.
\end{remark}

\section{Strong approximations for sums  }\label{roughsec3}\setcounter{equation}{0}
\subsection{General lemmas}\label{subsec3.1}
First, we will formulate three general results which will be used throughout this
paper. The following lemma is well known (see, for instance, Corollary
to Lemma 2.1 in \cite{Kha} or Lemma 1.3.10 in \cite{HK}).
\begin{lemma}\label{lem3.1}
Let $H(x,\om)$ be a bounded measurable function on the space $(\bbR^d\times\Om,\,\cB\times\cF)$,
where $\cB$ is the Borel $\sig$-algebra, such that for each $x\in\bbR^d$ the function $H(x,\cdot)$
is measurable with respect to a $\sig$-algebra $\cG\subset\cF$. Let $V$ be an $\bbR^d$-valued
random vector measurable with respect to another $\sig$-algebra $\cH\subset\cF$.
Then with probability one,
\begin{equation}\label{3.1}
 |E(H(V,\om)|\cH)-h(V)|\leq 2\phi(\cG,\cH)\| H\|_{\infty}
 \end{equation}
where $h(x)=EH(x,\cdot)$ and the $\phi$-dependence coefficient was defined in (\ref{2.1}).
In particular (which is essentially an equivalent statement), let
 $H(x_1,x_2),\, x_i\in\bbR^{d_i},\, i=1,2$ be a bounded Borel function and $V_i$ be
 $\bbR^{d_i}$-valued $\cG_i$-measurable random vectors, $i=1,2$ where $\cG_1,\cG_2\subset\cF$ are
 sub $\sig$-algebras. Then with probability one,
 \begin{equation*}
 |E(H(V_1,V_2)|\cG_1)-h(V_1)|\leq 2\phi(\cG_1,\cG_2)\| H\|_{\infty}.
 \end{equation*}
 \end{lemma}

 We will employ several times the following general moment estimate which appeared as Lemma 3.2.5 in \cite{HK} for random variables and was extended to random vectors in
 Lemma 3.4 from \cite{Ki20}.
  \begin{lemma}\label{lem3.2}
  Let $(\Om,\cF,P)$ be a probability space with a filtration of $\sig$-algebras $\cG_j,\, j\geq 1$ and
  a sequence of random $d$-dimensional vectors $\eta_j,\, j\geq 1$ such that $\eta_j$
  is $\cG_j$-measurable, $j=1,2,...$. Suppose that for some integer $M\geq 1$,
  \[
  A_{2M}=\sup_{i\geq 1}\sum_{j\geq i}\| E(\eta_j|\cG_i)\|_{2M}<\infty
  \]
  where $\|\eta\|_p=(E|\eta|^p)^{1/p}$ and $|\eta|$ is the Euclidean norm of a (random) vector $\eta$.
  Then for any integer $n\geq 1$,
  \begin{equation}\label{3.2}
  E|\sum_{j=1}^n\eta_j|^{2M}\leq 3(2M)!d^MA_{2M}^{2M}n^M.
  \end{equation}
  \end{lemma}

  In order to obtain uniform moment estimates required by Theorem \ref{thm2.1} we will need the following general estimate which appeared as Lemma 3.7 in \cite{Ki20}.
\begin{lemma}\label{lem3.3} Let $\eta_1,\eta_2,...,\eta_N$ be random $d$-dimensional vectors and
$\cH_1\subset\cH_2\subset...\subset\cH_N$ be a filtration of $\sig$-algebras such that $\eta_m$ is
$\cH_m$-measurable for each $m=1,2,...,N$. Assume also that $E|\eta_m|^{2M}<\infty$ for some $M\geq 1$
and each $m=1,...,n$. Set $S_m=\sum_{j=1}^m\eta_j$. Then
\begin{eqnarray}\label{3.3}
&E\max_{1\leq m\leq n}|S_m|^{2M}\leq 2^{2M-1}\big((\frac {2M}{2M-1})^{2M}E|S_n|^{2M}\\
&+E\max_{1\leq m\leq n-1}|\sum^n_{j=m+1}E(\eta_j|\cH_m)|^{2M}\big)
\leq 2^{2M-1}A_{2M}^{2M}(3(2M)!d^Mn^{M}+n)\nonumber
\end{eqnarray}
where $A_{2M}$ is the same as in Lemma \ref{lem3.2}.
\end{lemma}

The following result will be used for moment estimates of sums and iterated sums of random variables, the latter part seems to be
completely new.
\begin{lemma}\label{lem3.4}
Let $\zeta(k),\,\mu(k),\, k=0,1,...$ be two sequences of random variables on a probability space
$(\Om,\cF,P)$ such that
\begin{eqnarray*}
&E\zeta(k)=E\mu(k)=0\,\,\mbox{for all}\,\, k\geq 0,\,\,\sup_{k\geq 0} E(|\zeta(k)|^{2M}+|\mu(k)|^{2M})<\infty\,\,\,\mbox{and}\\
&|\zeta(k)-E(\zeta(k)|\cF_{k-n,k+n})|,\,|\mu(k)-E(\mu(k)|\cF_{k-n,k+n})|\leq\rho(n)
\end{eqnarray*}
where the probability $P$ is $\phi$-mixing with respect to the family of $\sig$-algebras $\cF_{kl}$
(as described in Section \ref{roughsec2}) with $\phi$ and
$\rho\geq 0$ satisfying $\sum_{n=0}^\infty n(\rho(n)+\phi(n))<\infty$. Set
\[
\bbR(\ell,m,n)=\sum_{k=m}^n\sum_{j=\ell(k)}^{k-1}(\zeta(j)\mu(k)-E(\zeta(j)\mu(k)))
\]
where $0\leq m\leq n$ are integers and $0\leq\ell(k)<k$ is an integer valued function (maybe constant). Then
\begin{equation}\label{3.4}
E\max_{m\leq n<N}|\bbR(\ell,m,n)|^{2M}\leq C_1^{\zeta,\mu}(M)(N-m)^M\max_{m\leq k\leq N}(k-\ell(k))^M
\end{equation}
where $C^{\zeta,\mu}(M)>0$ does not depend on $N,\, m$ or $\ell$. In fact, $C^{\zeta,\mu}(M)$ depends
only on $M,\rho$ and $\phi$ while it does not depend on the sequences $\zeta(k)$ and
$\mu(k)$ themselves. In particular,
\[
E\max_{m\leq n<N}|\sum_{k=m}^n\mu(k)|^{2M}\leq C_1^{1,\mu}(M)(N-m)^M
\]
which is obtained simplifying the proof below just by disregarding the sequence $\zeta(j)$.
If $\xi(k),\, k=0,1,...$ is a sequence of random vectors with the components $\xi_i(k),\, i=1,...,d$ satisfying the conditions
 above, then
 \[
E\max_{m\leq n<N}|\sum_{k=m}^n\xi(k)|^{2M}\leq C_1^\xi(M)(N-m)^M
\]
for $C_1^\xi(M)=d^{2M-1}\sum_{i=1}^dC_1^{1,\xi_i}$ where $|\cdot|$ denotes the Euclidean norm in $\bbR^d$.
 \end{lemma}
 \begin{proof}
 Set
 \[
 \zeta_r(k)=E(\zeta(k)|\cF_{k-r,k+r})\,\,\mbox{and}\,\, \mu_r(k)=E(\mu(k)|\cF_{k-r,k+r}).
 \]
 Then
 \[
 \zeta(k)=\lim_{i\to\infty}\zeta_{2^i}(k)=\zeta_1(k)+\sum_{r=1}^\infty(\zeta_{2^r}(k)-\zeta_{2^{r-1}}(k))
 \]
 and
  \[
 \mu(k)=\lim_{i\to\infty}\mu_{2^i}(k)=\mu_1(k)+\sum_{r=1}^\infty(\mu_{2^r}(k)-\mu_{2^{r-1}}(k))
 \]
 where convergence is in the $L^\infty$ sense since
 \[
 \|\zeta_{2^r}(k)-\zeta_{2^{r-1}}(k)\|_\infty,\, \|\mu_{2^r}(k)-\mu_{2^{r-1}}(k)\|_\infty\leq 2(\rho(2^r)+\rho(2^{r-1})).
 \]
 For $q,r=0,1,...$ denote
 \begin{eqnarray*}
& \vrho_{q,r}(j,k)=(\zeta_{2^q}(j)-\zeta_{2^{q-1}}(j))(\mu_{2^r}(k)-\mu_{2^{r-1}}(k))\\
&-E\big((\zeta_{2^q}(j)-\zeta_{2^{q-1}}(j))(\mu_{2^r}(k)-\mu_{2^{r-1}}(k))\big)
 \end{eqnarray*}
 and $Q_{q,r}(k)=\sum_{j=\ell(k)}^{k-1}\vrho_{q,r}(j,k)$ where we set for convenience $\rho(2^{-1})=\sup_{k\geq 0}
 (\|\zeta(k)\|_{2M}+\|\mu(k)\|_{2M})$ and
 $\zeta_{2^{-1}}(j)=\mu_{2^{-1}}(k)=0$ for all $j,k\geq 0$. Then
 \[
 \bbR(\ell,m,n)=\sum_{q,r=0}^\infty\sum_{k=m}^nQ_{q,r}(k).
 \]
 Next, introduce $\cG_m=\cG_m^{q,r}=\cF_{-\infty,m+\max(2^q,2^r)}$ and observe that $Q_{q,r}(k)$ is $\cG_k$-measurable.
 We will apply Lemmas \ref{lem3.2} and \ref{lem3.3} to the sums $\bbR(\ell,m,n)$.

 First, write
 \[
 Q_{q,r}(k)=Q_{q,r,n}^{(1)}(k)+Q_{q,r,n}^{(2)}(k)+Q_{q,r,n}^{(3)}(k)
 \]
 where
 \[
 Q_{q,r,n}^{(1)}(k)=\sum_{\ell(k)\leq j<\frac {k+n}2-2\max(2^q,2^r),\, j<k}\vrho_{q,r}(j,k),
 \]
  \[
 Q_{q,r,n}^{(2)}(k)=\sum_{\frac {k+n}2-2\max(2^q,2^r)\leq j<\frac {k+n}2+2\max(2^q,2^r),\, j<k}\vrho_{q,r}(j,k),
 \]
 \[
 \mbox{and}\,\,\,  Q_{q,r,n}^{(3)}(k)=\sum_{\frac {k+n}2+2\max(2^q,2^r)\leq j<k}\vrho_{q,r}(j,k).
 \]
 If $k-n\geq 4\max(2^q,2^r)$ then $\frac {k+n}2-\max(2^q,2^r)-n-\max(2^q,2^r)\geq 0$ and we can write
 \begin{eqnarray*}
 &\| E(Q^{(1)}_{q,r}(k)|\cG_n)\|_{2M}\leq 2\| E\big(E(\mu_{2^r}(k)-\mu_{2^{r-1}}(k)|\cF_{-\infty,\frac {k+n}2-\max(2^q,2^r)})\\
&\times \sum_{\ell(k)\leq j<\frac {k+n}2-2\max(2^q,2^r),\, j<k}(\zeta_{2^q}(j)-\zeta_{2^{q-1}}(j))|\cG_n\big)\|_{2M}\\
&\leq 2\| E(\mu_{2^r}(k)-\mu_{2^{r-1}}(k)|\cF_{-\infty,\frac {k+n}2-\max(2^q,2^r)})\\
&\times \sum_{\ell(k)\leq j<\frac {k+n}2-2\max(2^q,2^r),\, j<k}(\zeta_{2^q}(j)-\zeta_{2^{q-1}}(j))\|_{2M}\\
 &\leq 8\phi(\frac {k-n}2)(\rho(2^r)+\rho(2^{r-1}))\|\sum_{\ell(k)\leq j<\frac {k+n}2-2\max(2^q,2^r),\, j<k}
 (\zeta_{2^q}(j)-\zeta_{2^{q-1}}(j))\|_{2M}
 \end{eqnarray*}
 where we use (\ref{2.1}), (\ref{2.2}) and that $k-2^r-\frac {k+n}2+\max(2^q,2^r)\geq \frac {k-n}2$. If $0\leq k-n<4\max(2^q,2^r)$
 then we estimate
  \begin{eqnarray*}
 &\| E(Q^{(1)}_{q,r}(k)|\cG_n)\|_{2M}\\
 &\leq 4(\rho(2^r)+\rho(2^{r-1}))\|\sum_{\ell(k)\leq j<\frac {k+n}2-2\max(2^q,2^r),\, j<k}(\zeta_{2^q}(j)-\zeta_{2^{q-1}}(j))\|_{2M}.
 \end{eqnarray*}

 In order to bound the $2M$-moment norm of the last sum we will use Lemma \ref{lem3.2} setting $\cH_i=\cF_{-\infty,i+2^q}$
 and relying on (\ref{2.1})-(\ref{2.2}) we estimate for $j\geq i+2^{q+1}$,
 \[
 \| E(\zeta_{2^q}(j)-\zeta_{2^{q-1}}(j)|\cH_i)\|_{2M}\leq 4(\rho(2^q)+\rho(2^{q-1}))\phi(j-i-2^{q+1}).
 \]
 For $i\leq j< i+2^{q+1}$ we use just the obvious estimate
  \[
 \| E(\zeta_{2^q}(j)-\zeta_{2^{q-1}}(j)|\cH_i)\|_{2M}\leq 2(\rho(2^q)+\rho(2^{q-1})).
 \]
 Hence,
 \[
 A^\zeta_{2M}=\sup_{ i\geq 0}\sum_{j\geq i}\| E(\zeta_{2^q}(j)-\zeta_{2^{q-1}}(j)|\cH_i)\|_{2M}
 \leq 2(\rho(2^q)+\rho(2^{q-1}))(2^{q+1}+2\sum_{i=0}^\infty\phi(i)),
 \]
 and so by Lemma \ref{lem3.2},
 \begin{eqnarray*}
 &\|\sum_{\ell(k)\leq j<\frac {k+n}2-2\max(2^q,2^r),\, j<k}(\zeta_{2^q}(j)-\zeta_{2^{q-1}}(j))\|_{2M}\\
 &\leq (3(2M)!)^{1/2M}\sqrt dA^\zeta_{2M}\max_{m\leq k\leq n}\sqrt {k-\ell(k)}.
 \end{eqnarray*}

 Next, by (\ref{2.1}) and (\ref{2.2}),
 \begin{eqnarray*}
 &|E(Q^{(2)}_{q,r,n}(k)|\cG_n)|\leq \sum_{\frac {k+n}2-2\max(2^q,2^r)\leq j<\frac {k+n}2+2\max(2^q,2^r),\, j<k}|E(\vrho_{q,r}(j,k)|\cG_n)|\\
 &\leq 64\max(2^q,2^r)\phi(\frac {k-n}2-4\max(2^q,2^r))(\rho(2^r)+\rho(2^{r-1}))(\rho(2^q)+\rho(2^{q-1}))
 \end{eqnarray*}
 since each $\vrho_{q,r}(j,k)$ here is $\cF_{\frac {k+n}2-3\max(2^q,2^r),\infty}$-measurable, $\frac {k+n}2-3\max(2^q,2^r)-n-\max(2^q,2^r)=\frac {k-n}2-4\max(2^q,2^r)$
  and we take here $\phi(x)=1$ for any $x\leq 0$.
 Using again (\ref{2.1}) and (\ref{2.2}) we have also
 \begin{eqnarray*}
 & |E(Q^{(3)}_{q,r,n}(k)|\cG_n)|\leq\sum_{\frac {k+n}2+2\max(2^q,2^r)\leq j<k}|E(\vrho_{q,r}(j,k)|\cG_n)|\\
 &\leq 8(\rho(2^r)+\rho(2^{r-1}))(\rho(2^q)+\rho(2^{q-1}))(k-n)\phi(\frac {k-n}2)
 \end{eqnarray*}
  since each $\vrho_{q,r}(j,k)$ here is $\cF_{\frac {k+n}2+\max(2^q,2^r),\infty}$-measurable and $\frac {k+n}2+\max(2^q,2^r)-n-\max(2^q,2^r)=\frac {k-n}2$.

 Hence, by Lemmas \ref{lem3.2} and \ref{lem3.3},
 \[
 \|\max_{m\leq n\leq N}|\sum_{k=m}^nQ_{q,r}(k)|\|_{2M}\leq 2A^{q,r}_{2M}\big(3(2M)!d^M(N-m)^M+N-m\big)^{1/2M}
 \]
 where by the above
 \begin{eqnarray*}
 &A^{q,r}_{2M}=\sup_{0\leq n\leq N}\sum_{k\geq n}\| E(Q_{q,r}(k)|\cG_n)\|_{2M}\\
 &\leq 64(\rho(2^r)+\rho(2^{r-1}))(\rho(2^q)+\rho(2^{q-1}))\\
 &\times\big( (3(2M)!)^{1/2M}\sqrt d(\sum^\infty_{i=0}\phi(i)+\max(2^q,2^r))^2\max_{m\leq j\leq N}\sqrt {j-\ell(j)}\\
 &+\max(2^q,2^r)(4\max(2^q,2^r)+\sum^\infty_{i=0}\phi(i))+\sum^\infty_{i=1}i\phi(i)\big).
 \end{eqnarray*}
 This together with
 \[
 \|\max_{m\leq n\leq N}|\bbR(\ell,m,n)|\|_{2M}\leq\sum_{q,r=0}^\infty\|\max_{m\leq n\leq N}|\sum_{k=m}^nQ_{q,r}(k)|\|_{2M}
 \]
 yields (\ref{3.4}) completing the proof of the lemma.
 \end{proof}

 \subsection{Characteristic functions estimates}\label{subsec3.2}
 Next, we will follow the same path as in \cite{BP} which leads to strong approximation (almost sure invariance
 principle) theorem. For each $n\geq 1$ introduce the characteristic function
 \[
 f_n(w)=E\exp(i\langle w,\, n^{-1/2}\sum_{k=0}^{n-1}\xi(k)\rangle),\, w\in\bbR^e
 \]
 where $\langle\cdot,\cdot\rangle$ denotes the inner product. In the same way as in
  Lemma 3.10 from \cite{Ki21} we obtain
 \begin{lemma}\label{lem3.5}
 For any $n\geq 1$,
 \begin{equation}\label{3.5}
 |f_n(w)-\exp(-\frac 12\langle \vs w,\, w\rangle)|\leq C_2e^3n^{-\wp}
 \end{equation}
 for all $w\in\bbR^e$ with $|w|\leq n^{\wp/2}$ where we can take $\wp\leq\frac 1{20}$
 and a constant $C_2>0$ independent of $n$ and $e$.
 \end{lemma}

  Next, we follow \cite{BP} and \cite{KP} introducing blocks of high polynomial power length with gaps between
 them. Set $m_0=0$ and recursively $n_k=m_{k-1}+[k^\be],\, m_k=n_k+[k^{\be/4}],\, k=1,2,...$ where $\be>0$ is
 big and will be chosen later on. Now we define sums
 \begin{eqnarray*}
 &Q_k=\sum_{m_{k-1}\leq j<n_k}E\big(\xi(j)|\cF_{j-\frac 13[(k-1)^{\be/4}],j+\frac 13[(k-1)^{\be/4}]}\big)\\
 &\mbox{and}\quad R_k=\sum_{n_k\leq j<m_k}\xi(j),\, k=1,2,...
 \end{eqnarray*}
where the first sums play the role of blocks while the second ones are gaps whose total contribution turns out to
be negligible for our purposes. Set also $\ell_N(t)=\max\{ k:\, m_k\leq Nt\}$ and $\ell_N=\ell_N(T)$.

\begin{lemma}\label{lem3.6} With probability one for all $N\geq 1$,
\begin{equation}\label{3.6}
\sup_{0\leq t\leq T}|S_N(t)-N^{-1/2}\sum_{1\leq k\leq\ell_N(t)}Q_k|=O(N^{-\del})
\end{equation}
provided $\del>0$ is small and $\be>0$ is large enough.
\end{lemma}
\begin{proof}
Denote the left hand side of (\ref{3.6}) by $I$, then
\[
I\leq\sup_{0\leq t\leq T}I_1(t)+\sup_{0\leq t\leq T}I_2(t)+\sup_{0\leq t\leq T}I_3(t).
\]
Here,
\begin{eqnarray}\label{3.7}
&I_1(t)=N^{-1/2}\big\vert\sum_{1\leq k\leq\ell_N(t)}\big(\sum_{m_{k-1}\leq j<n_k}(\xi(j)\\
&-E(\xi(j)|\cF_{j-\frac 13[(k-1)^{\be/4}],j+\frac 13[(k-1)^{\be/4}]})\big)\big\vert\nonumber\\
&\leq N^{-1/2}\sum_{1\leq k\leq\ell_N(t)}k^\be\rho(\frac 13[(k-1)^{\be/4}])\leq C_3N^{-(\frac 12-\frac 1{\be+1})}\nonumber
\end{eqnarray}
since
\begin{equation}\label{3.8}
(\frac {TN}2)^{\frac 1{\be+1}}-2\leq\ell_N\leq (TN)^{\frac 1{\be+1}}
\end{equation}
and by (\ref{2.12}),
\[
C_3=\max_{k\geq 1}(k^\be\rho(\frac 13[(k-1)^{\be/4}]))<\infty.
\]
It remains to estimate
\begin{eqnarray*}
&I_2(t)=N^{-1/2}\big\vert\sum_{1\leq k\leq\ell_N(t)}\sum_{n_{k}\leq j<m_k}\xi(j)\big\vert\\
&\mbox{and}\,\,\, I_3(t)=N^{-1/2}\big\vert\sum_{m_{\ell_N(t)}\leq j<[Nt]}\xi(j)\big\vert .
\end{eqnarray*}

By (\ref{3.8}) and Lemma \ref{lem3.4} with $C_1(M)=C_1^\xi(M)$,
\begin{eqnarray*}
&E\sup_{0\leq t\leq T}I_2^{2M}(t)\leq N^{-M}\sum_{0\leq l\leq\ell_N}E\big\vert\sum_{1\leq k\leq l}\sum_{n_{k}\leq j<m_k}\xi(j)\big\vert^{2M}\\
&\leq C_1(M)N^{-M}\sum_{0\leq l\leq\ell_N}(\sum_{1\leq k\leq l}k^{\be/4})^M\leq C_4(M)N^{-\frac 14(3-\frac 7{\be+1})M+1}
\end{eqnarray*}
where $C_4(M)=C_1(M)3^{(\frac \be 4+2)M+1}$. By the Chebyshev inequality
\begin{equation}\label{3.9}
P\{\sup_{0\leq t\leq T}I_2(t)>N^{-\del}\}\leq N^{2M\del}E\sup_{0\leq t\leq T}I_2^{2M}(t).
\end{equation}
Choosing $\be>0$ big and $\del>0$ small enough so that $\frac 7{\be+1}+8\del\leq 1$ and taking $M\geq 6$ we obtain that
the right hand side of (\ref{3.9}) is bounded by $N^{-2}$, and so by the Borel-Cantelli lemma
\begin{equation}\label{3.10}
\sup_{0\leq t\leq T}I_2(t)=O(N^{-\del})\quad\mbox{a.s.}
\end{equation}

Next,
\begin{eqnarray*}
&\sup_{0\leq t\leq T}I_3^{2M}(t)=N^{-M}\max_{1\leq k\leq\ell_N}\max_{m_k\leq j<m_{k+1}\wedge Nt}|\sum_{m_k\leq i
\leq j}\xi(i)|^{2M}\\
&\leq N^{-M}\sum_{1\leq k\leq\ell_N}\sum_{m_k\leq j<m_{k+1}}|\sum_{m_k\leq i\leq j}\xi(i)|^{2M}.\\
\end{eqnarray*}
By Lemma \ref{lem3.4} with $C_1(M)=C_1^\xi(M)$,
\[
E|\sum_{m_k\leq i\leq j}\xi(i)|^{2M}\leq C_1(M)(j-m_k+1)^M,
\]
and so by (\ref{3.8}),
\[
E\sup_{0\leq t\leq T}I_3^{2M}(t)\leq C_1(M)N^{-M}\sum_{k=1}^{\ell_N}(m_{k+1}-m_k+1)^{M+1}\leq C_5(M)N^{-\frac M{\be+1}+1}
\]
where $C_5(M)=C_1(M)3^{M+1}$. By the Chebyshev inequality
\begin{equation}\label{3.11}
P\{\sup_{0\leq t\leq T}I_3(t)>N^{-\del}\}\leq N^{2M\del}E\sup_{0\leq t\leq T}I_3^{2M}(t).
\end{equation}
Choosing $\del\leq\frac 1{4(\be+1)}$ and $M\geq 12(\be+1)$ we bound the right hand side of (\ref{3.11}) by $N^{-2}$ which
together with the Borel-Cantelli lemma yields that
\begin{equation}\label{3.12}
\sup_{0\leq t\leq T}I_3(t)=O(N^{-\del})\quad\mbox{a.s.}
\end{equation}
Finally, (\ref{3.6}) follows from (\ref{3.7}), (\ref{3.10}) and (\ref{3.12}).
\end{proof}

Next, set
\begin{equation}\label{3.13}
\cG_k=\cF_{-\infty,n_k+\frac 13[k^{\be/4}]},
\end{equation}
and so $Q_k$ is $\cG_k$-measurable. The following result is a corollary of Lemmas \ref{lem3.1} and \ref{lem3.5}.
\begin{lemma}\label{lem3.7} For any $k\geq 1$,
\begin{eqnarray}\label{3.14}
&|E(\exp(i\langle w,\,(n_k-m_{k-1})^{-1/2}Q_k\rangle |\cG_{k-1})-\exp(-\frac 12\langle\vs w,w\rangle)|\\
&\leq 2\phi(\frac 13[(k-1)^{\be/4}])+\rho(\frac 13[(k-1)^{\be/4}])+C_2d^3[k^\be]^{-\wp}\nonumber
 \end{eqnarray}
 for all $w\in\bbR^e$ with $|w|\leq(n_k-m_{k-1})^{\wp/2}$.
 \end{lemma}
 \begin{proof} Set
 \[
 F_k=\exp(i\langle w,\,(n_k-m_{k-1})^{-1/2}Q_k\rangle).
 \]
 Then by Lemma \ref{lem3.1},
 \[
 |E(F_k|\cG_{k-1})-EF_k|\leq 2\phi(\frac 13[(k-1)^{\be/4}]).
 \]
 Since $|e^{i(a+b)}-e^{ib}|\leq |a|$, we obtain from (\ref{2.3})
 taking into account the stationarity of $\xi(k)$'s that,
 \[
 |EF_k-f_{n_k-m_{k-1}}(w)|\leq |w|k^{\be/2}\rho(\frac 13[(k-1)^{\be/4}]),
 \]
 and (\ref{3.14}) follows from (\ref{3.5}).
 \end{proof}

\subsection{Strong approximations}\label{subsec3.3}

We will rely on the following result which is a version of Theorem 1 in \cite{BP} with some features
taken from Theorem 4.6 in \cite{DP} (see also Theorem 3 in \cite{MP}).
\begin{theorem}\label{thm3.8} Let $\{ V_k,\, k\geq 1\}$ be a sequence of random vectors with values in $\bbR^e$
defined on some probability space $(\Om,\cF,P)$ and such that $V_k$ is measurable with respect to $\cG_k,\, k=1,2,...$
where $\cG_k,\, k\geq 1$ is a filtration of countably generated sub-$\sig$-algebras of $\cF$. Assume that the probability
space is rich enough so that there exists on it a sequence of uniformly distributed on $[0,1]$ independent random variables
$U_k,\, k\geq 1$ independent of $\vee_{k\geq 1}\cG_k$. Let $G$ be a probability distribution on $\bbR^e$ with the characteristic
function $g$. Suppose that for some nonnegative numbers $\nu_m,\del_m$ and $K_m\geq 10^8d$,
 \begin{equation}\label{3.15}
 E\big\vert E(\exp(i\langle w,V_k\rangle)|\cG_{k-1})-g(w)\big\vert \leq\nu_k
 \end{equation}
 for all $w\in\bbR^e$ with $|w|\leq K_k$ and
 \begin{equation}\label{3.16}
 G\{ x:\, |x|\geq\frac 14K_k\}<\del_k.
 \end{equation}
 Then there exists a sequence $\{ W_k,\, k\geq 1\}$ of $\bbR^e$-valued random vectors defined on
 $(\Om,\cF,P)$ with the properties

 (i) $W_k$ is $\cG_k\vee\sig\{U_k\}$-measurable for each $k\geq 1$;

 (ii) each $W_k,\, k\geq 1$ has the distribution $G$ and $W_k$ is independent of $\sig\{ U_1,...,U_{k-1}\}\vee
 \cG_{k-1}$, and so also of $W_1,...,W_{k-1}$;

 (iii) Let $\vr_k=16K^{-1}_k\log K_k+2\nu_k^{1/2}K_k^d+2\del_k^{1/2}$. Then
 \begin{equation}\label{3.17}
 P\{ |V_k-W_k|\geq\vr_k\}\leq\vr_k.
 \end{equation}
 \end{theorem}

In order to apply Theorem \ref{thm3.8} we take $V_k=(n_k-m_{k-1})^{-1/2}Q_k,\, \cG_k$ given by (\ref{3.13}) and
\[
g(w)=\exp(-\frac 12\langle\vs w,w\rangle)
\]
so that $G$ is the mean zero $d$-dimensional Gaussian distribution with the covariance matrix $\vs$. Relying on Lemma \ref{lem3.7} we take $\wp=\frac 1{20}$,
\begin{equation*}
K_k=(n_k-m_{k-1})^{\wp/4d}\leq (n_k-m_{k-1})^{\wp/2}\,\,\mbox{and}\,\,
\nu_k=C_6k^{-\be\wp}
\end{equation*}
for some $C_6>0$ independent of $k\geq 1$. By the Chebyshev inequality we have also
\begin{eqnarray*}
&G\{ x:\, |x|\geq \frac {K_k}4\}=P\{|\Psi|\geq\frac 14(n_k-m_{k-1})^{\wp/4d}\}\\
&\leq 4d\|\vs\|(n_k-m_{k-1})^{-\wp/2d}\leq C_7k^{-\wp\be/2d}
\end{eqnarray*}
for some $C_7>0$ which does not depend on $k$.

Now Theorem \ref{thm3.8} provides us with random vectors $W_k,\, k\geq 1$ satisfying the properties (i)--(iii), in
particular, the random vector $W_k$ has the mean zero Gaussian distribution with the covariance matrix $\vs$, it is
independent of $W_1,...,W_{k-1}$ and the property (iii) holds true with
\[
\vr_k=4\frac \wp d(n_k-m_{k-1})^{-\wp/4d}\log(n_k-m_{k-1})+2C_6^{1/2}k^{- \be\wp/4}+2C_7^{1/2}k^{-\be\wp/4d}.
\]
Next, we choose $\be>160d$ which gives
\begin{equation}\label{3.18}
\vr_k\leq C_8k^{-2}
\end{equation}
for all $k\geq 1$ where $C_8>0$ does not depend on $k$.

Next, let $W(t),\, t\geq 0$ be a $e$-dimensional Brownian motion with the covariance matrix $\vs$ at the time 1. Then the sequence
 of random vectors $\tilde W_k=(n_k-m_{k-1})^{-1/2}(W(n_k)-W(m_{k-1})),\, k=1,2,...$ and $W_k,\, k\geq 1$ have the same
 distributions. Moreover, we can redefine the process $\xi(n),\,-\infty<n<\infty$ and choose a Brownian motion $W(s),\,
 s\geq 0$ (with the covariance matrix $\vs$ at the time 1) preserving their distributions so that the joint distribution of the sequences
  of pairs $(V_k,W_k)$ and of $(\tilde V_k,\tilde W_k)$, where $\tilde V_k$ is constructed as $V_k$ but using the redefined
   process $\xi(n)$, will be the same. Indeed, by the Kolmogorov extension theorem (see, for instance, \cite{Bil}) such pair
   of processes exists if we impose consistent restrictions on their joint finite dimensional distributions. But since the pair
   of processes $\xi(n),\, n\geq 0$ and $W_k,\, k\geq 1$ satisfying the conditions of Theorem \ref{thm3.8} exist as asserted
   there, these conditions are consistent and the required pair of processes indeed exists.
   A more precise justification of this relies on Lemma A1 from \cite{BP}. Namely, let $R$
   be the joint distribution of the process $\xi(n),\,-\infty<n<\infty$ and of the sequence
   $W_k,\, k\geq 1$ and let $\tilde R$ be the joint distribution of the sequence
   $\tilde W_k,\, k\geq 1$ and a $e$-dimensional Brownian motion $W(t),\, t\geq 0$
   with the covariance matrix $\vs$ at the time 1. Since the second marginal of $R$ coincides with the
   first marginal of $\tilde R$, it follows by Lemma A1 of \cite{BP} that the process $\xi$
   and the sequence $W_k,\, k\geq 1$ can be redefined on a richer probability space where there exists a Brownian motion $W(t),\, t\geq 0$ with the covariance matrix $\vs$ (at the time 1) such
   that $W_k=(n_k-m_{k-1})^{-1/2}(W(n_k)-W(m_{k-1}))$, and so from now on we will rely on
   this equality and assume that these $W_k$'s satisfy (\ref{3.17}) with $\vr_k$ satisfying
   (\ref{3.18}). It follows by the Borel-Cantelli lemma that there exists a random variable $D=D(\om)<\infty$ a.s. such that
   \begin{equation}\label{3.19}
   |V_k-W_k|\leq Dk^{-2}\quad\mbox{a.s.}
   \end{equation}

   Now we can obtain the following result.
   \begin{lemma}\label{lem3.9} With probability one,
  \begin{equation}\label{3.20}
  \sup_{0\leq t\leq T}|\sum_{1\leq k\leq\ell_N(t)}Q_k-W(tN)|=O(N^{\frac 12-\del})
  \end{equation}
  for some $\del>0$ which does not depend on $N$.
  \end{lemma}
  \begin{proof} We have
  \[
  J_N(t)=|\sum_{1\leq k\leq \ell_N(t)}Q_k-W(tN)|\leq  J_N^{(1)}(t)+| J_N^{(2)}(t)|
  \]
  where by (\ref{3.8}) and (\ref{3.19}),
   \begin{equation}\label{3.21}
  J_N^{(1)}(t)=\sum_{1\leq k\leq \ell_N(t)}(n_k-m_{k-1})^{1/2}|V_k-W_k|\leq D\sum_{1\leq k\leq \ell_N}
  [k^\be]^{1/2}k^{-2}\leq D2^{\be/2}N^{\frac 12(1-\frac 3{\be+1})}
  \end{equation}
 and
 \[
  J_N^{(2)}(t)=W(tN)-W(m_{\ell_N(t)})+\sum_{1\leq k\leq\ell_N(t)}(W(m_k)-W(n_k)).
  \]
 Clearly, $ J_N^{(2)}(t),\, t\geq 0$ is a martingale in $t$, and so
 \begin{eqnarray}\label{3.22}
& E\sup_{0\leq t\leq T}|J^{(2)}_N(t)|^{2M}\leq (\frac {2M}{2M-1})^{2M}E|J^{(2)}_N(T)|^{2M}\\
&\leq (\frac {4M}{2M-1})^{2M}(E|W(TN)-W(m_{\ell_N})|^{2M}+E|J^{(3)}_N(T)|^{2M})\nonumber
\end{eqnarray}
where $J_N^{(3)}(t)=\sum_{1\leq k\leq\ell_N}(W(m_k)-W(n_k))$. Now, by (\ref{3.8}),
\begin{eqnarray}\label{3.23}
&E|W(TN)-W(m_{\ell_N})|^{2M}\leq\|\vs^{1/2}\|^{2M}(\prod_{j=1}^M(2j-1))(2\ell_N^\be)^M\\
&\leq\|\vs^{1/2}\|^{2M}\sqrt {2M!}2^MN^{(1-\frac 1{\be+1})M}.\nonumber
\end{eqnarray}

Next, $J^{(3)}_N(T)$ is a sum of independent random vectors and we will estimate last term in the right hand side
of (\ref{3.22}) relying on Lemma \ref{lem3.2} with $\eta_j=W(m_j)-W(n_j)$ and $\cG_j=\sig\{ W(t),\, t\leq m_j\}$ for
$j=1,...,\ell_N$. Then
\[
A_{2M}=\sup_{i\geq 1}\sum_{j\geq i}\| E(\eta_j|\cG_i)\|_{2M}=\sup_{i\geq 1}\|\eta_i\|_{2M}\leq\sqrt {2M}(\ell_N+1)^{\be/8},
\]
and so by Lemma \ref{lem3.2},
\begin{eqnarray}\label{3.24}
&E|J^{(3)}_N(T)|^{2M}\leq 3(2M)!(2M)^{M}(\ell_N+1)^{M(1+\frac \be 4)}\\
&\leq 3(2M)!(2M)^{M}2^{M(1+\frac \be 4)}N^{\frac M4(1+\frac 3{\be+1})}.\nonumber
\end{eqnarray}
By (\ref{3.22})--(\ref{3.24}) and the Chebyshev inequality
\begin{eqnarray}\label{3.25}
&P\{\sup_{0\leq t\leq T}|J^{(2)}_N(t)|\geq N^{\frac 12-\del}\}\leq N^{-M(1-2\del)}E\sup_{0\leq t\leq T}|J^{(2)}_N(t)|^{2M}\\
&\leq C_9(M)(N^{-M(\frac 1{\be+1}-2\del)}+N^{-M(\frac 34-\frac 1{\be+1}-2\del)})\nonumber
\end{eqnarray}
where $C_9(M)>0$ does not depend on $N$. Now choose $\be\geq 2,\,\del<\frac 1{4(\be+1)}$ and $M\geq 2(\be+1)$ then the right
hand side of (\ref{3.25}) forms a converging sequence, and so by the Borel-Cantelli lemma
\[
\sup_{0\leq t\leq T}|J^{(2)}_N(t)|=O(N^{\frac 12-\del})\quad\mbox{a.s.}
\]
which together with (\ref{3.21}) completes the proof of Lemma \ref{lem3.9}.
  \end{proof}

Now combining Lemmas \ref{lem3.6} and \ref{lem3.9} we obtain that for some $\del>0$ and all $N\geq 1$,
\begin{equation}\label{3.26}
\sup_{0\leq t\leq T}|S_N(t)-W_N(t)|=O(N^{-\del})\quad\mbox{a.s.},
\end{equation}
where $W_N(t)=N^{-1/2}W(tN)$ is another Brownian motion.

\subsection{$p$-variation norm estimates}\label{subsec3.4}

Thus, (\ref{2.16}) is proved in the supremum norm and we will extend it next for the $p$-variation norm.
First, for any $\al\in(0,1)$ and $\be\in(0,1/2)$ (which has nothing to do with $\be$ in the previous subsection)
we estimate by Lemma \ref{lem3.4} and the Chebyshev inequality
\begin{eqnarray}\label{3.27}
&P\big\{\max_{[TN]\wedge(k+N^\al)\geq l>k\geq 0}\frac {|S_N(l/N)-S_N(k/N)|}{|(l-k)/N|^{\frac 12-\be}}>1\big\}\\
&\leq\sum_{[TN]\wedge(k+N^\al)\geq l>k\geq 0}P\big\{|\sum_{k\leq j<l}\xi(j)|>N^\be(l-k)^{\frac 12-\be}\big\}\nonumber\\
&\leq N^{-2\be M}\sum_{[TN]\wedge(k+N^\al)\geq l>k\geq 0}(l-k)^{-M+2\be M}E|\sum_{k\leq j<l}\xi(j)|^{2M}\nonumber\\
&\leq C_1^\xi(M)TN^{1-2\be M}\sum_{N^\al\geq j\geq 1}j^{2\be M}\le 2^{2\be M+1}C_1^\xi(M)TN^{-2\be M(1-\al)+1+\al}.\nonumber
\end{eqnarray}
Next, we choose $M\geq\frac {3+\al}{2\be(1-\al)}$ which makes the right hand side of (\ref{3.27}) a term in a converging
series. Hence, by the Borel--Cantelli lemma for any $\al\in(0,1)$ and $\be\in(0,1/2)$ there exists a finite a.s. random
variable $C^S_{\al,\be}=C^S_{\al,\be}(\om)$ such that for all $N\geq 1$,
\begin{equation}\label{3.28}
|S_N(l/N)-S_N(k/N)|\leq C^S_{\al,\be}\big\vert\frac lN-\frac kN\big\vert^{\frac 12-\be}\,\,\mbox{if}\,\, k+N^\al\geq l>k
\geq 0,\, l\leq[TN].
\end{equation}

Next, define
\[
\hat W_N(t)=W_N(\frac {[Nt]}N)=\sum_{0\leq k<[Nt]}(W_N(\frac {k+1}N)-W_N(\frac kN)).
\]
Let $0=t_0<t_1<...<t_m=T$ and observe that if $[t_iN]=[t_{i+1}N]$ then
\[
S_N(t_{i+1})=S_N(t_i)\quad\mbox{and}\quad\hat W_N(t_{i+1})=\hat W_N(t_i).
\]
Hence,
\begin{eqnarray}\label{3.29}
&\sum_{0\leq i<m}|S_N(t_{i+1})-\hat W_N(t_{i+1})-S_N(t_{i})+\hat W_N(t_i)|^p\\
&=\sum_{0\leq j<n}|S_N(\frac {k_{j+1}}N)-W_N(\frac {k_{j+1}}N)-S_N(\frac {k_{j}}N)+W_N(\frac {k_{j}}N)|^p\nonumber\\
&\leq J^{(1)}_N+2^{p-1}(J_N^{(2)}+J_N^{(3)})\nonumber
\end{eqnarray}
where
\[
J_N^{(1)}=\sum_{0\leq j<n,\, k_{j+1}-k_j>N^\al}|S_N(\frac {k_{j+1}}N)-W_N(\frac {k_{j+1}}N)-S_N(\frac {k_{j}}N)+W_N
(\frac {k_{j}}N)|^p,
\]
\[
J_N^{(2)}=\sum_{0\leq j<n,\, k_{j+1}-k_j\leq N^\al}|S_N(\frac {k_{j+1}}N)-S_N(\frac {k_{j}}N)|^p,
\]
\[
J_N^{(3)}=\sum_{0\leq j<n,\, k_{j+1}-k_j\leq N^\al}|W_N(\frac {k_{j+1}}N)-W_N(\frac {k_{j}}N)|^p,
\]
$k_j=[t_{i_j}N]$ and $0=t_{i_0}<t_{i_1}<...<t_{i_n}=T$ is the maximal subsequence of $t_0,t_1,...,t_m$ such that
$[t_{i_j}N]<[t_{i_{j+1}}N],\, j=0,1,...,n-1$.

In order to estimate $J^{(1)}_N$ we use (\ref{3.26}) and observe that there exist no more than $[TN^{1-\al}]$ intervals
$[k_{j},k_{j+1}]$ such that $k_{j+1}-k_j>N^\al$ which gives that
\begin{equation}\label{3.30}
J^{(1)}_N\leq C^{(1)}N^{1-\al-p\del}
\end{equation}
for some a.s. finite random variable $C^{(1)}=C^{(1)}(\om)>0$ which does not depend on $N$, $n$ or on the choice of $t_1,...,t_m$
and we choose $\al$ so close to 1 that $1-\al-p\del<0$. In order to estimate $J^{(2)}_N$ we
use (\ref{3.28}) and observe that $\sum_{0\leq j<n}|k_{j+1}-k_j|\leq TN$ which gives
\begin{eqnarray}\label{3.31}
&J^{(2)}_N\leq (C^S_{\al,\be})^pN^{-p(\frac 12-\be)}\sum_{0\leq j<n,\, k_{j+1}-k_j\leq N^\al}|k_{j+1}-k_j|^{p(\frac 12-\be)}\\
&\leq (C^S_{\al,\be})^pTN^{-(1-\al)(\frac p2-1-p\be)}\nonumber
\end{eqnarray}
where we assume that
\begin{equation}\label{3.32}
0<\be<\frac 12-\frac 1p
\end{equation}
which is consistent since $p>2$.

It remains to estimate $J^{(3)}_N$ (which essentially is contained in the proof of H\" older continuity of the Brownian
motion). We start with the estimate similar to (\ref{3.27}) relying on the Chebyshev inequality and the standard moment
estimates of Gaussian random vectors,
\begin{eqnarray}\label{3.33}
&P\big\{\max_{[TN]\wedge(k+N^\al)\geq l>k\geq 0}\frac {|W_N(l/N)-W_N(k/N)|}{|(l-k)/N|^{\frac 12-\be}}>1\big\}\\
&\leq\sum_{[TN]\wedge(k+N^\al)\geq l>k\geq 0}P\big\{\frac {|W_N(l/N)-W_N(k/N)|}{|(l-k)/N|^{\frac 12-\be}}>1\big\}\nonumber\\
&\leq C_{10}(M)N^{-2\be M(1-\al)+1+\al}\nonumber
\end{eqnarray}
where $C_{10}(M)>0$ does not depend on $N\geq 1$. Choose again $M\geq\frac {3+\al}{2\be(1-\al)}$ which makes the right hand
side of (\ref{3.33}) a term in a converging series. Again, by the Borel--Cantelli lemma for any $\al\in(0,1)$ and
 $\be\in(0,1/2)$ there exists a finite a.s. random variable $C_{\al,\be}=C_{\al,\be}(\om)$ such that for all $N\geq 1$,
 \begin{equation}\label{3.34}
|W_N(l/N)-W_N(k/N)|\leq C_{\al,\be}\big\vert\frac lN-\frac kN\big\vert^{\frac 12-\be}\,\,\mbox{if}\,\, k+N^\al\geq l>k
\geq 0,\, l\leq[TN].
\end{equation}
Proceeding as in (\ref{3.31}) we obtain that
\begin{equation}\label{3.35}
J^{(3)}_N\leq C^p_{\al,\be}TN^{-(1-\al)(\frac p2-1-p\be)}
\end{equation}
where we assume again (\ref{3.32}). Finally, we conclude from (\ref{3.29})--(\ref{3.31}) and (\ref{3.35}) that
\begin{equation}\label{3.36}
\| S_N-\hat W_N\|_{p,[0,T]}=O(N^{-\tilde\del})\quad\mbox{a.s.},
\end{equation}
where $\tilde\del=\min(\al+p\del-1,(1-\al)(\frac p2-1-p\be))>0$, assuming (\ref{3.32}) and choosing $\al>1-p\del$.

Next, set
\[
\bar\bbW_N^{ij}(t)=\int_0^t W^i_N(u)dW_N^j(u)
\]
and
\[
\hat\bbW_N^{ij}(t)=\sum_{0\leq l<[Nt]}(W^j_N(\frac {l+1}N)-W^j_N(\frac lN))W_N^i(\frac lN).
\]
The relation (\ref{2.16}) from Theorem \ref{thm2.2} follows from (\ref{3.36}) and the following result which will be
used also in Sections \ref{roughsec4} and \ref{roughsec5}.
\begin{lemma}\label{lem3.10} For all $T>0$ and $p\in(2,3)$,
 \begin{equation}\label{3.37}
 \| W_N-\hat W_N\|_{p,[0,T]}=O(N^{-\del})\quad\mbox{a.s. as}\quad N\geq 1
 \end{equation}
 and
  \begin{equation}\label{3.38}
 \max_{1\leq i,j\leq d}\|\bar\bbW^{ij}_N-\hat\bbW^{ij}_N\|_{\frac p2,[0,T]}=O(N^{-\del})\quad\mbox{a.s. as}\quad N\geq1
 \end{equation}
 for some $\del>0$ which does not depend on $N\geq 1$.
 \end{lemma}
 We postpone the proof of this result till the end of Section \ref{roughsec6}. Observe that Lemma \ref{lem3.10}
 gives estimates for a version of the Euler--Maruyama rough paths approximation in $p$-variation norms for a family
 of processes depending on a parameter which by itself controls the step of approximations.

 \section{Strong approximations for iterated sums }\label{roughsec4}\setcounter{equation}{0}
\subsection{Estimates in the supremum norm}\label{subsec4.1}

As in the previous section we will prove first the estimate (\ref{2.17}) in the supremum norm and then will
extend it to the $p/2$-variation norm. Set $m_N=[N^{1-\ka}]$ with a small $\ka>0$ which will be chosen later
on, $\nu_N(l)=\max\{ jm_N:\, jm_N<l\}$ if $l>m_N$ and
\[
R_i(k)=R_i(k,N)=\sum_{l=(k-1)m_N}^{km_N-1}\xi_i(l)\,\,\mbox{for}\,\, k=1,2,...,\iota_N(T)
\]
 where $\iota_N(t)=[[Nt]m_N^{-1}]$.
For $1\leq i,j\leq e$ define
\begin{equation}\label{4.1}
\bbU_N^{ij}(t)=N^{-1}\sum_{l=m_N}^{\iota_N(t)m_N-1}\xi_j(l)\sum_{k=0}^{\nu_N(l)}\xi_i(k)=N^{-1}\sum_{1<l\leq\iota_N(t)}
R_j(l)\sum_{k=0}^{(l-1)m_N-1}\xi_i(k).
\end{equation}
Set also
\[
\bar\bbS_N^{ij}(t)=\bbS_N^{ij}(t)-t\sum_{l=1}^\infty E(\xi_i(0)\xi_j(l)).
\]

We will need the following result.
\begin{lemma}\label{lem4.1}
For all $i,j=1,...,e$ and $N\geq 1$,
\begin{equation}\label{4.2}
\sup_{0\leq t\leq T}|\bar\bbS_N^{ij}(t)-\bbU_N^{ij}(t)|=O(N^{-\del_1})\,\,\mbox{a.s.}
\end{equation}
for some $\del_1>0$ which does not depend on $N$.
\end{lemma}
\begin{proof}
First, we write
\begin{equation}\label{4.3}
|\bar\bbS_N^{ij}(t)-\bbU_N^{ij}(t)|\leq |\cI_N^{(1)}(t)|+|\cI_N^{(2)}(t)|+|\cI_N^{(3)}(t)|
+|\cI_N^{(4)}(t)|
\end{equation}
where
\[
\cI_N^{(1)}(t)=\cI_N^{1;ij}(t)=N^{-1}\sum_{l=m_N}^{\iota_N(t)m_N-1}\sum_{k=\nu_N(l)+1}^{l-1}\big(\xi_j(l)\xi_i(k)-
E(\xi_j(l)\xi_i(k))\big),
\]
\[
\cI_N^{(2)}(t)=\cI_N^{2;ij}(t)=N^{-1}\sum_{l=\iota_N(t)m_N}^{[Nt]-1}\sum_{k=0}^{l-1}\big(\xi_j(l)\xi_i(k)-
E(\xi_j(l)\xi_i(k))\big),
\]
\[
\cI_N^{(3)}(t)=\cI_N^{3;ij}(t)=N^{-1}\sum_{l=1}^{m_N-1}\sum_{k=0}^{l-1}\big(\xi_j(l)\xi_i(k)-
E(\xi_j(l)\xi_i(k))\big)
\]
and
\[
\cI_N^{(4)}(t)=\cI_N^{4;ij}(t)=\cI_N^{4,1;ij}(t)-\cI_N^{4,2;ij}(t)
\]
with
\[
\cI_N^{4,1;ij}(t)=N^{-1}\sum_{l=1}^{[Nt]-1}\sum_{k=0}^{l-1}E(\xi_j(l)\xi_i(k))-t\sum_{l=1}^\infty E(\xi_i(0)\xi_j(l))
\]
and
\[
\cI_N^{4,2;ij}(t)=N^{-1}\sum_{l=m_N}^{\iota_N(t)m_N-1}\sum_{k=0}^{\nu_N(l)}E(\xi_j(l)\xi_i(k)).
\]

By Lemma \ref{lem3.4},
\begin{eqnarray}\label{4.11}
&E\sup_{0\leq t\leq T}|\cI_N^{(1;i,j)}(t)|^{2M}=E\max_{0\leq k<[TN]}|\cI_N^{(1;i,j)}(k/N)|^{2M}\\
&\leq C_{11}(M)T^MN^{-M\ka},\,\,
 E\sup_{0\leq t\leq T}|\cI_N^{(2;i,j)}(t)|^{2M}\leq C_{11}(M)T^MN^{-M\ka}\nonumber\\
&\mbox{and}\,\,\, E\sup_{0\leq t\leq T}|\cI_N^{(3;i,j)}(t)|^{2M}\leq C_{11}(M)N^{-2M\ka}\nonumber
\end{eqnarray}
where $C_{11}(M)>0$ does not depend on $N$. Choosing $\del_1$ and $M$ such that $\del_1<\frac 12\ka$
and $M\geq 2(\ka-2\del_1)^{-1}$ we use, first, the Chebyshev inequality and then the Borel--Cantelli lemma to obtain
that with probability one,
\begin{eqnarray}\label{4.12}
&\sup_{0\leq t\leq T}|\cI_N^{(1;i,j)}(t)|=O(N^{-\del_1}),\,\,
\sup_{0\leq t\leq T}|\cI_N^{(2;i,j)}(t)|=O(N^{-\del_1})\\
&\mbox{and}\,\,\, \sup_{0\leq t\leq T}|\cI_N^{(3;i,j)}(t)|=O(N^{-2\del_1}).\nonumber
\end{eqnarray}

It remains to estimate $\cI_N^{(4;i,j)}(t)$. By (\ref{2.3}), (\ref{2.6}) and Lemma \ref{lem3.1},
\begin{eqnarray}\label{4.21}
&|E(\xi_j(l)\xi_i(k))|\leq 2L\rho(|k-l|/3)+|E\big(E(\xi_j(l)|\cF_{l-[\frac 13|k-l|],l+[\frac 13|k-l|]})\\
&\times E(\xi_i(k)|\cF_{k-[\frac 13|k-l|],k+[\frac 13|k-l|]})\big)|\leq 2L(L\phi(|k-l|/3)+\rho(|k-l|/3)).\nonumber
\end{eqnarray}
By the stationarity
\[
\sum_{k=0}^{l-1}E(\xi_j(l)\xi_i(k))=\sum_{n=1}^lE(\xi_j(n)\xi_i(0))
\]
and by (\ref{2.12}) and (\ref{4.21}) the limit
\[
\cL_{ij}=\lim_{l\to\infty}\sum_{n=1}^lE(\xi_j(n)\xi_i(0))=\sum_{n=1}^\infty E(\xi_j(n)\xi_i(0))
\]
exists and, moreover,
\begin{equation}\label{4.22}
|\sum_{k=0}^{l-1}E(\xi_j(l)\xi_i(k))-\cL_{ij}|\leq 2L\sum_{n=l+1}^\infty(L\phi(n/3)+\rho(n/3)).
\end{equation}
It follows from (\ref{2.12}) and (\ref{2.22}) that
\begin{equation}\label{4.23}
\sup_{0\leq t\leq T}|\cI_N^{(4,1;i,j)}(t)|\leq 2LN^{-1}\sum_{l=0}^\infty\sum_{n=l+1}^\infty(L\phi(n/3)+\rho(n/3))
+N^{-1}|\cL_{ij}|\leq C_{12}N^{-1}
\end{equation}
for some $C_{12}>0$ which does not depend on $N$.

Finally, by (\ref{4.21}),
\[
|\sum_{k=0}^{\nu_N(l)}E(\xi_j(l)\xi_i(k))|\leq 2L\sum_{n=l-\nu_N(l)}^\infty(L\phi(n/3)+\rho(n/3)),
\]
and so by (\ref{2.12}),
\begin{equation}\label{4.24}
\sup_{0\leq t\leq T}|\cI_N^{(4,2;i,j)}(t)|\leq 2LTN^{-(1-\ka)}\sum_{l=1}^{m_N}\sum_{n=l}^\infty(L\phi(n/3)+\rho(n/3))
\leq C_{13}N^{-(1-\ka)}
\end{equation}
for some $C_{13}>0$ which does not depend on $N$. The lemma now follows from (\ref{4.12}),
(\ref{4.23}) and (\ref{4.24}) together with the Chebyshev inequality and the Borel--Cantelli lemma.
\end{proof}

Now set
\[
S^i_N(t)=N^{-1/2}\sum_{0\leq k<[Nt]}\xi_i(k),\,\, i=1,...,e
\]
and observe that
\[
\bbU^{ij}_N(t)=\sum_{2\leq l\leq\iota_N(t)}\big(S^j_N\big(\frac {lm_N}N\big)-S^j_N\big(\frac {(l-1)m_N}N\big)\big)S^i_N
\big(\frac {(l-1)m_N}N\big).
\]
Define
\[
\bbV^{ij}_N(t)=\sum_{2\leq l\leq\iota_N(t)}\big(W^j_N\big(\frac {lm_N}N\big)-W^j_N\big(\frac {(l-1)m_N}N\big)\big)W^i_N
\big(\frac {(l-1)m_N}N\big)
\]
where $W_N=(W^1_N,...,W^e_N)$ is the $e$-dimensional Brownian motion with the covariance matrix $\vs$ (at the time 1)
appearing in (\ref{2.16}) which was constructed in Section \ref{roughsec3}. Then
\begin{eqnarray}\label{4.25}
&\sup_{0\leq t\leq T}|\bbU^{ij}_N(t)-\bbV^{ij}_N(t)|\leq\sum_{2\leq l\leq\iota_N(T)}\big(\big(\big\vert S^j_N\big(\frac {lm_N}N\big)-W^j_N\big(\frac {lm_N}N\big)\big\vert\\
&+\big\vert S^j_N\big(\frac {(l-1)m_N}N\big)-W^j_N\big(\frac {(l-1)m_N}N\big)\big\vert\big)\big\vert
S^i_N\big(\frac {(l-1)m_N}N\big)\big\vert\nonumber\\
&+\big\vert W^j_N\big(\frac {lm_N}N\big)-W^j_N\big(\frac {(l-1)m_N}N\big)\big\vert\big\vert  S^i_N\big(\frac {(l-1)m_N}N\big)
-W^i_N\big(\frac {(l-1)m_N}N\big)\big\vert\big).\nonumber
\end{eqnarray}

By Lemma \ref{lem3.4},
\begin{equation}\label{4.26}
E\max_{2\leq l\leq\iota_N(T)}\big\vert S^i_N\big(\frac {(l-1)m_N}N\big)\big\vert^{2M}
\leq C_{14}(M)T^{M}
\end{equation}
for some $C_{14}(M)>0$ which does not depend on $N$.
By the Chebyshev inequality for any $\gam>0$,
\begin{equation}\label{4.27}
P\{\max_{2\leq l\leq\iota_N(T)}\big\vert S^i_N\big(\frac {(l-1)m_N}N\big)\big\vert>N^\gam\}\leq C_{14}(M)T^{M}N^{-2M\gam}.
\end{equation}
Taking $M\geq\gam^{-1}$ the right hand side of (\ref{4.27}) becomes a term of a converging series and by the
Borel--Cantelli lemma we obtain that for any $\gam>0$,
\begin{equation}\label{4.28}
\max_{2\leq l\leq\iota_N(T)}\big\vert S^i_N\big(\frac {(l-1)m_N}N\big)\big\vert=O(N^\gam)\quad\mbox{a.s.}
\end{equation}
Next, write
\begin{eqnarray*}
&E\max_{2\leq l\leq\iota_N(T)}\big\vert W^j_N\big(\frac {lm_N}N\big)-W^j_N\big(\frac {(l-1)m_N}N\big)\big\vert^{2M}\\
&\leq\sum_{2\leq l\leq\iota_N(T)}E\big\vert W^j_N\big(\frac {lm_N}N\big)-W^j_N\big(\frac {(l-1)m_N}N\big)\big\vert^{2M}.
\end{eqnarray*}
Using the standard moment estimates for the Brownian motion and relying on the Chebyshev inequality and the Borel--Cantelli
lemma we obtain similarly to (\ref{4.27}) and (\ref{4.28}) that for $\gam<\ka/2$,
\begin{equation}\label{4.29}
\max_{2\leq l\leq\iota_N(T)}\big\vert W^j_N\big(\frac {lm_N}N\big)-W^j_N\big(\frac {(l-1)m_N}N\big)\big\vert=O(N^{-\gam})\quad\mbox{a.s}.
\end{equation}
Now, combining (\ref{2.16}), (\ref{4.25}), (\ref{4.28}) and (\ref{4.29}) we obtain that
\begin{equation}\label{4.30}
\sup_{0\leq t\leq T}|\bbU^{ij}_N(t)-\bbV^{ij}_N(t)|=O(N^{-\del_2})\quad\mbox{a.s.}
\end{equation}
where $\del_2=\del-\ka-\gam$ and we choose $\ka$ and $\gam$ so small that $\del_2>0$.

Next, observe that
\begin{eqnarray}\label{4.31}
&\sup_{0\leq t\leq T}|\int_0^tW_N^i(s)dW^j_N(s)-\bbV^{ij}_N(t)|\\
&\leq\sup_{0\leq t\leq T}|J_N^{(1;i,j)}(t)|+\sup_{0\leq t\leq m_NN^{-1}}|J_N^{(2;i,j)}(t)|+\sup_{0\leq t\leq T}|J_N^{(3;i,j)}(t)|
\nonumber\end{eqnarray}
where
\[
J_N^{(1;i,j)}(t)=\sum_{2\leq l\leq\iota_N(t)}\int_{(l-1)m_NN^{-1}}^{lm_NN^{-1}}\big(W^i_N(s)-
W^i_N\big(\frac {(l-1)m_N}N\big)\big)dW^j(s),
\]
\[
J_N^{(2;i,j)}(t)=\int_0^tW_N^i(s)dW^j_N(s)\,\,\mbox{and}\,\,
J_N^{(3;i,j)}(t)=\int_{(\iota_N(t)\vee 1)m_NN^{-1}}^{t}W_N^i(s)dW^j_N(s).
\]
By the standard (martingale) moment estimates for stochastic integrals (see, for instance, \cite{Mao}, Section 1.7),
\[
E\sup_{0\leq t\leq T}|J_N^{(1;i,j)}(t)|^{2M}\leq C_{15}(M,T)N^{-M\ka},
\]
\[
E\sup_{0\leq t\leq m_NN^{-1}}|J_N^{(2;i,j)}(t)|^{2M}\leq C_{15}(M,T)N^{-M\ka}\quad\mbox{and}
\]
\begin{eqnarray*}
&E\sup_{0\leq t\leq T}|J_N^{(3;i,j)}(t)|^{2M}\\
&\leq\sum_{2\leq l\leq\iota_N(T)}E\sup_{0\leq u\leq m_N}
|\int_{(l-1)m_NN^{-1}}^{((l-1)m_N+u)N^{-1}}W_N^i(s)dW^j_N(s)|^{2M}\\
&\leq C_{15}(M,T)N^{-(M-1)\ka}
\end{eqnarray*}
where $C_{15}(M,T)>0$ does not depend on $N$. Taking $\del_3<\frac 12\ka$ and $M\geq(2\ka+1)(\ka-2\del_3)^{-1}$ and
employing the Chebyshev inequality together with the Borel--Cantelli lemma in the same way as above, we conclude that
\begin{equation}\label{4.32}
\sup_{0\leq t\leq T}|J_N^{(1;i,j)}(t)|+\sup_{0\leq t\leq T}|J_N^{(2;i,j)}(t)|+\sup_{0\leq t\leq T}|J_N^{(3;i,j)}(t)|=
O(N^{-\del_3})\,\,\,\mbox{a.s.}
\end{equation}
This together with (\ref{4.2}), (\ref{4.30}) and (\ref{4.31}) completes the proof of (\ref{2.17}) in the supremum norm.

\subsection{$p/2$-variation norm estimates}\label{subsec4.2}
Next, we extend the supremum norm estimate of Section \ref{subsec4.1} to the $p/2$-variation norm estimate which
will yield (\ref{2.17}) of Theorem \ref{thm2.2}. First, we will derive certain H\" older continuity type estimates
for our sums. For $0\leq s<t\leq T$ and $i,j=1,...,e$ set
\[
\bar\bbS_N^{ij}(s,t)=N^{-1}\sum_{[sN]\leq k<l<[Nt]}\xi_i(k)\xi_j(l)-(t-s)\sum_{l=1}^\infty E(\xi_i(0)\xi_j(l)),
\]
and so
\begin{equation}\label{4.33}
\bbS_N^{ij}(s,t)=\bar\bbS_N^{ij}(s,t)+(t-s)\sum_{l=1}^\infty E(\xi_i(0)\xi_j(l)),
\end{equation}
recalling that by (\ref{2.12}) and (\ref{4.21}) the series in the right hand side of (\ref{4.33}) converges absolutely.
\begin{lemma}\label{lem4.1+}
There exists  a finite a.s. random variable $C^\bbS_{\al,\be}>0$ which does not depend on $n,m$ or $N$ such that
\begin{equation}\label{4.38}
\big\vert\bar\bbS_N^{ij}(\frac mN,\frac nN))\big\vert\leq C^\bbS_{\al,\be}|\frac nN-\frac mN|^{1-\be}\,\,\mbox{provided}\,\,
m+N^\al\geq n>m\geq 0,\,[TN]>n.
\end{equation}
\end{lemma}
\begin{proof}
We will estimate $\bar\bbS_N^{ij}(s,t)$ relying on Lemma \ref{3.4}.
 Set
\[
\mu_{kl}^{ij}=\xi_j(l)\sum_{r=k}^{l-1}\xi_i(r)\,\,\mbox{and}\,\,\Sig_N^{ij}(s,t)=N^{-1}
\sum_{l=[sN]}^{[tN]-1}(\mu^{ij}_{[sN]l}-E\mu^{ij}_{[sN]l}).
\]
 By (\ref{2.3}), (\ref{2.6}) and (\ref{4.21}) for $0\leq m<n<[TN]$,
\begin{equation}\label{4.34}
\big\vert\bar\bbS_N^{ij}(\frac mN,\frac nN)-\Sig_N^{ij}(\frac mN,\frac nN)\big\vert\leq (\frac nN-\frac mN)
\sum_{l=1}^\infty |E(\xi_i(0)\xi_j(l))|.
\end{equation}

Next, for any $\al,\be\in(0,1)$ by the Chebyshev inequality
\begin{eqnarray}\label{4.35}
&\cP=P\big\{\max_{m+N^\al\geq n>m\geq 0,\,[TN]>n}\frac {|\Sig_N^{ij}(\frac mN,\frac nN)|}{|(n-m)/N|^{1-\be}}>1\}\\
&\leq\sum_{m+N^\al\geq n>m\geq 0,\,[TN]>n}P\big\{ |\sum_{m\leq r<n}(\mu^{ij}_{mr}-E\mu^{ij}_{mr})|>N^\be(n-m)^{1-\be}\big\}
\nonumber\\
&\leq N^{-2\be M}\sum_{m+N^\al\geq n>m\geq 0,\,[TN]>n}(n-m)^{-2M(1-\be)}\nonumber\\
&\times E|\sum_{m\leq r<n}(\mu^{ij}_{mr}-E\mu^{ij}_{mr})|^{2M}.
\nonumber\end{eqnarray}
By Lemma \ref{lem3.4},
\begin{equation}\label{4.36}
E|\sum_{m\leq r<n}(\mu^{ij}_{mr}-E\mu^{ij}_{mr})|^{2M}\leq C_{16}(M)(n-m)^{2M}
\end{equation}
for some $C_{16}(M)>0$ which does not depend on $m,n$ or $N$.
This together with (\ref{4.35}) yields that
\begin{equation}\label{4.37}
\cP\leq C_{17}(M)N^{-2M\be(1-\al)+1+\al}
\end{equation}
where $ C_{17}(M)=C_{16}(M)2^{2M\be+1}(2M\be+1)^{-1}$. For any $\al,\be\in(0,1)$ we
pick up $M\geq 1$ such that $2M\be(1-\al)-1-\al\geq 2$ which makes the right hand side of (\ref{4.37}) a term of
a converging sequence. Thus, by the Borel--Cantelli lemma we conclude that there exists a finite a.s. random variable
$C^\Sig_{\al,\be}=C^\Sig_{\al,\be}(\om)$ such that for all $N\geq 1$,
\[
|\Sig_N^{ij}(\frac mN,\frac nN))|\leq C^\Sig_{\al,\be}|\frac nN-\frac mN|^{1-\be}\,\,\mbox{provided}\,\,
m+N^\al\geq n>m\geq 0,\,[TN]>n,
\]
which together with (\ref{4.33}) and (\ref{4.34}) yields (\ref{4.38}) for all $N\geq 1$.
\end{proof}

Next, we proceed similarly to Section \ref{subsec3.4}. For $0\leq s<t\leq T$ define
\[
\hat\bbW_N^{ij}(s,t)=\sum_{[Ns]\leq l<[Nt]}(W^j_N(\frac {l+1}N)-W^j_N(\frac lN))(W^i_N(\frac lN)-W^i_N(\frac {[Ns]}N))
\]
and
\[
\bar\bbW_N^{ij}(s,t)=\int_s^t(W^i_N(u)-W^i_N(s))dW_N^j(u)
\]
while $\hat\bbW_N^{ij}(0,t)$ and $\bar\bbW_N^{ij}(0,t)$ will be denoted, as before, just by $\hat\bbW_N^{ij}(t)$ and
 $\bar\bbW_N^{ij}(t)$, respectively. Let $0\leq t_0<t_1<...<t_m=T$ and observe that if $[t_qN]=[t_{q+1}N]$ then
 \[
 \hat\bbW_N^{ij}(t_q,t_{q+1})=0\,\,\mbox{and}\,\, |\bar\bbS_N^{ij}(t_q,t_{q+1})|=D_{ij}(t_{q+1}-t_q)\leq D_{ij}N^{-1}
 \]
 where $D_{ij}=|\sum_{l=1}^\infty E(\xi_i(0)\xi_j(l))|<\infty$. Hence,
\begin{eqnarray}\label{4.39}
&\sum_{0\leq q<m} |\bar\bbS_N^{ij}(t_q,t_{q+1})- \hat\bbW_N^{ij}(t_q,t_{q+1})|^{p/2}\\
&\leq\sum_{0\leq r<n}\big\vert\bar\bbS_N^{ij}(\frac {k_{r}}N,\frac {k_{r+1}}N)-\hat\bbW_N^{ij}(\frac {k_{r}}N,\frac {k_{r+1}}N)|^{p/2}\nonumber\\
&+J^{(0)}_N\leq J^{(0)}_N+J_N^{(1)}+2^{p/2-1}(J_N^{(2)}+J_N^{(3)}),\nonumber
\end{eqnarray}
where
\begin{equation}\label{4.39+}
J^{(0)}_N=D_{ij}^{p/2}\sum_{q:|t_{q+1}-t_q|<N^{-1}}|t_{q+1}-t_q|^{p/2}\leq D_{ij}^{p/2}TN^{-(\frac p2-1)},
\end{equation}
\[
J_N^{(1)}=\sum_{0\leq r<n,k_{r+1}-k_r>N^\al}\big\vert\bar\bbS_N^{ij}(\frac {k_{r}}N,\frac {k_{r+1}}N)-\hat\bbW_N^{ij}(\frac {k_{r}}N,\frac {k_{r+1}}N)\big\vert^{p/2},
\]
\[
J_N^{(2)}=\sum_{0\leq r<n,k_{r+1}-k_r\leq N^\al}\big\vert\bar\bbS_N^{ij}(\frac {k_{r}}N,\frac {k_{r+1}}N)\big\vert^{p/2}\,\,
\mbox{and},
\]
\[
J_N^{(3)}=\sum_{0\leq r<n,k_{r+1}-k_r\leq N^\al}\big\vert\hat\bbW_N^{ij}(\frac {k_{r}}N,\frac {k_{r+1}}N)|^{p/2},
\]
$k_r=[t_{q_r}N]$ and $0=t_{q_0}<t_{q_1}<...<t_{q_m}=T$ is the maximal subsequence of $t_0,t_1,...,t_m$ such that $[t_{q_r}N]
<[t_{q_{r+1}}N]$, $r=0,1,...,m-1$.

Observe that
\[
\bar\bbS^{ij}_N(s,t)=\bar\bbS^{ij}_N(t)-\bar\bbS^{ij}_N(s)-S^i_N(s)(S^j_N(t)-S^j_N(s)),
\]
\[
\bar\bbW^{ij}_N(s,t)=\bar\bbW^{ij}_N(t)-\bar\bbW^{ij}_N(s)-W^i_N(s)(W^j_N(t)-W^j_N(s))
\]
and
\[
\hat\bbW^{ij}_N(s,t)=\hat\bbW^{ij}_N(t)-\hat\bbW^{ij}_N(s)-W^i_N(\frac {[Ns]}N)(W^j_N(\frac {[Nt]}N)-W^j_N(\frac {[Ns]}N)).
\]
In order to estimate $J^{(1)}_N$ we use (\ref{2.16}) and (\ref{2.17}) for the supremum norm proved in Sections \ref{subsec3.3}
and \ref{subsec4.1} together with Lemma \ref{lem3.10} and
 observe that there exists no more than $[TN^{1-\al}]$ disjoint intervals $(k_{r},k_{r+1})$ in $[0,[TN]]$
with the length exceeding $N^\al$ which gives as in (\ref{3.30}) that
\begin{equation}\label{4.40}
J_N^{(1)}\leq C_{18}N^{1-\al-p\del/2}(1+\sup_{0\leq t\leq T}|S_N^j(t)|^{p/2}+\sup_{0\leq t\leq T}|W^i_N(t)|^{p/2}
+\sup_{0\leq t\leq T}|W^j_N(t)|^{p/2})
\end{equation}
where $C_{18}>0$ is an a.s. finite random variable which does not depend on $N$ or on the choice of $t_1,...,t_m$.
 Using (\ref{2.16}) we have
\begin{equation}\label{4.41}
\sup_{0\leq t\leq T}|S^j_N(t)|^{p/2}\leq 2^{p/2-1}(C_{19}N^{-p\del/2}+\sup_{0\leq t\leq T}|W^j_N(t)|^{p/2})\,\,
\mbox{whenever}\,\, 1\leq j\leq e
\end{equation}
where $C_{19}>0$ is an a.s. finite random variable which does not depend on $N$.
By the standard martingale moment estimates for the Brownian motion for any $M\geq 1$,
\[
E\sup_{0\leq t\leq T}|W_N^j(t)|^{pM}\leq C_{20}(M,T)<\infty\,\,\mbox{whenever}\,\, 1\leq j\leq e,
\]
where $C_{20}(M,T)>0$ does not depend on $N$. Applying as above the Chebyshev inequality and then the Borel--Cantelli
lemma we see that for any $\gam>0$,
\[
\sup_{0\leq t\leq T}|W_N^j(t)|=O(N^\gam)\quad\mbox{a.s.}
\]
This together with (\ref{4.40}) and (\ref{4.41}) gives
\begin{equation}\label{4.42}
J^{(1)}_N=O(N^{1-\al-p\del/2-p\gam/2})
\end{equation}
where we choose $\al$ so close to 1 that $1-\al-p\del/2<0$.

By (\ref{4.38}) we obtain similarly to (\ref{3.31}) that
\begin{eqnarray}\label{4.43}
&J^{(2)}_N\leq C^\bbS_{\al,\be}N^{-p(1-\be)/2}\sum_{0\leq j<n,\, k_{j+1}-k_j\leq N^\al}|k_{j+1}-k_j|^{p(1-\be)/2}\\
&\leq C^\bbS_{\al,\be}TN^{-(1-\al)(p/2-1-p\be/2)}\nonumber
\end{eqnarray}
where we assume that
\begin{equation}\label{4.44}
0<\be<1-\frac 2p
\end{equation}
which is consistent since $p>2$.

In order to estimate $J_N^{(3)}$ we proceed in the same way as in (\ref{3.33})--(\ref{3.35}) and (\ref{4.35})
writing
\begin{eqnarray}\label{4.45}
&\quad P\big\{\max_{m+N^\al\geq n>m\geq 0,\,[TN]>n}\frac {|\sum_{m\leq
 l<n}(W^j_N(\frac {l+1}N)-W^j_N(\frac lN))(W_N^i(\frac lN)-W^i_N(\frac mN))|}{|(n-m)/N|^{1-\be}}\\
 &>1\big\}\leq N^{2M(1-\be)}\sum_{m+N^\al\geq n>m\geq 0,\,[TN]>n}(n-m)^{-2M(1-\be)}\nonumber\\
&\times E\big\vert\sum_{m\leq
 l<n}(W^j_N(\frac {l+1}N)-W^j_N(\frac lN))(W_N^i(\frac lN)-W^i_N(\frac mN))\big\vert^{2M}\nonumber\\
&\leq C_{21}(M)N^{-2M\be}\sum_{m+N^\al\geq n>m\geq 0,\,[TN]>n}(n-m)^{2M\be}\nonumber\\
&\leq C_{21}(M)TN^{-2M\be(1-\al)+1+\al}\nonumber
\end{eqnarray}
where $C_{21}(M)>0$ does not depend on $N\geq 1$ and we rely on the Chebyshev inequality and the standard moment estimates
for the martingale $M_n=\sum_{m\leq l<n}(W^j_N(\frac {l+1}N)-W^j_N(\frac lN))(W_N^i(\frac lN)-W^i_N(\frac mN)$
 (see, for instance, \cite{Mao}) or, alternatively, use Lemma \ref{lem3.4}. Choosing $M\geq (3+\al)(2\be(1-\al))^{-1}$ we
 obtain in the right
 hand side of (\ref{4.45}) a term of a converging sequence and application of the Borel--Cantelli lemma provides us with a
 finite a.s. random variable $C^W_{\al,\be}>0$ such that for all $N\geq 1$,
 \begin{equation}\label{4.46}
 \big\vert\int_{m/N}^{n/N}(W^i_N(u)-W^i_N(m/N))dW_N^j(u)\big\vert\leq C^W_{\al,\be}\big\vert\frac nN-\frac mN\big\vert^{1-\be}
 \end{equation}
 whenever $N^\al\geq n-m>0,\, n<[TN],\, m\geq 0$. In the same way as in (\ref{4.43}) we have now
 \begin{equation}\label{4.47}
 J_N^{(3)}\leq C^W_{\al,\be}TN^{-(1-\al)(p/2-1-p\be/2)}.
 \end{equation}
 Collecting (\ref{4.33}), (\ref{4.39}), (\ref{4.42}), (\ref{4.43}) and (\ref{4.44}) we obtain (\ref{2.17}) for some $\del>0$
 which together with Lemma \ref{lem3.10} completes the proof of Theorem \ref{thm2.2}.  \qed

\section{Continuous time case: proof of Theorem \ref{thm2.5} }\label{roughsec5}\setcounter{equation}{0}
\subsection{Basic estimates}\label{subsec5.1}
First observe that by (\ref{2.6}), (\ref{2.19}) and (\ref{2.21}),
\begin{eqnarray}\label{5.1}
&|(\eta\circ\vt^m(\om)-E(\eta\circ\vt^m|\cF_{m-n,m+n})(\om)|\\
&=|\int_0^{\tau\circ\vt^m(\om)}\xi(s,\vt^m\om)ds-E(
\int_0^{\tau\circ\vt^m(\om)}\xi(s,\vt^m\om)ds|\cF_{m-n,m+n})(\om)|\nonumber\\
&\leq 2L|\tau\circ\vt^m(\om)-E(\tau\circ\vt^m|\cF_{m-n,m+n})(\om)|+|\int_0^{E(\tau\circ\vt^m(\om)|\cF_{m-n,m+n})}
(\xi(s,\vt^m\om)\nonumber\\
&-E(\xi(s,\vt^m\om)|\cF_{m-n,m+n})(\om))ds|\leq(2L+\hat L)\rho(n),\nonumber
\end{eqnarray}
and so the sequence $\eta(k),\, k\in\bbZ$ satisfies the conditions of Theorem \ref{thm2.1}.

With a slight abuse of notations we set now for each $t\in[0,T]$,
\[
S_\ve(t)=\ve\sum_{0\leq k<[\ve^{-2}t]}\eta(k)\,\,\mbox{and}\,\,\bbS^{ij}_\ve(t)=\ve^2\sum_{0\leq k<l<[\ve^{-2}t]}
\eta_i(k)\eta_j(l).
\]
In view of the assumptions of Theorem \ref{thm2.4} we can apply Theorem \ref{thm2.2} to $S_{1/\sqrt N}$ and
$\bbS_{1/\sqrt N}$ to obtain
\begin{equation}\label{5.2}
\| S_{1/\sqrt N}-W_N\|_{p,[0,T]}=O(N^{-\del})\quad\mbox{ a.s.}
\end{equation}
and
\begin{equation}\label{5.3}
\max_{0\leq i,j\leq d}\|\bbS^{ij}_{1/\sqrt N}-\bbW^{ij}_N\|_{p/2,[0,T]}=O(N^{-\del})\quad\mbox{ a.s.}
\end{equation}
where $W_N=N^{-1/2}\cW(Nt)$ and $\cW$ is the universal Brownian motion constructed via Theorem \ref{thm3.8} for
the sequence $\eta(k),\, k\in\bbZ$ in place of the sequence $\xi(k),\, k\in\bbZ$ considered there. Here
\[
\bbW^{ij}_N(t)=\int_0^tW_N^i(s)dW_N^j(s)+t\sum_{l=1}^\infty E(\eta_i(0)\eta_j(l)).
\]

Next, set $N_\ve=[\ve^{-2}]$. Then, by (\ref{2.6}) and (\ref{2.19}),
\begin{eqnarray}\label{5.4}
&|S_\ve(t)-S_{1/N_\ve}(t)|\leq |\ve\sqrt {N_\ve}-1|N^{-1}_\ve|\sum_{0\leq k<[N_\ve t]}\eta(k)|\\
&+\ve|\sum_{[N_\ve t]\leq k<[\ve^{-2}t]}\eta(k)|\leq\ve^2|S_{1/\sqrt {N_\ve}}(t)|+\ve(T+1)L\hat L.\nonumber
\end{eqnarray}
This together with (\ref{5.2}) and the arguments similar to Section \ref{subsec3.4} yields easily that
\begin{equation}\label{5.5}
\| S_\ve-W_{[\ve^{-2}]}\|_{p,[0,T]}=O(\ve^\del)\quad\mbox{a.s.}
\end{equation}
for some $\del>0$ where a.s. is simultaneously over $\ve\in(0,1)$.

Next,
\begin{eqnarray}\label{5.6}
&|\bbS^{ij}_\ve(t)-\bbS^{ij}_{1/N_\ve}(t)|\leq |\ve^2\sqrt {N_\ve}-1|N^{-1}_\ve|\sum_{0\leq k<l<[N_\ve t]}\eta_i(k)\eta_j(l)|\\
&+\ve^2|\sum_{[N_\ve t]\leq l<[\ve^{-2}t]}\eta_j(k)\sum_{k=0}^{l-1}\eta_i(k)|\nonumber\\
&\leq\ve^2|\bbS^{ij}_{1/\sqrt {N_\ve}}(t)|+\ve L\hat L|\sum_{[N_\ve t]\leq l<[\ve^{-2}t]}S^i_{1/\sqrt {N_\ve}}((l-1)\ve^2)|.\nonumber
\end{eqnarray}
Relying on (\ref{5.2}), (\ref{5.3}) and the arguments similar to Section \ref{subsec4.2} we derive easily that
\begin{equation}\label{5.7}
\| \bbS^{ij}_\ve-\bbW^{ij}_{[\ve^{-2}]}\|_{p/2,[0,T]}=O(\ve^\del)\quad\mbox{a.s.}
\end{equation}
for some $\del>0$ where a.s. is simultaneously over $\ve\in(0,1)$. In fact, employing the standard moment estimates for the Brownian motion together with its H\" older continuity and the Borel--Cantelli lemma
we can replace $W_{[\ve^{-2}]}(t)=[\ve^{-2}]^{-1/2}\cW([\ve^{-2}]t)$
 in (\ref{5.5}) by $W^\ve(t)=\ve\cW(\ve^{-2}t)$ and $\bbW^{ij}_{[\ve^{-2}]}$ in (\ref{5.7}) by
 \[
 \tilde\bbW_{ij}^\ve(t)=\int_0^tW_i^\ve(s)dW^\ve_j(s)+t\sum_{l=1}^\infty E(\eta_i(0)\eta_j(l)).
 \]
 Hence, in addition to (\ref{5.5}) and (\ref{5.7}) we have also
 \begin{equation}\label{5.8}
 \| S_\ve-W^\ve\|_{p,[0,T]}=O(\ve^\del)\quad\mbox{a.s.}
 \end{equation}
 and
 \begin{equation}\label{5.9}
\| \bbS^{ij}_\ve-\tilde\bbW^\ve_{ij}\|_{p/2,[0,T]}=O(\ve^\del)\quad\mbox{a.s.}
\end{equation}

\subsection{A renewal type lemma}\label{subsec5.2}

The following result will be used several times in this section and though it is essentially known we will provide its self-contained proof for completeness.
\begin{lemma}\label{lem5.3}
Let $n(s)=n(s,\om)=0$ if $\tau(\om)>s$ and
\[
n(s)=n(s,\om)=\max\{ k:\,\sum_{j=0}^{k-1}\tau\circ\vt^j(\om)\leq s\}.
\]
Then for any $M\geq 1$ and $s,t\geq 0$,
\begin{equation}\label{5.12}
E|n(s\bar\tau)-s|^{2M}\leq K(M)s^M\,\,\mbox{and}\,\, E\sup_{0\leq s\leq t}|n(s\bar\tau)-s|^{2M}\leq K(M)t^{M+1},
\end{equation}
where $K(M)>0$ does not depend on $s$ and $t$, and for any $\gam>0$,
\begin{equation}\label{5.13}
|n(s\bar\tau)-s|=O(s^{\frac 12+\gam})\quad\mbox{a.s.}
\end{equation}
\end{lemma}
\begin{proof}
Observe first that without loss of generality it suffices to prove (\ref{5.12}) and (\ref{5.13}) for $s=k,\, k=1,2,...$
since $n(s\bar\tau)-n([s]\bar\tau)\leq\hat L^2$ in view of (\ref{2.19}). Next,
\begin{equation}\label{5.14}
m^{-M}E|n(m\bar\tau)-m|^{2M}\leq 1+\sum_{k=1}^\infty P\{ |n(m\bar\tau)-m|>k^{1/2M}\sqrt m\}.
\end{equation}
Now we have the following events inclusions
\begin{eqnarray*}
&\{|n(m\bar\tau)-m|>k^{1/2M}\sqrt m\}\subset\{\sum_{0\leq j\leq m+k^{1/2M}\sqrt m}\tau\circ\vt^j<m\bar\tau\}\\
&\cup\{\sum_{0\leq j\leq m-k^{1/2M}\sqrt m}\tau\circ\vt^j>m\bar\tau\}\subset\Gam_k\cup\Del_k
\end{eqnarray*}
where
\[
\Gam_k=\{\sum_{0\leq j\leq m+k^{1/2M}\sqrt m}(\tau\circ\vt^j-\bar\tau)<-[k^{1/2M}\sqrt m]\bar\tau\}
\]
and
\[
\Del_k=\{\sum_{0\leq j\leq m-k^{1/2M}\sqrt m}(\tau\circ\vt^j-\bar\tau)>[k^{1/2M}\sqrt m]\bar\tau\}.
\]

Next, by Lemma \ref{lem3.4},
\[
E(\sum_{0\leq j\leq n-1}(\tau\circ\vt^j-\bar\tau))^{6M}\leq C^\tau_1(3M)n^{3M}
\]
where the index $\tau$ in $C_1^\tau$ means that we apply Lemma \ref{lem3.4} for sums of $\tau$'s in place of $\xi$'s there.
Hence, by the Chebyshev inequality for $k,m\geq 1$,
\begin{eqnarray*}
&\max(P(\Gam_k),\, P(\Del_k))\leq C_1^\tau(3M)(m+k^{1/2M}\sqrt m)^{3M}[k^{1/2M}\sqrt m]^{-6M}\bar\tau^{-6M}\\
&\leq\tilde K(M)k^{-3/2}
\end{eqnarray*}
for some $\tilde K(M)>0$ which does not depend on $k$ and $m$. Summing in $k\geq 1$ we obtain the
first estimate in (\ref{5.12}) from (\ref{5.14}). The second estimate in (\ref{5.12}) follows by
\[
E\max_{0\leq m\leq[t]+1}|n(m\bar\tau)-m|^{2M}\leq\sum_{0\leq m\leq[t]+1}E|n(m\bar\tau)-m|^{2M}\leq K(M)
\sum_{0\leq m\leq[t]+1}m^M.
\]

Finally, by (\ref{5.12}) and the Chebyshev inequality,
\[
P\{|n(m\bar\tau)-m|>m^{\frac 12+\gam}\}\leq K(M)m^{-M(1+2\gam)}m^M=K(M)m^{-2M\gam}.
\]
Choosing $M\geq\gam^{-1}$ we obtain in the right hand side here a term of a converging sequence and by the
Borel--Cantelli lemma (\ref{5.13}) follows.
\end{proof}

\subsection{Proof of (\ref{2.25}) and (\ref{2.26}) in the supremum norm}\label{subsec5.3}

Set
\[
U^\ve(t)=\ve\sum_{0\leq k<n(t\bar\tau\ve^{-2})}\eta(k)\,\,\mbox{and}\,\,\bbU^\ve_{ij}(t)=\ve^2\sum_{0\leq k<l<
n(t\bar\tau\ve^{-2})}\eta_i(k)\eta_j(l).
\]
Then for $M\geq 1$,
\begin{eqnarray*}
&\sup_{0\leq t\leq T}|S_\ve(t)-U^\ve(t)|^{2M}\\
&=\ve^{2M}\sup_{0\leq t\leq T}|\sum_{\min([\ve^{-2}t],n(t\bar\tau\ve^{-2}))
\leq k<\max([\ve^{-2}t],n(t\bar\tau\ve^{-2}))}\eta(k)|^{2M}\\
&\leq\ve^{2M}\sum_{0\leq m\leq\ve^{-2}T}\sup_{0\leq t\leq T}\max_{0\leq k\leq |n(t\bar\tau\ve^{-2})-\ve^{-2}t|}
|\sum_{l=0}^k\eta(l+m)|^{2M}\\
&\leq\ve^{2M}\sum_{0\leq m\leq\ve^{-2}T}\max_{0\leq k\leq \ve^{-3/2}}|\sum_{l=0}^k\eta(l+m)|^{2M}\\
&+\ve^{2M}\sum_{0\leq m\leq\ve^{-2}T}\max_{0\leq k\leq T\hat L^2\ve^{-2}}|\sum_{l=0}^k\eta(l+m)|^{2M}
\bbI_{\sup_{0\leq t\leq T}|n(t\bar\tau\ve^{-2})-\ve^{-2}t|>\ve^{-3/2}}
\end{eqnarray*}
since $|n(t\bar\tau\ve^{-2})-\ve^{-2}t|\leq T\hat L^2\ve^{-2}$. Applying Lemmas \ref{lem3.3}, \ref{lem3.4} or
Theorem B from \cite{Ser} to sums $\sum_{l=0}^k\eta(l)$ and taking into account stationarity of the sequence $\eta(l),\, l\geq 0$
we obtain that
\[
E\max_{0\leq k\leq\ve^{-3/2}}|\sum_{l=0}^k\eta(l+m)|^{2M}\leq C_{22}(M)\ve^{-3M/2}
\]
for some $C_{22}(M)>0$ which does not depend on $\ve>0$. Using, in addition, Lemma \ref{lem5.3} together with the
Cauchy-Schwarz and Chebyshev inequalities we derive that
\begin{eqnarray*}
&E(\max_{0\leq k\leq T\hat L^2\ve^{-2}}|\sum_{l=0}^k\eta(l+m)|^{2M}\bbI_{\sup_{0\leq t\leq T}
|n(t\bar\tau\ve^{-2})-\ve^{-2}t|>\ve^{-3/2}})\\
&\leq\big(E\max_{0\leq k\leq T\hat L\ve^{-2}}|\sum_{l=0}^k\eta(l+m)|^{2(M+1)}\big)^{\frac M{M+1}}\\
&\times\big(P\{\sup_{0\leq t\leq T}|n(t\bar\tau\ve^{-2})-\ve^{-2}t|>\ve^{-3/2}\}\big)^{\frac 1{M+1}}\\
&\leq\tilde C(M)(T\hat L^2\ve^{-2}+1)^M\ve^{3(M+1)}(E\sup_{0\leq t\leq T}|n(t\bar\tau\ve^{-2})-
\ve^{-2}t|^{2(M+1)^2})^{\frac 1{M+1}}\\
&\leq C_{23}(M)\ve^{-M}
\end{eqnarray*}
for some $\tilde C(M),\, C_{23}(M)>0$ which do not depend on $\ve\in(0,1)$.

Hence,
\[
\sup_{0\leq t\leq T}|S_\ve(t)-U^\ve(t)|^{2M}\leq C_{22}(M)T\ve^{\frac M2-2}+C_{23}(M)\ve^{M-2},
\]
and so by the Chebyshev inequality,
\[
P\{\sup_{0\leq t\leq T}|S_\ve(t)-U^\ve(t)|\geq\ve^{1/8}\}\leq(C_{22}(M)+ C_{23}(M))T\ve^{\frac M4-2}.
\]
Taking $M\geq 24$ and $\ve=\ve_n=n^{-1/2}$ we obtain by the Borel--Cantelli lemma that
\[
\sup_{0\leq t\leq T}|S_{\ve_n}(t)-U^{\ve_n}(t)|=O(n^{-1/16})\,\,\mbox{a.s.}
\]
When $(n+1)^{-1/2}\leq\ve\leq n^{-1/2}$ then
\[
|S_\ve(t)-S_{\ve_n}(t)|\leq (2T+1)Ln^{-1/2}\,\,\mbox{and}\,\,|U^\ve(t)-U^{\ve_n}(t)|\leq (2T+1)L\hat Ln^{-1/2},
\]
and so
\begin{equation}\label{5.14+}
\sup_{0\leq t\leq T}|S_\ve(t)-U^\ve(t)|=O(\ve^{1/8})\,\,\mbox{a.s.}
\end{equation}

Next we estimate
\[
\sup_{0\leq t\leq T}|\bbS_\ve^{ij}(t)-\bbU^\ve_{ij}(t)|^{2M}\leq 2^{2M-1}\ve^{4M}(I^{2M}_{\ve,1}+I^{2M}_{\ve,2}I^{2M}_{\ve,3})
\]
where
\begin{eqnarray*}
&I_{\ve,1}=\sup_{0\leq t\leq T}|\sum_{\min([\ve^{-2}t],n(t\bar\tau\ve^{-2}))
\leq k<l<\max([\ve^{-2}t],n(t\bar\tau\ve^{-2}))}\eta_i(k)\eta_j(l)|,\\
&I_{\ve,2}=|\sum_{0\leq k<\min([\ve^{-2}T],n(T\bar\tau\ve^{-2}))}\eta_i(k)|\,\,\mbox{and}\\
&I_{\ve,3}=\sup_{0\leq t\leq T}|\sum_{\min([\ve^{-2}t],n(t\bar\tau\ve^{-2}))
\leq l<\max([\ve^{-2}t],n(t\bar\tau\ve^{-2}))}\eta_j(l)|.
\end{eqnarray*}
Similarly to the above
\[
I^{2M}_{\ve,1}\leq\sum_{0\leq m\leq\ve^{-2}T}\sum_{0\leq n\leq\ve^{-3/2}}J^{2M}(m,n)+\sum_{0\leq m\leq\ve^{-2}T}
\sum_{0\leq n\leq L\hat LT\ve^{-2}}J_\ve^{2M}(m,n)
\]
where
\begin{eqnarray*}
&J(m,n)=\sum_{0\leq k<l<n}\eta_i(k+m)\eta_j(l+m)\,\,\mbox{and}\\
& J_\ve(m,n)=J(m,n)\bbI_{\sup_{0\leq t\leq T}|n(t\bar\tau\ve^{-2})-\ve^{-2}t|>\ve^{-3/2}}.
\end{eqnarray*}
By the stationarity of the sequence $\eta(r),\, r=0,1,...$,
\[
EJ^{2M}(m,n)=EJ^{2M}(0,n)\leq 2^{2M-1}(EJ^{2M}_1(n)+EJ_2^{2M}(n))
\]
where
\[
J_1(n)=J(0,n)-J_2(n)\,\,\mbox{and}\,\, J_2(n)=\sum_{0\leq k<l<n}E(\eta_i(k)\eta_j(l)).
\]
By Lemma \ref{lem3.4},
\[
EJ_1^{2M}(n)\leq C_{24}(M)n^{2M}
\]
for some $C_{24}(M)>0$ which does not depend on $n$.

In view of the assumptions of Theorem \ref{thm2.4} on the coefficients $\phi$ and $\rho$ the estimates (\ref{4.21})
and (\ref{4.22}) considered for the sequence $\eta(r),\, r\geq 0$ (in place of $\xi(r),\, r\geq 0$ there) hold true,
as well, which implies that
\[
|J_2(n)|\leq C_{25}n
\]
for some $C_{25}>0$ which does not depend on $n$. Now, by the Cauchy--Schwarz and Chebyshev inequalities, similarly to the above,
\begin{eqnarray*}
&EJ_\ve^{2M}(m,n)=E(J_1(n)+J_2(n))^{2M}\bbI_{\sup_{0\leq t\leq T}|n(t\bar\tau\ve^{-2})-\ve^{-2}t|>\ve^{-3/2}}\\
&\leq\big(E(J_1(n)+J_2(n))^{2(M+1)})^{\frac M{M+1}}\big(P\{\sup_{0\leq t\leq T}
|n(t\bar\tau\ve^{-2})-\ve^{-2}t|>\ve^{-3/2}\}\big)^{\frac 1{M+1}}\\
&\leq C_{26}(M)n^{2M}\ve^M
\end{eqnarray*}
for some $C_{26}(M)>0$ which does not depend on $n$ and $\ve$. Collecting the above inequalities we obtain that,
\[
\ve^{4M}EI^{2M}_{\ve,1}\leq C_{27}(M)\ve^{M-4}
\]
for some $ C_{27}(M)>0$ which does not depend on $\ve$.

Next, relying on Lemma \ref{lem3.4} considered for the sequence $\eta(r),\, r\geq 0$ and taking into account that
$n(T\bar\tau\ve^{-2})\leq T\hat L^2\ve^{-2}$ we obtain that
\[
EI^{2M}_{\ve,2}\leq C_{28}\ve^{-2M}
\]
for some $ C_{28}>0$ which does not depend on $\ve$. Since,
\[
I_{\ve,3}=\ve^{-1}\sup_{0\leq t\leq T}|S_\ve(t)-U^\ve(t)|
\]
we can use the estimates at the beginning of this subsection to obtain that
\[
\ve^{4M}E(I^{2M}_{\ve,2}I^{2M}_{\ve,3})\leq C_{29}\ve^{\frac M2-2}
\]
for some $ C_{29}>0$ which does not depend on $\ve>0$. Proceeding similarly to the above with the Chebyshev
inequality and the Borel--Cantelli lemma we arrive finally at
\begin{equation}\label{5.14++}
\max_{0\leq i,j\leq d}\sup_{0\leq t\leq T}|\bbS_\ve(t)-\bbU^\ve_{ij}(t)|=O(\ve^{1/8})\quad\mbox{a.s.}
\end{equation}

Next, we compare $V^\ve$ with $U^\ve$ and $\bbV^\ve$ with $\bbU^\ve$ with $V^\ve$ and $\bbV^\ve$ defined before
Theorem \ref{thm2.5}. Clearly, by (\ref{2.6}) and (\ref{2.19}),
\begin{equation}\label{5.15}
\sup_{0\leq t\leq T}|V^\ve(t)-U^\ve(t)|\leq\ve L\hat L.
\end{equation}
Next,
\begin{equation}\label{5.16}
\sup_{0\leq t\leq T}|\bbV_{ij}^\ve(t)-tEF_{ij}-\bbU_{ij}^\ve(t)|\leq\sup_{0\leq t\leq T}|J_1(t)|+\sup_{0\leq t\leq T}|J_2(t)|
+\sup_{0\leq t\leq T}|J_3(t)|
\end{equation}
where $F_{ij}(\om)=\int_0^{\tau(\om)}\xi_j(s,\om)ds\int_0^s\xi_i(u,\om)du$,
\[
J_1(t)=\ve^2\int_{\sig(t\bar\tau\ve^{-2})}^{t\bar\tau\ve^{-2}}\xi_j(s)\int_0^s\xi_i(u)duds,\,\sig(s)=\sig(s,\om)=
\sum_{j=0}^{n(s)-1}\tau\circ\vt^j(\om),
\]
\[
J_2(t)=\ve^2(\int_0^{\sig(t\bar\tau\ve^{-2})}\xi_j(s)\int_{\sig(s)}^s\xi_i(u)duds-n(t\bar\tau\ve^{-2})EF_{ij})
\]
\[
\mbox{and}\quad J_3(t)=(tEF_{ij}-\ve^2n(t\bar\tau\ve^{-2})EF_{ij}).
\]

Now, by (\ref{2.6}) and (\ref{2.19}),
\begin{eqnarray}\label{5.17}
&|J_1(t)|\leq L\ve^2\int_{\sig(t\bar\tau\ve^{-2})}^{t\bar\tau\ve^{-2}}ds|\int_0^s\xi_i(u)du|\\
&\leq L^2\hat L\ve^2+L\hat L\ve^2|\int_0^{t\bar\tau\ve^{-2}}\xi_i(u)du|\leq L^2\hat L\ve^2(1+\hat L)+L\hat L\ve^2
|\sum_{k=0}^{n(t\bar\tau\ve^{-2})}\eta_i(k)|.\nonumber
\end{eqnarray}
By (\ref{2.6}), (\ref{2.19}) and (\ref{5.13}) for any $\gam>0$,
\[
\ve^2\big\vert\sum_{k=0}^{n(t\bar\tau\ve^{-2})}\eta_i(k)-\sum_{k=0}^{[t\ve^{-2}]-1}\eta_i(k)\big\vert=O(\ve^{1-\gam})\,\,
\mbox{a.s.}
\]
It follows that
\begin{equation}\label{5.18}
\sup_{0\leq t\leq T}|J_1(t)|\leq C_{30}\ve^{1-\gam}+L\hat L\ve^2\max_{1\leq k\leq[T\ve^{-2}]}|\sum_{l=0}^{k-1}\eta_i(l)|
\end{equation}
for some a.s. finite random variable $C_{30}=C_{30}(\om)>0$ which does not depend on $\ve$.

Next, applying Lemma \ref{lem3.4} to the sequence $\eta(k),\, k\geq 0$ which is possible in view of (\ref{5.1}), we obtain
\[
E\max_{1\leq k\leq[T\ve^{-2}]}|\sum_{l=0}^{k-1}\eta_i(l)|^{2M}\leq\sum_{1\leq k\leq[T\ve^{-2}]}E|\sum_{l=0}^{k-1}\eta_i(l)|^{2M}
\leq C_1^\eta(M)T^M\ve^{-2M}
\]
where $C_1^\eta>0$ does not depend on $\ve$. Hence, by the Chebyshev inequality
\begin{equation}\label{5.19}
P\{\ve^2\max_{1\leq k\leq[T\ve^{-2}]}|\sum_{l=0}^{k-1}\eta_i(l)|>\ve^{1-\gam}\}\leq C_1^\eta(M)T^M\ve^{M\gam}.
\end{equation}
Taking $\ve=\ve_n=\frac 1{\sqrt n},\,\gam\in(0,1),\, M\geq 2/\gam$ and applying the Borel-Cantelli lemma we obtain that
\[
n^{-1}\max_{1\leq k\leq[Tn]}|\sum_{l=0}^{k-1}\eta_i(l)|=O(n^{-\frac 12(1-\gam)})\,\,\,\mbox{a.s.}
\]
If $\frac 1{\sqrt n}\leq\ve\leq\frac 1{\sqrt {n-1}}$ then
\begin{eqnarray*}
&\ve^2\max_{1\leq k\leq[T\ve^{-2}]}|\sum_{l=0}^{k-1}\eta_i(l)|\\
&\leq\frac 1{n-1}\max_{1\leq k\leq[Tn]}
|\sum_{l=0}^{k-1}\eta_i(l)|=O(n^{-\frac 12(1-\gam)})=O(\ve^{1-\gam})\,\,\mbox{a.s.},
\end{eqnarray*}
and so
\begin{equation}\label{5.20}
\sup_{0\leq t\leq T}|J_1(t)|=O(\ve^{1-\gam})\quad\mbox{a.s.}
\end{equation}

Next,
\begin{equation}\label{5.21}
|J_2(t)|=\ve^2|\sum_{k=0}^{n(t\bar\tau\ve^{-2})}(F_{ij}\circ\vt^k-EF_{ij})|\leq|J_4(t)|+2(L\hat L)^2\ve^2|n(t\bar\tau\ve^{-2})
-[t\ve^{-2}]|
\end{equation}
where
\[
J_4(t)=\ve^2\sum_{k=0}^{[t\ve^{-2}]}(F_{ij}\circ\vt^k-EF_{ij}).
\]
Now observe that by (\ref{2.6}), (\ref{2.19}) and (\ref{2.21}),
\begin{eqnarray}\label{5.22}
&|F_{ij}\circ\vt^k-E(F_{ij}\circ\vt^k|\cF_{k-n,k+n})|\\
&=|\int_0^{\tau\circ\vt^k(\om)}\xi_j(s,\vt\om)ds\int_0^s\xi_i(u,\vt^k\om)du\nonumber\\
&-E(\int_0^{\tau\circ\vt^k(\om)}\xi_j(s,\vt\om)ds\int_0^s
\xi_i(u,\vt^k\om)du|\cF_{k-n,k+n})|\nonumber\\
&\leq 2L\hat L|\tau\circ\vt^k(\om)-E(\tau\circ\vt^k|\cF_{k-n,k+n})(\om)|+G_{ij}(E(\tau\circ\vt^k|\cF_{k-n,k+n})(\om),\,\om)
\nonumber\end{eqnarray}
where
\begin{eqnarray*}
&G_{ij}(r,\om)=|\int_0^r\int_0^s\big (\xi_j(s,\vt^k\om)\xi_i(u,\vt^k\om)-E(\xi_j(s,\vt^k\om)\xi_i(u,\vt^k\om)|\cF_{k-n,k+n})\big)
dsdu|\\
&\leq 2Lr^2\rho(n).
\end{eqnarray*}
Hence, the left hand side of (\ref{5.22}) does not exceed $2L\hat L\rho(n)(1+\hat L)$, and so we can apply Lemma \ref{lem3.4}
with $F_{ij}\circ\vt^k$'s in place of $\xi(k)$'s to obtain that
\begin{eqnarray}\label{5.23}
&\,\,\,E\sup_{0\leq t\leq T}J_4^{2M}(t)\leq\ve^{4M}E\max_{1\leq n<[T\ve^{-2}]}|\sum_{k=0}^{n-1}(F_{ij}\circ\vt^k-EF_{ij})|^{2M}\\
&\leq\ve^{2M}C_1^F(M)T^M
\nonumber\end{eqnarray}
where $C_1^F>0$ does not depend on $\ve$. Arguing as for (\ref{5.20}) we see that for any $\gam>0$,
\[
\sup_{0\leq t\leq T}|J_4(t)|=O(\ve^{1-\gam})\quad\mbox{a.s.}
\]
which together with (\ref{5.13}) and (\ref{5.21}) gives that for any $\gam>0$,
\begin{equation}\label{5.24}
\sup_{0\leq t\leq T}|J_2(t)|=O(\ve^{1-\gam})\quad\mbox{a.s.}
\end{equation}

Estimating $J_3$ by (\ref{5.13}) we obtain that for any $\gam>0$,
\begin{equation}\label{5.25}
\sup_{0\leq t\leq T}|J_3(t)|\leq |EF_{ij}|\sup_{0\leq t\leq T}|t-\ve^2n(t\bar\tau\ve^{-2})|=O(\ve^{1-\gam})\quad\mbox{a.s.}
\end{equation}
Finally, collecting (\ref{5.8}), (\ref{5.9}), (\ref{5.14+}), (\ref{5.14++}), (\ref{5.15}), (\ref{5.18}), (\ref{5.24})
 and (\ref{5.25}) we obtain (\ref{2.25}) and (\ref{2.26}) but only for the supremum norm.

\subsection{Completing the proof of Theorem \ref{thm2.5}}\label{subsec5.4}

For $0\leq s<t$ set
\[
V^\ve(s,t)=V^\ve(t)-V^\ve(s)=\ve\int_{s\bar\tau\ve^{-2}}^{t\bar\tau\ve^{-2}}\xi(u)du,
\]
\begin{eqnarray*}
&U^\ve(s,t)=U^\ve(t)-U^\ve(s)=\ve\sum_{n(s\bar\tau\ve^{-2})\leq k<n(t\bar\tau\ve^{-2})}\eta(k)\\
&\mbox{and}\,\,\,\hat U(s,t)=\ve\sum_{[s\ve^{-2}]\leq k<[t\ve^{-2}]}\eta(k).
\end{eqnarray*}
First, observe that by (\ref{2.6}) for any $\be\in(0,1/2)$,
\begin{equation}\label{5.26}
|V^\ve(s,t)|\leq\ve^{-1}L\bar\tau(t-s)\leq L\bar\tau(t-s)^{\frac 12-\be}\ve^{-1}(t-s)^{\frac 12+\be}=O((t-s)^{\frac 12-\be})
\end{equation}
provided
\begin{equation}\label{5.27}
(t-s)=O(\ve^{\frac 2{1+2\be}}).
\end{equation}

Next, we are going to obtain H\" older type uniform estimates of $|V^\ve(s,t)|(t-s)^{(\frac 12-\be)}$ similar to (\ref{5.26})
for small $t-s$ satisfying
\begin{equation}\label{5.28}
s+\ve^{1-\al}\geq t\geq s+\ve^2>s\geq 0
\end{equation}
where $\al\in(0,1)$ is close to 1 and it is chosen similarly to Section \ref{subsec3.4}. Observe that by (\ref{2.6}) and (\ref{2.19}),
\[
|V^\ve(s,t)-V^\ve([s\ve^{-2}]\ve^2,\,[t\ve^{-2}]\ve^2)|\leq 2\ve L\hat L,
\]
and so under (\ref{5.28}),
\[
\frac {|V^\ve(s,t)|}{(t-s)^{\frac 12-\be}}\leq\sqrt 2\frac {|V^\ve([s\ve^{-2}]\ve^2,\,[t\ve^{-2}]\ve^2)|}{(
[t\ve^{-2}]-[s\ve^{-2}])^{\frac 12-\be}\ve^{1-2\be}}+2\ve^{2\be}L\hat L.
\]
Hence,
\begin{eqnarray}\label{5.29}
&\sup_{0\leq s<s+\ve^2\leq t\leq s+\ve^{1-\al},\, t\leq T}\frac {|V^\ve(s,t)|}{(t-s)^{\frac 12-\be}}\\
&\leq\sqrt 2\max_{0\leq k<l\leq k+\ve^{-(1+\al)},\, l\leq T\ve^{-2}}\frac {|V^\ve(k\ve^2,\,l\ve^2)|}
{(l-k)^{\frac 12-\be}\ve^{1-2\be}}+2\ve^{2\be}L\hat L.\nonumber
\end{eqnarray}

In order to estimate the moments of the right hand side of (\ref{5.29}) we introduce
\[
\hat V^\ve_{kl}(\om,u)=\ve\int_{k\bar\tau}^{l\bar\tau}\xi(v,(\om,u))dv=\ve\int_{k\bar\tau+u}^{l\bar\tau+u}\xi(v,(\om,0))dv
\]
and observe that by (\ref{2.6}) and (\ref{2.19}),
\begin{equation}\label{5.30}
|V^\ve(k\ve^2,l\ve^2)-\hat V^\ve_{kl}(\om,u)|\leq 2uL\ve\leq 2L\hat L\ve.
\end{equation}
Since $\xi$ is a stationary process on the probability space $(\hat\Om,\hat\cF,\hat P)$ we see that
\[
\int|\hat V^\ve_{kl}(\om,u)|^{2M}d\hat P(\om,u)=\int|\hat V^\ve_{0l-k}(\om,u)|^{2M}d\hat P(\om,u),
\]
and so by (\ref{5.30}),
\begin{eqnarray}\label{5.31}
&E|V^\ve(k\ve^2,l\ve^2)|^{2M}\leq 2^{4M-2}\hat L^2E|V^\ve(0,(l-k)\ve^2)|^{2M}\\
&+2^{4M-1}(2^{2M}+1)\hat L(L\hat L)^{2M}\ve^{2M}.\nonumber
\end{eqnarray}
By (\ref{5.15}),
\[
|V^\ve(0,(l-k)\ve^2)-U^\ve(0,(l-k)\ve^2)|\leq 2\ve L\hat L,
\]
and so
\begin{equation}\label{5.32}
E|V^\ve(0,(l-k)\ve^2)|^{2M}\leq 2^{2M-1}E|U^\ve(0,(l-k)\ve^2)|^{2M}+2^{4M-1}(L\hat L)^{2M}\ve^{2M}.
\end{equation}

In order to estimate the right hand side of (\ref{5.32}) we observe first that
\[
|U^\ve(0,(l-k)\ve^2)-\hat U^\ve(0,(l-k)\ve^2)|\leq\ve L\hat L|n((l-k)\bar\tau)-(l-k)|,
\]
and so
\begin{eqnarray}\label{5.33}
&E|U^\ve(0,(l-k)\ve^2)|^{2M}\leq 2^{2M-1}E|\hat U^\ve(0,(l-k)\ve^2)|^{2M}\\
&+2^{2M-1}(L\hat L)^{2M}\ve^{2M}E|n((l-k)\bar\tau)-(l-k)|^{2M}.\nonumber
\end{eqnarray}
By Lemma \ref{lem3.4} applied to sums of $\eta(k)$'s we obtain that
\begin{equation}\label{5.34}
E|\hat U^\ve(0,(l-k)\ve^2)|^{2M}\leq C_1^\eta(M)(l-k)^M\ve^{2M}
\end{equation}
where $C_1^\eta>0$ does not depend on $\ve, l$ and $k$. By (\ref{5.12}) and (\ref{5.31})--(\ref{5.34}),
\begin{equation}\label{5.35}
E|V^\ve(k\ve^2,l\ve^2)|^{2M}\leq C_{31}\ve^{2M}(l-k)^M
\end{equation}
for some $C_{31}>0$ which does not depend on $\ve,k$ and $l$. Hence, by the Chebyshev inequality
\begin{eqnarray}\label{5.36}
&P\big\{ \max_{0\leq k<l\leq k+\ve^{-(1+\al)},\, l\leq T\ve^{-2}}\frac {|V^\ve(k\ve^2,\,l\ve^2)|}
{(l-k)^{\frac 12-\be}\ve^{1-2\be}}>1\big\}\\
&\leq\sum_{0\leq k<l\leq k+\ve^{-(1+\al)},\, l\leq T\ve^{-2}}P\big\{\frac {|V^\ve(k\ve^2,\,l\ve^2)|}
{(l-k)^{\frac 12-\be}\ve^{1-2\be}}>1\big\}\nonumber\\
&\leq\sum_{0\leq k<l\leq k+\ve^{-(1+\al)},\, l\leq T\ve^{-2}}\ve^{-2M(1-2\be)}(l-k)^{-M(1-2\be)}E|V^\ve(k\ve^2,l\ve^2)|^{2M}
\nonumber\\
&\leq C_{31}\ve^{4\be M}\sum_{0\leq k<l\leq k+\ve^{-(1+\al)},\, l\leq T\ve^{-2}}(l-k)^{2\be M}\leq C_{31}T
\ve^{2\be M(1-\al)-4}.\nonumber
\end{eqnarray}

Now, take $\ve=\ve_N=\frac 1{\sqrt N}$ and $M\geq 4\be^{-1}(1-\al)^{-1}$, then the right hand side of (\ref{5.36}) is a term
of a converging sequence, and so by the Borel--Cantelli lemma there exists an a.s. finite random variable $C_{32}>0$ which
does not depend on $\ve,k$ and $l$ and such that
\[
|V^\ve(k\ve^2,l\ve^2)|\leq C_{32}(l\ve^2-k\ve^2)^{\frac 12-\be}
\]
whenever $\ve^{-(1+\al)}\geq l-k\geq 1$. Since for any $\frac 1{\sqrt {N+1}}\leq\ve\leq\frac 1{\sqrt N}$, we have
$N+1\geq\ve^{-2}\geq N$, and so
\[
\sup_{0\leq s<t\leq T}|V^\ve(s,t)-V^{\ve_N}(s,t)|\leq 2T\hat L\ve,
\]
we conclude from here and from (\ref{5.29}) that
\begin{equation}\label{5.37}
|V^\ve(s,t)|\leq C_{\al,\be}|t-s|^{\frac 12-\be}
\end{equation}
whenever (\ref{5.28}) holds true, where $C_{\al,\be}>0$ is a a.s. finite random variable which does not depend on
$s,t$ and $\ve$. Taking into account (\ref{5.26}) and (\ref{5.27}) we conclude that (\ref{5.37}) holds true just under
the condition $|t-s|\leq\ve^{1-\al}$.

Now we can complete the proof of (\ref{2.25}) from Theorem \ref{thm2.5}. Set $N_\ve=[\ve^{-2}]$ and define
$\hat W^\ve(t)=W^\ve(N_\ve^{-1}[N_\ve t])$. Let $0=t_0<t_1<...<t_m=T$.
Next, we proceed similarly to (\ref{3.29})--(\ref{3.35}) writing
\[
\sum_{0\leq i<m}|V^\ve(t_i,t_{i+1})-\hat W^\ve(t_i,t_{i+1})|^p\leq J^\ve_1+2^{p-1}(J_2^\ve+J_3^\ve)
\]
where for $\al\in(0,1)$,
\[
J_1^\ve=\sum_{0\leq i<m,\, t_{i+1}-t_i>N_\ve^{-(1-\al)}}|V^\ve(t_i,t_{i+1})-\hat W^\ve(t_i,t_{i+1})|^p,
\]
\[
J_2^\ve=\sum_{0\leq i<m,\, t_{i+1}-t_i\leq N_\ve^{-(1-\al)}}|V^\ve(t_i,t_{i+1})|^p\quad\mbox{and}
\]
\[
J_3^\ve=\sum_{0\leq i<m,\, t_{i+1}-t_i\leq N_\ve^{-(1-\al)}}|\hat W^\ve(t_i,t_{i+1})|^p.
\]

Taking into account that
\[
|V^\ve(t_i,t_{i+1})-V^\ve(N^{-1}_\ve[N_\ve t_i],N^{-1}_\ve[N_\ve t_{i+1}])|\leq 2\ve(1-\ve^2)^{-1}L\hat L,
\]
that there exists no more than $[TN_\ve^{1-\al}]$ intervals $[t_i,t_{i+1}]$ with $t_{i+1}-t_i>N_\ve^{-(1-\al)}$ and
relying on the supremum norm estimate in (\ref{2.25}) we obtain that
\begin{eqnarray*}
&J_1^\ve\leq 2^{2p-1}T(L\hat L)^p\ve^{p-2(1-\al)}(1-\ve^2)^{-p}\\
&+2^{p-1}\sum_{0\leq j<n,k_{j+1}-k_j>N_\ve^\al}\sum_{0\leq j<n,k_{j+1}-k_j>N^\al_\ve}|V^\ve(\frac {k_j}{N_\ve},
\frac {k_{j+1}}{N_\ve})-\hat W^\ve(\frac {k_j}{N_\ve},\frac {k_{j+1}}{N_\ve})|^p\\
&\leq 2^{2p-1}\ve^{-2(1-\al)}((L\hat L)^p\ve^p(1-\ve^2)^{-p}+O(\ve^{p\del}))
\end{eqnarray*}
where $k_j=[t_{i_j}N_\ve]$ and $0=t_{i_0}<t_{i_1}<...<t_{i_n}=T$ is the maximal subsequence of $t_0,t_1,...,t_m$ with
$[t_{i_{j+1}}N_\ve]>[t_{i_j}N_\ve]$.

Next, using (\ref{5.37}) we obtain similarly to (\ref{3.31}) that
\[
J^\ve_2\leq C^p_{\al,\be}T\ve^{2(1-\al)(\frac p2-1-p\be)}(1-\ve^2)^{-1}
\]
while proceeding similarly to (\ref{3.33})--(\ref{3.35}) we obtain also
\[
J^\ve_3\leq\tilde C^p_{\al,\be}T\ve^{2(1-\al)(\frac p2-1-p\be)}(1-\ve^2)^{-1}
\]
where $C_{\al,\be},\tilde C_{\al,\be}>0$ are a.s. finite random variables which do not depend on $\ve>0$ or the choice of
$t_1,...,t_m$. Taking $\al$ so close to 1 that $p\del>2(1-\al)$ and choosing $\be$ satisfying (\ref{3.32}) we obtain that
\[
\| V^\ve-\hat W^\ve\|_{p,[0,T]}=O(\ve^{\tilde\del})\quad\mbox{a.s.}
\]
for some $\tilde\del>0$ which does not depend on $\ve>0$. This together with Lemma \ref{lem3.10} completes
the proof of (\ref{2.25}) from Theorem \ref{thm2.5}.

It remains to complete the proof of (\ref{2.26}). For $0\leq s<t$ set
\begin{eqnarray*}
&\bbV^\ve_{ij}(s,t)=\ve^2\int_{s\bar\tau\ve^{-2}}^{t\bar\tau\ve^{-2}}\xi_j(v)dv\int_{s\bar\tau\ve^{-2}}^v\xi_i(u)du,\\
&\bbU^\ve_{ij}(s,t)=\ve^2\sum_{n(s\bar\tau\ve^{-2})\leq k<l<n(t\bar\tau\ve^{-2})}\eta_i(k)\eta_j(l),\\
&\hat\bbU^\ve_{ij}(s,t)=\ve^2\sum_{[s\ve^{-2}]\leq k<l<[t\ve^{-2}]}\eta_i(k)\eta_j(l)\quad\mbox{and}\\
&\hat\bbW_{ij}^\ve(s,t)=\sum_{[N_\ve s]\leq l<[N_\ve t]}(W^{\ve}_j(\frac {l+1}{N_\ve})-W^{\ve}_j(\frac l{N_\ve}))
(W_i^\ve(\frac l{N_\ve})-W^\ve_i(\frac {[N_\ve s]}{N_\ve}))
\end{eqnarray*}
while $\hat\bbW^\ve_{ij}(0,t)$ equals $\hat\bbW^\ve_{ij}(t)$ defined before Lemma \ref{lem3.10}.
By (\ref{2.6}) for any $\be\in(0,1)$,
\begin{equation}\label{5.38}
|\bbV^\ve_{ij}(s,t)|\leq\ve^{-2}L^2\bar\tau^2(t-s)^2\leq L^2\bar\tau^2(t-s)^{1-\be}\ve^{-2}(t-s)^{1+\be}=O((t-s)^{1-\be})
\end{equation}
provided
\begin{equation}\label{5.39}
(t-s)=O(\ve^{\frac 2{1+\be}}).
\end{equation}

Next, we are going to obtain H\" older type estimates of $|\bbV^\ve_{ij}(s,t)|(t-s)^{-(1-\be)}$ similar to (\ref{5.38}) for
$t-s$ satisfying (\ref{5.28}). Observe that by (\ref{2.6}) and (\ref{2.19}),
\begin{eqnarray*}
&|\bbV^\ve_{ij}(s,t)-\bbV^\ve_{ij}([s\ve^{-2}]\ve^2,\, [t\ve^{-2}]\ve^2 )|\leq\ve^2(L\hat L)^2+\ve^2L\int_{[t\ve^{-2}]\bar\tau}
^{t\ve^{-2}\bar\tau}dv|\int^v_{s\ve^{-2}\bar\tau}\xi_i(u)du|\\
&+\ve^2|\int_{[s\ve^{-2}]\bar\tau}^{[t\ve^{-2}]\bar\tau}\xi_j(v)dv\int^{s\ve^{-2}\bar\tau}_{[s\ve^{-2}]\bar\tau}\xi_i(u)du|
\nonumber\\
&\leq 4\ve^2(L\hat L)^2+\ve L\hat L(|V_i^\ve(s,t)|+|V_j^\ve(s,t)|),\nonumber
\end{eqnarray*}
and so under (\ref{5.28}),
\[
\frac {|\bbV^\ve_{ij}(s,t)|}{(t-s)^{1-\be}}\leq 2\frac {|\bbV^\ve_{ij}([s\ve^{-2}]\ve^2,\, [t\ve^{-2}]\ve^2)|}
{([t\ve^{-2}]-[s\ve^{-2}])^{1-\be}\ve^{2(1-\be)}}+4(L\hat L)^2\ve^{2\be}+L\hat L\ve^\be\frac {|V_i^\ve(s,t)|+|V_j^\ve(s,t)|}
{(t-s)^{\frac 12(1-\be)}}.
\]
Hence,
\begin{eqnarray}\label{5.40}
&\sup_{0\leq s<s+\ve^2\leq t\leq s+\ve^{1-\al},\, t\leq T}\frac {|\bbV^\ve_{ij}(s,t)|}{(t-s)^{1-\be}}\\
&\leq 2\max_{0\leq k<l\leq k+\ve^{-(1+\al)},\, l\leq T\ve^{-2}}\frac {|\bbV_{ij}^\ve(k\ve^2,\,l\ve^2)|}
{(l-k)^{1-\be}\ve^{2(1-\be)}}+4\ve^{2\be}L\hat L\nonumber\\
&+L\hat L\ve^\be\sup_{0\leq s<s+\ve^2\leq t\leq s+\ve^{1-\al},\, t\leq T} \frac {|V_i^\ve(s,t)|+|V_j^\ve(s,t)|}
{(t-s)^{\frac 12(1-\be)}}.\nonumber
\end{eqnarray}

In order to estimate the moments of the right hand side of (\ref{5.40}) we introduce
\[
\hat\bbV_{ij}^\ve(k,l)(\om,u)=\ve^2\int_{k\bar\tau}^{l\bar\tau}\xi(v,(\om,u))dv\int_{k\bar\tau}^v\xi(w,(\om,u))dw
\]
and observe that by (\ref{2.6}) and (\ref{2.19}) similarly to the above,
\begin{equation}\label{5.41}
|\bbV^\ve_{ij}(k\ve^2,l\ve^2)-\hat\bbV_{ij}^\ve(k,l)(\om,u)|\leq 4\ve^2(L\hat L)^2+\ve L\hat L(|
V^\ve_i(k\ve^2,l\ve^2)|+|V^\ve_j(k\ve^2,l\ve^2)|).
\end{equation}
Again, we use the stationarity of the process $\xi$ on the probability space $(\hat\Om,\hat\cF,\hat P)$ to obtain that
\[
\int|\hat\bbV_{ij}^\ve(k,l)(\om,u)|^{2M}d\hat P(\om,u)=\int|\hat\bbV_{ij}^\ve(0,l-k)(\om,u)|^{2M}d\hat P(\om,u),
\]
and so by (\ref{5.41}),
\begin{eqnarray}\label{5.42}
&E|\bbV_{ij}^\ve(k\ve^2,l\ve^2)|^{2M}\leq 2^{4M-2}\hat L^2E|\bbV_{ij}^\ve(0,(l-k)\ve^2)|^{2M}\\
&+2^{6M-1}(2^{2M}+1)\hat L(L\hat L)^{2M}\ve^{4M}\nonumber\\
&+2^{4M-2}\ve^{2M}(L\hat L)^{2M}(E|V^\ve_i(k\ve^2,l\ve^2)|^{2M}+|V^\ve_j(k\ve^2,l\ve^2)|^{2M}).\nonumber
\end{eqnarray}

Next, by (\ref{5.16}), (\ref{5.17}), (\ref{5.21}) and (\ref{5.25}),
\begin{eqnarray}\label{5.43}
&|\bbV_{ij}^\ve(0,(l-k)\ve^2)-\ve^2(l-k)EF_{ij}-\bbU_{ij}^\ve(0,(l-k)\ve^2)|\\
&\leq (L\hat L)^2\ve^2(l-k)+L^2\hat L\ve^2(2+\hat L)\nonumber\\
&+4(L\hat L)^2\ve^2|n((l-k)\bar\tau)-(l-k)|+\ve^2|\sum_{m=0}^{l-k-1}(F_{ij}\circ\vt^m-EF_{ij})|\nonumber
\end{eqnarray}
where we took into account that $|F_{ij}|\leq (L\hat L)^2$. Also, we obtain easily by (\ref{2.6}) that
\begin{eqnarray}\label{5.44}
&|\bbU_{ij}^\ve(0,(l-k)\ve^2)-\hat\bbU_{ij}^\ve(0,(l-k)\ve^2)|\\
&\leq\ve^2 L^2|n((l-k)\bar\tau)-(l-k)|^2+\ve^2L|n((l-k)\bar\tau)-(l-k)||\hat U^\ve(0,(l-k)\ve^2)|.\nonumber
\end{eqnarray}

By Lemma \ref{lem3.4},
\[
E|\hat\bbU_{ij}^\ve(0,(l-k)\ve^2)|^{2M}\leq C_{33}\ve^{4M}(l-k)^{2M}
\]
where $C_{33}>0$ does not depend on $l,k$ or $\ve$. Combining this with
(\ref{5.12}), (\ref{5.23}), (\ref{5.34}), (\ref{5.36}) and (\ref{5.42})--(\ref{5.44}) we obtain that
\begin{equation}\label{5.45}
E|\bbV_{ij}^\ve(k\ve^2,l\ve^2)|^{2M}\leq C_{34}\ve^{4M}(l-k)^{2M}
\end{equation}
for some $C_{34}>0$ which does not depend on $\ve,k$ or $l$. Hence, by the Chebyshev inequality
\begin{eqnarray}\label{5.46}
&P\big\{ \max_{0\leq k<l\leq k+\ve^{-(1+\al)},\, l\leq T\ve^{-2}}\frac {|\bbV_{ij}^\ve(k\ve^2,\,l\ve^2)|}
{(l-k)^{1-\be}\ve^{2(1-\be)}}>1\big\}\\
&\leq\sum_{0\leq k<l\leq k+\ve^{-(1+\al)},\, l\leq T\ve^{-2}}P\big\{\frac {|\bbV_{ij}^\ve(k\ve^2,\,l\ve^2)|}
{(l-k)^{1-\be}\ve^{2(1-\be)}}>1\big\}\nonumber\\
&\leq C_{34}\ve^{4M}\sum_{0\leq k<l\leq k+\ve^{-(1+\al)},\, l\leq T\ve^{-2}}(l-k)^{2\be M}\leq C_{34}
\ve^{2M\be(1-\al)-4}\nonumber
\end{eqnarray}
where $\al\in(0,1)$ is close to 1 and it is chosen similarly to Section \ref{subsec4.2}. Taking $\ve=\ve_N=\frac 1{\sqrt N}$
 and $M\geq 3\be^{-1}(1-\al)^{-1}$, using (\ref{5.38}) when (\ref{5.39}) is
satisfied, relying on the Borel--Cantelli lemma and arguing as for (\ref{5.37}) we obtain that
\begin{equation}\label{5.47}
|\bbV_{ij}^\ve(s,t)|\leq C_{\al,\be}|t-s|^{1-\be}
\end{equation}
whenever $|t-s|\leq\ve^{1-\al}$ holds true, where $C_{\al,\be}>0$ is another a.s. finite random variable which does not
depend on $s,t$ and $\ve$.

Let $0=t_0<t_1<...<t_m=T$ and write
\[
\sum_{0\leq q<m}|\bbV_{ij}^\ve(t_q,t_{q+1})-\hat\bbW_{ij}^\ve(t_q,t_{q+1})|^{p/2}\leq \cJ^\ve_1+2^{\frac p2-1}
(\cJ_2^\ve+\cJ_3^\ve)
\]
where for $\al\in(0,1)$,
\[
\cJ_1^\ve=\sum_{0\leq q<m,\, t_{i+1}-t_i>N_\ve^{-(1-\al)}}|\bbV_{ij}^\ve(t_q,t_{q+1})-\hat\bbW_{ij}^\ve(t_q,t_{q+1})|^{p/2},
\]
\[
\cJ_2^\ve=\sum_{0\leq q<m,\, t_{q+1}-t_q\leq N_\ve^{-(1-\al)}}|\bbV_{ij}^\ve(t_q,t_{q+1})|^{p/2}\quad\mbox{and}
\]
\[
J_3^\ve=\sum_{0\leq q<m,\, t_{q+1}-t_q\leq N_\ve^{-(1-\al)}}|\hat\bbW_{ij}^\ve(t_q,t_{q+1})|^{p/2}.
\]
Observe that
\[
\bbV^\ve_{ij}(s,t)=\bbV^\ve_{ij}(t)-\bbV^\ve_{ij}(s)-V^\ve_i(s)(V^\ve_j(t)-V^\ve_j(s))
\]
and
\[
\hat\bbW^\ve_{ij}(s,t)=\hat\bbW_j(t)-\hat\bbW_{ij}^\ve(s)-\hat W^\ve_i(s)(\hat W_j^\ve(t)-\hat W_j^\ve(s)).
\]
We estimate $\cJ^\ve_1$ taking into account Lemma \ref{lem3.10}, the supremum norm estimates in (\ref{2.25}) and
(\ref{2.26}), the fact that there exist no more than $[TN_\ve^{1-\al}]$ intervals $[t_q,t_{q+1}]$ with length
exceeding $N_\ve^{-(1-\al)}$ and then proceed similarly to (\ref{4.40})--(\ref{4.42}) obtaining that
\[
\cJ_1^\ve=O(\ve^{p\del+2(\al-1)})
\]
where $1-\al>0$ can be taken arbitrarily small and $\del>0$ comes from Lemma \ref{lem3.10} and the supremum
norm estimates of Theorem
\ref{thm2.5}. Next, $\cJ_2^\ve$ is estimated using (\ref{5.47}) taking into account that $\sum_{0\leq q<m}|t_{q+1}-t_q|=T$
and arguing similarly to (\ref{3.31}). Finally, $\cJ^\ve_3$ is estimated similarly to (\ref{3.33})--(\ref{3.35}) and
(\ref{4.45})--(\ref{4.47}). This completes the proof of (\ref{2.26}).
 Now, Theorem \ref{thm2.4} follows from Theorem \ref{thm2.5} and the rough paths theory arguments given in Section \ref{ss:ProofThm24} below.

\section{Rough paths and diffusion approximation}\label{roughsec6}\setcounter{equation}{0}


%
%
%


\def \x{U}
\def \bx{\mathbf{U}}
\def \mbx{\mathbb{U}}
\def \bB{\mathbf{B}}
\def \mbB{\mathbb{B}}
\def \tx{\tilde U}
\def \driftA{A}
\def \driftB{\tilde A}
\def \tbx{\tilde{\mathbf{U}}}
\def \tmbx{\tilde{\mathbb{U}}}
\def \mvar{}


We start this section with a review of some elements of rough path theory, pointing whenever possible to \cite{FH}. Most results therein are formulated in H\"older spaces, the extension
to c\`adl\`ag $p$-variation spaces is found in \cite{FS, FZ}.

\subsection{Review and local Lipschitz continuity of It\^o-Lyons map}
\subsubsection{Rough paths} \label{sec:RP}
\def \dim{d}

Consider a c\`adl\`ag path $U: [0,T] \to \R^e$ of finite $p$-variation on $[0,T]$, so that
\begin{equation} \label{equ:pvarnorm}
                   \| \x \|_{p\mvar,[0,T]} := \left( \sup_{\op} \sum_{[s,t]\in \op} |  \x (s,t)|^{p}   \right)^{\frac 1 p} < \infty \
\end{equation}
with path increments $ \x (s,t)  := \x (t) - \x (s) \in \R^e$, additive in the sense that $ \x (s,t) +  \x (t,u) =  \x (s,u)$.
 For $p \in [1,\infty)$ this defines a seminorm (it does not separate constants).
For $p = 1$, iterated Riemann--Stieltjes (RS) integration defines second order increments $ \mbx(s,t) \in \R^e \otimes \R^e$, c\`adl\`ag in both variables,
\begin{equation} \label{equ:secondLevel}
          \mbx^{ij}(s,t) := \int_{(s,t]} ( \x^i (r-) - \x^i (s) )  d \x^j (r), \qquad 1 \le i,j \le e.
\end{equation}
Such increments are non-additive; elementary (additivity) properties of the integral gives {\em Chen's relation} (cf. \cite[Ch. 2]{FH})
\begin{equation} \label{equ:Chen}
     \mbx(s,t)  +   \x(s,t) \otimes  \x(t,u) + \mbx(t,u) =   \mbx(s,u), \quad 0 \le s \le t \le u \le T,
\end{equation}
with tensor notation that relieves us from spelling out coordinates.
Thanks to classical works of Young (cf. \cite[Ch. 4]{FH}) this extends to $p \in [1,2)$, such that\footnote{The spaces  $\R^e, \R^e \otimes \R^e$ are equipped with compatible norms, all denoted by $ | \cdot |$, compatible in the sense that $| v \otimes w| \le |v| |w|$ for all $v,w \in \R^e$.}
\begin{equation} \label{equ:pvarnorm2}
          \| \mbx \|_{(p/2)\mvar,[0,T]} = \left( \sup_{\op} \sum_{[s,t]\in \op} | \mbx (s,t) |^{\frac p 2}   \right)^{\frac 2 p}< \infty .
\end{equation}
Let now $p \in [2,3)$. There is no more (Riemann--Stieltjes or Young) meaning to \eqref{equ:secondLevel}, 
instead we consider $\mbx$ as part of what we mean by a path:
by definition, a (level-$2$, c\`adl\`ag) {\it $p$-rough path} (over $\R^e$, on $[0,T]$) is a pair $\bx = (\x, \mbx)$, c\`adl\`ag, where one imposes the algebraic Chen relation \eqref{equ:Chen} and the analytic regularity conditions  \eqref{equ:pvarnorm},  \eqref{equ:pvarnorm2}, so that
\begin{equation} \label{equ:doublenorm}
                   \| \bx \|_{p\mvar,[0,T]} := \| \x \|_{p\mvar,[0,T]} +   \| \mbx \|_{(p/2)\mvar,[0,T]} < \infty \ .
\end{equation}
If $U(0)$ is fixed, or upon identifying paths with identical increments, one can equivalently regard $\bx$ as $2$-parameter (c\`adl\`ag) function $(s,t) \mapsto \bx (s,t) = ( \x (s,t), \mbx(s,t)) \in \R^e \oplus (\R^e \otimes \R^e) =: G$, a (Lie)group equipped with multiplication $(a,M) \star (b,N) = (a + b, M + a \otimes b + N)$, inverse $(a,M)^{-1}:= (-a, -M + a\otimes a),$ and identity $(0,0)$. Addivity of $ \x$ and Chen's relation then take the appealing form
\begin{equation} \label{equ:Chenstar}
\bx (s,t) \star \bx (t,u) = \bx (s,u).
\end{equation}
From $\bx (s,t) = \bx (0,s)^{-1} \star \bx (0,t)$ we see that $t \mapsto \bx (0,t)$ contains all information which suggests an essentially equivalent definition of rough path as genuine $G$-valued c\`adl\`ag path $t \mapsto \bx (t)$, with induced group increments  $\bx (s,t) = ( \x (s,t), \mbx(s,t)) = \bx (s) ^{-1}\star \bx (t)$, subject only to the regularity condition \eqref{equ:doublenorm}.

\medskip

In many cases $\mbx$ arises from some sort of (possibly stochastic) integration of some path (or process) $U$ against itself, and hence scales like $\lambda^2$ upon replacing $U$ by $\lambda U$. This suggests a purely analytic {\em dilation} of rough paths, with pointwise definition
\begin{equation} \label{equ:dilation}
\delta_\lambda \bx(s,t) :=  (\lambda \x(s,t), \lambda^2 \mbx(s,t)).
\end{equation}  The {\em homogenous rough path norm}
\begin{equation} \label{equ:tripleenorm}
                   ||| \bx |||_{p\mvar,[0,T]} := \| \x\|_{p\mvar,[0,T]} +   \| \mbx \|^{1/2}_{(p/2)\mvar,[0,T]} < \infty \ .
\end{equation}
then has the desirable property $||| \delta_\lambda \bx |||_{p\mvar,[0,T]}  =  \lambda ||| \bx |||_{p\mvar,[0,T]}, \lambda \ge 0,$ and is often preferable to its non-homogenous counterpart \eqref{equ:doublenorm}. The latter however gives rise to the {\it (inhomogenous) $p$-rough path distance}\footnote{The notation $\| \bx ; \tbx \|_{p\mvar,[0,T]}$ is a gentle reminder of the non-linear nature of rough path spaces, as is evident from Chen's relation.}
\begin{equation} \label{equ:doublenormdist}
                \| \bx ; \tbx \|_{p\mvar,[0,T]} :=  \| \bx - \tbx \|_{p\mvar,[0,T]}  :=   \| \x- \tx \|_{p\mvar,[0,T]} +   \| \mbx -\tmbx \|_{(p/2)\mvar,[0,T]},  \
\end{equation}
with respect to which the It\^o-Lyons map turns out locally Lipschitz continuous. (At occasions, its homogenous counterpart can also be useful.)



\def \w{\mathcal{W}}
\def \bw{\mathbf{W}}
\def \mbw{\mathbb{W}}

\def \Ito{\mathrm{It\hat{o}}}
\def \Strato{\mathrm{Strato}}

\subsubsection{Semimartingales as rough paths} \label{sec:SemiRP} The main motivation for this construction comes from stochastic analysis. Indeed, if $\x = \x (t,\omega)$ is c\`adl\`ag semimartingale, on $\R^e$, then a.s. its It\^o lift
$\bx^\Ito (t;\omega) = (\x (t;\omega) , \mbx^\Ito (0,t;\omega))$, with increments $\x (s,t;\omega) = \x (t;\omega) - \x (s;\omega) \in \R^e$ and
\begin{equation} \label{equ:Ito2ndlevel}
          \mbx^{ij;\Ito}(s,t;\omega) = \left( \int_{(s,t]} ( \x^i (r-) - \x^i (s) )  d \x^j (r) \right) (\omega), \qquad 1 \le i,j \le e,
\end{equation}
 has the correct $p$ and $p/2$ variation regularity, any $p>2$, and hence constitutes a $p$-rough path, for any $p \in (2,3)$, over $\R^e$ and on compact time horizon $[0,T]$. The afore-mentioned $p$-variation regularity of a semimartingale is classical (see \cite[Thm. 1]{Lep}; the argument relies on representing a c\`ad\`ag martingale as time-changed Brownian motion).   
In case $e =1$, this already gives (via It\^o's formula) the $p/2$ variation regularity of $\mbx^\Ito$; for the general case of multidimensional c\`adl\`ag semimartingales see \cite{CF, FZK}, for the continuous case see \cite[Ch.14]{FV} and references therein.

\subsection{Rough differential equations}\label{subsec6.2}

\def \drift {b}
\def \vol {\sigma}

\def \ext{\mathrm{ext}}
\def \y{Y}
\def \ym {Y^-}
\def \yp{Y'}

\def \ty{\tilde Y}
\def \tym {\tilde Y^-}
\def \typ{\tilde Y'}

\def \dimm{e}


Let $\drift$ and $\vol_1, ..., \vol_e$ be vector fields on $\R^d$,
sufficiently smooth for all derivatives below to exist. As is common in this context, we regard $(\vol_1, ... ,\vol_e)$ as $(d \times e)$-matrix valued function $\vol: \R^d \to L(\R^e,\R^d)$.
By one of several equivalent definitions, e.g. \cite[Ch.8.7]{FH}, it is said that $\y$ solves the (c\`adl\`ag) {\it rough differential equation (RDE)} 
$$
                d\y = \drift (\ym)dt + \vol(\ym) d \bx
$$
iff, for all $0\le s < t \le T$, and $i=1,...,d$ one has\footnote{In coordinates,
$(D\vol (\y_s) \vol (\y_s) \mbx_{s,t})_i =\sum_{1\leq l,j\leq e}\sum_{1\leq k\leq d} \partial_k \vol_{ij} (\y_s) \vol_{kl}
(\y_s) \mbx_{s,t}^{lj}$.}
\begin{equation}                 \label{def:davieRDE}
               \y_t - \y_s = \drift(\y_s) (t-s) + \vol (\y_s) \x_{s,t} + D\vol (\y_s) \vol (\y_s) \mbx_{s,t} + R_{s,t} \ ,
\end{equation}
with small remainder $R$, in the sense that
$$ \sup_{\op (\eps)} \sum_{[s,t]\in \op (\eps)} | R_{s,t}|   \to 0  \ \ \text{ as $\eps \to 0$,} $$
with supremum taken over all partitions $\op (\eps)$ of $[0,T]$, with mesh-size less than $\eps$.
This definition first encodes that $(\y,\yp) = (\y,\vol(\y)) \in \mathscr{D}^{p/2}_{\x}$, is a controlled (c\`adl\`ag) rough path, cf. \cite[Ch.4]{FH}),\cite{FZ},
and ``$p/2$''-remainder
$\y_{s,t}^\# = D\vol (\y_s) \vol (\y_s) \mbx_{s,t} + R_{s,t}$ for which $\| \y^\# \|_{(p/2)\mvar} < \infty$.
As a consequence (of c\`adl\`ag rough integration theory
\cite{FS}), we can sum over $[s,t] \in \op$, which is a partition of $[0,u]$, and see that $\y$ satisfies a bona fide rough integral equation (RDE), for all $u \in (0,T]$,
\begin{equation} \label{equ:RDE1}
            \y_u = \y_0 + \int_{(0,u]} \drift (\ym_s) ds + \int_{(0,u]} \vol (\ym_s) d \bx .
\end{equation}
Conversely, (\ref{def:davieRDE}) is satisfied by every solution of this integral equation. We are interested in discrete-time approximations. Concerning the drift term this amounts to replace $d s$ by $d A^N$ with the step function $A_s^N := [s N] / N$. We note that $\Delta^N_s := s - A^N(s) \to 0$, uniformly, with the rate $1$, i.e. $\| \Delta^N \|_\infty = O(N^{-1})$. We also have uniform $1$-variation bounds on compacts. By an easy interpolation argument (see e.g \cite[Sec. 8.5]{FV}), $\| \Delta^N \|_{q,[0,T]}  \le \| \Delta^N \|_\infty^{1-1/q}  \| \Delta^N \|_{1,[0,T]}^{1/q}$, we see that we have $q$-variation convergence with the rate $(1-1/q)>0$, whenever $q>1$.
%
 This motivates to consider extensions to more general drift terms of the form
\begin{equation} \label{equ:RDE2}
                d\y = \drift (\ym)d \driftA + \vol(\ym) d \bx.
\end{equation}
with c\`adl\`ag $t \mapsto A(t)$ of finite $q$-variation. For $q \le p/2$ and bounded $\drift$, which suffices for our purposes, the contribution of this drift term can be absorbed in $\y^\#$, hence can be treated as a perturbation of the drift-free case. 
The following theorem gives (well-known) conditions for well-posedness and a quantitative, local Lipschitz estimate for the solution (a.k.a. It\^o-Lyons) map, comparing 
$\y$ to the solution of another RDE,
\begin{equation} \label{equ:RDE3}
           d\ty = \drift (\tym)d \driftB + \vol(\tym) d \tbx.
\end{equation}


\begin{thm}
\label{thm:RDEcont} Let $p \in [2,3), q \in [1,p/2]$ and $b, \sigma \in C^3_b$.
Then there exist unique c\`adl\`ag solutions to \eqref{equ:RDE2} and \eqref{equ:RDE3} with given initial data $\y_0$ and $\ty_0$. Moreover, the solution map is locally Lipschitz in the precise sense
 $$ \| \y - \ty\|_{p\mvar,[0,T]} \le C e^{C \ell^3} \big\{ \| \driftA - \driftB \|_{q\mvar,[0,T]} +   \| \bx - \tbx \|_{p\mvar,[0,T]}
+ |\y_0 - \ty_0| \big\}
$$
for some $C=C(p;b,\sigma)$, whenever
$$
\max \{ || \driftA ||_{q\mvar,[0,T]}, || \driftB ||_{q\mvar,[0,T]},
||| \bx |||_{p\mvar,[0,T]},
||| \tbx |||_{p\mvar,[0,T]}
\}
\le \ell
$$
\end{thm}
\begin{remark} (i) In our application $p>2$ so we can take $q>1$, as fits our needs. (ii) In a setting of continuous geometric rough paths this estimate is found in
\cite[Thm 10.26]{FV}. The c\`adl\`ag extension is found in \cite[Thm 3.9]{FZ} but only written in the drift free case $\drift \equiv 0$.
The Lipschitz estimate for c\`adl\`ag RDEs with drift appears in \cite{CFKMZ}, but without explicit dependence on $\ell$.
(ii) We have not pushed for optimal assumptions on $b, \sigma$. In the present form, this gives us the convenience, used in the proof below,  to reduce everything to the drift-free case.\end{remark}


\begin{proof} As noted before \eqref{def:davieRDE} we write $(\sigma_1, ... , \sigma_e) \leftrightarrow  \sigma$, i.e. identify the noise vector fields with the map $\R^d \ni y \mapsto ((\xi^i) \mapsto \sum_i \sigma_i (y) \xi^i) \in L(\R^e,\R^d)$. We can treat \eqref{equ:RDE2} and \eqref{equ:RDE3} as (drift-free) RDEs with vector fields
$$
     (b, \sigma_1, ... , \sigma_e)  \leftrightarrow   \sigma^\ext,
$$
with $\sigma^\ext (y) \in L(\R^{1+e},\R^d)$, driven by the $\bx^\ext$, the canonically defined rough path associated to $(\driftA,\x,\mbx)$, with all ``missing'' iterated integrals (between $A$ and components of $\x$) canonically defined in
the Young sense (see Section 4.1 in \cite{FH}). Moreover, standard estimates for Young integrals imply
$$
       ||| \bx^\ext |||_{p\mvar,[0,T]} \le c \Big( || \driftA ||_{q\mvar,[0,T]} + ||| \bx |||_{p\mvar,[0,T]} \Big),
$$
for some constant $c=c(p,q)$, as well as,
\begin{equation} \label{equ:aux1}
      \| \bx^\ext - \tbx^\ext \|_{p\mvar,[0,T]} \le c \ \ell \Big( \| \driftA - \driftB \|_{q\mvar,[0,T]} +   \| \bx - \tbx \|_{p\mvar,[0,T]}  \Big).
\end{equation}
The claimed estimates then follow by applying \cite[Thm 3.9]{FZ}.
%
\end{proof}

We state a corollary for families of RDEs, indexed by $N \in \mathbb{N}$, of the form
$$
    d\y_N = \drift (\ym_N)d \driftA_N + \vol(\ym_N) d \bx_N, \quad    d\ty_N = \drift (\tym_N)d \driftB_N + \vol(\tym_N) d \tbx_N
$$
with initial data $\y_N(0) = \y (0) $ and  $\ty_N(0) = \ty (0)$, respectively.

\begin{corollary} \label{cor:convwithrates} Let $p,q,b,\sigma$ be as in the previous theorem.
  Assume $\y (0) = \ty (0)$ and let there exist constants $C$ and $\delta >
  0$ such that for all $N \in \mathbb{N}$ we have
  \begin{equation}  \label{cond_A1}
    \| \driftA_N - \driftB_N \|_{q\mvar; [0, T]} +
  \| \bx_N - \tbx_N \|_{p\mvar; [0, T]} \le C N^{- \delta},
    \end{equation}
  as well as $$ \| \driftA_N \|_{q\mvar; [0, T]} + \| \driftB_N \|_{q\mvar; [0, T]}  \le C$$ and
    \begin{equation} \label{cond_A2}
   ||| \tbx_N |||_{p \mvar; [0, T]} \le C \log \log (N \vee 3).
   \end{equation}
  Then, for any $\delta' \in (0,
  \delta)$, and some constant $C'$ not dependent on $N$,
  \[ \sup_{0 \le t \le T} |  \y_N (t) - \ty_N (t) | \le C' N^{-
     \delta'}.  \]
     \end{corollary}
  \begin{proof}
    Without loss of generality $T = 1$ and $C \ge 1$. Apply Theorem \ref{thm:RDEcont} with $\bx = \bx_N, \tbx = \tbx_N$ and
 $$
          || \tbx_N |||_{p} + ||| \bx_N - \tbx_N |||_{p}  \le C \log \log (N \vee 3) + C =: \ell_N
$$
    It is easy to see that, for every $\eta > 0$, and as $N \to \infty$,
    $$
      C \exp ( C \ell_N^3)  = O (N^{\eta}).
    $$
    In combination with $ \| \bx_N - \tbx_N \|_{p\mvar; [0, T]} = O (N^{- \delta})$ the results follows.
  \end{proof}

A word on the assumptions of the previous corollary. With $ \driftA_N (s) := [s N] / N$ and $\driftB_N(s) \equiv s$, we already pointed out, as part of the motivation that led us to \eqref{equ:RDE2}, an easy interpolation argument that gives  $ \| \driftA_N - \driftB_N \|_{q\mvar; [0, T]} = O(N^{-(1-1/q)})$. More interestingly, the iterated-logarithmic bound \eqref{cond_A2} precisely holds
for families of rescaled Brownian rough paths, as we will now see.

\subsection{Brownian rough paths with parameters $(\Sigma,\Gamma)$}\label{subsec6.3} 

\subsubsection{Brownian rough paths} \label{sec:BrowRP}
Consider a $e$-dimensional Brownian motion
$\w = \w (\omega)$ with a given covariance
$\Sigma = E [ \w (1) \otimes \w (1) ] \in \R^e \otimes \R^e$.
Specializing \eqref{equ:Ito2ndlevel} to the present situation, we have the {\em It\^o Brownian rough path}
$\bw^\Ito = (\w, \mbw^{\Ito})$.
For any $\Gamma \in \R^e \otimes \R^e$,
we may then consider the second level perturbation  $\bw = (\w, \mbw)$, with
\begin{equation}     \label{equ:BRP}
\mbw (s, t) = \mbw^\Ito (s, t) + \Gamma (t - s)  = \int_s^t (\w (r) - \w (s)) \otimes d \w(r) + \Gamma (t - s).
\end{equation}
This yields a class of {\em Brownian rough paths}, with law determined by the parameters $(\Sigma, \Gamma)$. 
Such a Brownian rough paths is really a (Lie) group-valued Brownian motions in the sense that  $t \mapsto \bw(t) \in G$ has stationary and independent (group) increments, with Brownian scaling valid in the sense that $t \mapsto\delta_{\lambda} \bw (t / \lambda^2), \lambda >0 $, is again a Brownian rough path, equal in law to $\bw$.
(Dilation $\delta$ was introduced in \eqref{equ:dilation}.)
A familiar situation is $\Gamma = \frac 1 2 \Sigma$, the resulting Brownian rough path is then precisely the {\em Stratonovich Brownian rough path}, with
$$
            \mbw^{\mathrm{Strato};ij}(s,t) =  \int_s^t (\w^i (r) - \w^i (s)) \circ d \w^j(r) = \mbw^{\Ito;ij} (s, t) + \frac 1 2 \Sigma^{ij} (t - s).
$$
If we furthermore specialize to $\Sigma = \mathrm{Id}$, so that $\w$ is a {\em standard} Brownian motion, the Brownian rough paths with parameters
$(\mathrm{Id},0)$ (resp. $(\mathrm{Id}, \tfrac{1}{2}\mathrm{Id})$) will be referred to as {\em standard} It\^o- (resp. Stratonovich) Brownian rough paths.
See Chapter 4 in \cite{FH} for a detailed discussion.

\subsubsection{Differential equations driven by Brownian rough paths} \label{sec:DEBRP}
Call $Y=Y(\omega)$ the solution to the rough differential equation driven by a typical realization of the Brownian rough paths, that is
$$
                d\y = \drift (\y)dt + \vol(\y) d \bw = \drift (\y)dt + \vol(\y) d (\w, \mbw).
$$
It is well-known \cite[Theorem 9.1]{FH} that this yields a solution to the It\^o, resp. Stratonovich, stochastic differential equation whenever $\mbw = \mbw^{\Ito}$, resp. $\mbw^{\mathrm{Strato}}$. (This extends further to semimartingales, \cite[Ch.14]{FV}, \cite{CF}.) For a general Brownian rough path, with $\mbw (s, t) = \mbw^\Ito (s, t) + \Gamma (t - s)$ as given in \eqref{equ:BRP}, the very definition of a RDE solution \eqref{def:davieRDE}, with $(\x_{s,t}, \mbx_{s,t})$ replaced by $(\w_{s,t}, \mbw_{s,t})$ immediately shows that
\begin{equation} \label{equ:roughSDE}
                d\y = \drift (\y)dt + \vol(\y) d (\w, \mbw) = \tilde \drift (\y)dt + \vol(\y) d (\w, \mbw^{\Ito})
\end{equation}
with the drift vector field $\tilde \drift = \drift + c$ determined for $i=1,..,d$ by 
\begin{equation} \label{equ:star}
         c_i(y) = (D\vol (y) \vol (y) \Gamma)_i =\sum_{1\leq l,j\leq e}\sum_{1\leq k\leq d}  \partial_k \vol_{ij}
          (\y_s) \vol_{kl}(\y_s) \Gamma^{lj}.
\end{equation}
 The reason is simply that the defining second order term, part of the very definition \eqref{def:davieRDE}, expands as
$$
    \sum_{1\leq l,j\leq e}\sum_{1\leq k\leq d}  \partial_k \vol_{ij} (\y_s) \vol_{kl}(\y_s) \mbw_{s,t}^{lj} =
     \sum_{1\leq l,j\leq e}\sum_{1\leq k\leq d} \partial_k \vol_{ij} (\y_s) \vol_{kl}(\y_s) \mbw_{s,t}^{\Ito,lj} + c_i (t-s).
$$
Appealing again to \cite[Theorem 9.1]{FH}, we see that the (random) RDE solution to \eqref{equ:roughSDE}, with deterministic initial data is a strong solution to the It\^o SDE $dY = \tilde \drift (Y)dt + \vol (Y) d B$.


\subsubsection{Rescaling Brownian rough paths and LIL type estimates}

Consider now a Brownian rough path  $\bw = (\w, \mbw)$ with parameters $(\Sigma, \Gamma)$ and introduce
$$\mathbf{W}_N (t) =
\delta_{N^{- 1 / 2}} \bw (Nt)$$
for $N \in \mathbb{N}$. With $\mathbf{W}_N = (W_N, \mathbb{W}_N)$ this means
$ W_N (t) = N^{- 1 / 2} \w (N t)$ and
$$
     \mathbb{W}_N (s,t) =  N^{-1} \mbw (N s, Nt) = \int_s^t (W_N (r) - W_N (s)) \otimes d W_N (r) + \Gamma (t-s)
     $$
so that each $(W_N, \mathbb{W}_N)$ is a Brownian rough path with the same parameters $(\Gamma, \Sigma)$.
%

\begin{proposition} \label{prop:lil}
  For every $T > 0$, and $p>2$, 
  \[  ||| \bw_N |||_{p \mvar; [0, T]} \le C_T (\omega) \sqrt{ \log \log  (N \vee 3)}.  \]
      for some a.s. finite random variable $C_T$, and all $N \in \mathbb{N}$.
 \end{proposition}
  \begin{proof} By assumption, $\bw_N$ is obtained by scaling from $\bw = (\w,\mbw)$, a Brownian rough path with parameters $(\Sigma, \Gamma)$.
  It suffices to treat the case of standard Stratonovich Brownian rough path, i.e. $(\Sigma, \Gamma) = (\mathrm{Id},  \tfrac{1}{2}\mathrm{Id}))$, as introduced in Section \ref{sec:BrowRP}.
  Indeed, this reduction is easily obtained from writing $\w = \sqrt{\Sigma}B$ in terms of a standard $e$-dimensional Brownian motion $B$ and
  $$\mbw (s, t) = \int_s^t \w(s,r)  \otimes \circ d \w (r)  +
(\Gamma - \tfrac{1}{2} \mathrm{Id} ) (t - s)
$$
 so that, for some constant $c=c(\Sigma, \Gamma, T)$,
 $$ \| \mbw \|_{(p/2)\mvar,[0,T]} \le  c( \| \bbB \|_{(p/2)\mvar,[0,T]}+1)
 $$
 where
  $$\bbB (s, t) = \int_s^t (B(r)-B(s))\otimes \circ d B(r).
  $$

 We thus consider the case of the standard Stratonovich Brownian rough path from here on. This allows to use directly the Strassen law established in $p$-variation rough path topology \cite{LQZ} which states that, with iterated logarithm $\log_2 = \log \circ \log$,
 %
    \[ \mathbf{Z}_N : = \delta_{(2 N \log_2 N)^{- 1 / 2}}
       \mathbf{W} (N \cdot) \]
    is a.s. relatively compact, in the space of geometric $p$-rough paths. The set
    of limit points given the canonical lift of the {\em Cameron--Martin unit ball}, that is 
    $$
        \left\{  (H, \mathbb{H}): H: [0,T] \to \R^e, \text{absolutely continuous: }  H(0)=0, \int_0^T |\dot{H}(t)|^2 dt \le 1, \right\},
    $$
    where it is understood in the above that 
    $ \mathbb{H} (s, t) = \int_s^t  H(s,r)  \otimes \dot{H} (r) dr $. Using in particular Cauchy-Schwarz, 
    $$\| H \|_{p,[0,T]} \le \| H \|_{1,[0,T]} =  \int_0^T |\dot{H}(t)| dt \le  \sqrt{T} \Big(\int_0^T |\dot{H}(t)|^2 dt \Big)^{1/2}; $$ 
    also, the right-hand side of
     $|\mathbb{H} (s, t)|^{1/2} \le \| H \|_{1,[s,t]} \le \sqrt{|t-s|} \sqrt{\int_s^t |\dot{H}(r)|^2 dr}$ telescopes to $\sqrt{T}\sqrt{\int_0^T |\dot{H}(t)|^2 dt}$ upon summation over any partition of $[0,T]$, hence
        $$
              \| \mathbb{H} \|_{(p/2),[0,T]} \le \| \mathbb{H} \|_{(1/2),[0,T]} \le T \int_0^T |\dot{H}(t)|^2 dt.
    $$
    Restricting to $H$ in the Cameron--Martin unit ball,
$$
        ||| \mathbf{H} |||_{p\mvar,[0,T]} = \| H \|_{p,[0,T]} +  \| \mathbb{H} \|_{(p/2),[0,T]} ^{1/2}
    \le  2 \sqrt{T} \| \dot{H} \|_{L ^2} \le 2 \sqrt{T}.
$$
    By Strassen's law for the Brownian rough path \cite{LQZ}, a.s. with $N \to \infty$,
    \[ \inf_{\| \dot{H} \|_{L^2} \le 1} |||  \mathbf{Z}_N
        ; \mathbf{H} |||_{p\mvar,[0,T]} \rightarrow 0, \]
    so we can pick $(H_N)$ in the Cameron-Martin unit ball so that, a.s.
    $||| \mathbf{Z}_N ; \mathbf{H}_N |||_{p\mvar} \rightarrow 0$.
    But then
    \[ ||| \mathbf{Z}_N |||_{p\mvar,[0,T]} \le
       ||| \mathbf{Z}_N ; \mathbf{H}_N |||_{p\mvar,[0,T]}  + |||  \mathbf{H}_N |||_{p\mvar,[0,T]}  = O(1)
       \]
    so that $||| \mathbf{Z}_N |||_{p\mvar}
    \le C (\omega)$, for some a.s. finite random variable $C (\omega) = C_T (\omega)$. Now
    \[ ||| \mathbf{Z}_N |||_{p\mvar,[0,T]} = (2 \log_2 N)^{- 1 / 2} ||| \delta_{N^{- 1 / 2}} \mathbf{W}
       (N.) |||_{p\mvar,[0,T]} = (2 \log_2 N)^{- 1 / 2}
       ||| \mathbf{W}_N |||_{p\mvar,[0,T]} \]
    hence, absorbing $2^{1 / 2}$ in the constant $C (\omega),$ we have
    \[ ||| \mathbf{W}_N |||_{p\mvar,[0,T]} \le C
       (\omega) (\log_2 N)^{1 / 2} . \]
  \end{proof}

\begin{remark} \label{rem:LILeps}
Proposition \ref{prop:lil} holds with integer $N$ replaced by $1/\varepsilon$, as
a.s. estimate, uniform over all $\eps \in (0,1]$. A suitable ``continuous'' formulation of the Strassen
law for the Brownian rough path is found in \cite[Ex. 13.46]{FV}, always in conjunction with Brownian scaling
and the remark that H\"older - refines $p$-variation (rough path) topology.
\end{remark}


\subsection{Diffusion approximations}\label{subsec6.4}
\subsubsection{Discrete dynamics and proof of Theorem \ref{thm2.1}}

We rewrite \eqref{2.4} as
\begin{eqnarray}
&X_N((n+1)/N)=X_N(n/N)+\frac 1Nb(X_N(n/N))\\
&+\sig(X_N(n/N)) (S_N((n+1)/N)-S_N(n/N))\nonumber
\end{eqnarray}
and further as a c\`adl\`ag differential equation.  Specifically, we regard the rescaled partial sum process $S_N$ as piecewise constant (c\`adl\`ag) process and also write $t_N := [t N] / N$  so that
$$
      d X_N = \drift \Big( X_N^- \Big) d t_N
    + \vol ( X_N^-)
    d S_N.
$$
This equation makes sense (equivalently) as Riemann-Stieltjes integral equations and as (c\`adl\`ag) rough integral equation, written as
$$
     d X_N = \drift \Big( X_N^-\Big) d t_N
    + \vol ( X_N^-)
    d \mathbf{S}_N,
$$
where $\mathbf{S}_N = (S_N, \mathbb{S}_N)$ is the (pathwise) canonical lift of $S_N$, a piecewise constant c\`adl\`ag process.
The assumptions of Theorem \ref{thm2.1} guarantee, by Theorem \ref{thm2.2}, that
$$
    \|  \mathbf{S}_N ; \mathbf{W}_N \|_{p\mvar,[0,T]} = O(N^{-\delta}) \quad \text{ a.s. }
$$
where, in the terminology of Section \ref{sec:BrowRP}, we have that $\mathbf{W}_N = (W_N, \mathbb{W}_N)$ is a Brownian rough path, obtained by rescaling a universal Brownian rough path, with parameters $(\Sigma, \Gamma)$ identified in Theorem \ref{thm2.2}, with covariance $\Sigma = \varsigma$ from \eqref{2.8}, and $\Gamma =  \hat\vs$ from \eqref{2.8}, \eqref{2.9}, with components given by,
$$
               \qquad \Gamma_{ij} = \hat\vs_{ij}=\lim_{k\to\infty}\frac 1k\sum_{n=0}^k\sum_{m=-k}^{n-1} E(\xi_i(m)\xi_j(n)) = E \Big(\sum_{l =1}^\infty \xi_i (0) \xi_j (l) \Big),\ i,j = 1, ...,e.
$$
With Section \ref{sec:DEBRP}, we see that Proposition \ref{prop:lil} applies to the family $\mathbf{W}_N$ and we can conclude with Corollary \ref{cor:convwithrates} that
$$
\sup_{0\leq t\leq T}|X_N(t)-\Xi_N(t)|=O(N^{-\del})\quad\mbox{a.s. as}\quad N\geq 1
$$
where $\Xi_N$ is the unique solution of the rough differential equation
\begin{equation}\label{2.11sec6}
d\Xi_N(t)=\sig(\Xi_N(t))d \mathbf{W}_N (t)+ b(\Xi_N(t))dt = \sig(\Xi_N(t))d \mathbf{W}^{\Ito}_N (t)+\tilde b(\Xi_N(t))dt
\end{equation}
with $\tilde b = b + c$
where thanks to \eqref{equ:star}, $c = c(x)$ is given by 
$$
c_i(x)=\sum_{j,l=1}^e\sum_{k=1}^d\frac {\partial\sig_{ij}(x)}{\partial x_k}\hat\vs_{lj}\sig_{kl}(x),\, i=1,...,d.
$$ 

 By basic consistency results of stochastic and rough integration (\cite{CF}, also \cite[Ch. 5]{FH}) the (random) RDE solution
$\Xi_N$ is also the (unique) solution to the classical It\^o stochastic differential equations
$$
d\Xi_N(t) = \sig(\Xi_N(t))d W_N (t)+\tilde b(\Xi_N(t))dt.
$$

\subsubsection{Continuous dynamics and proof of Theorem \ref{thm2.4}} \label{ss:ProofThm24}

 We recall from Theorem \ref{thm2.5} 
 the definition
$$
      V^\ve(t)=\ve\int_0^{t\bar\tau\ve^{-2}}\xi(s)ds
 $$
 where $\bar\tau \in (0,\infty)$. Performing a deterministic time-change $t \to t/{\bar\tau}$ if necessary, we
 can assume $\bar\tau = 1$ and write \eqref{2.20} as
 $$
dX^\ve(t) = b(X^\ve(t))dt +
\sig(X^\ve(t)) dV^\ve,
$$
and further (equivalently) as rough integral equation
$$
   dX^\ve(t) = b(X^\ve(t))dt + \sig(X^\ve(t)) d\mathbf{V}^\ve ,\,\, t\in[0,T],
$$
where $\mathbf{V}^\ve = (V^\ve, \mathbb{V}^\ve)$ is the (pathwise) canonical lift of $V^\ve$, i.e. $\mathbb{V}^\ve (s,t) = \int_s^t (\delta V^\ve) (s,r) \otimes d V^\ve (r)$.
Theorem \ref{thm2.5} tells us precisely that
$$
    \|  \mathbf{V}^\ve ; \mathbf{W}^\ve \|_{p\mvar,[0,T]} = O(\ve^\del)\quad\mbox{a.s.}
$$
for some $\del>0$ a.s. taken simultaneously over $\ve\in(0,1)$. By construction, preceding equation \eqref{5.8}, the family of Brownian rough paths $\{\mathbf{W}^\ve: \ve\in(0,1) \}$ is obtained by rescaling a universal Brownian rough path. In the terminology of Section \ref{sec:BrowRP}, we have that $\mathbf{W}^\ve = (W^\ve, \mathbb{W}^\ve)$ is a Brownian rough path (by construction,  with parameters $(\Sigma, \Gamma)$ where the covariance $\Sigma=\vs$ is given by (\ref{2.22}) and $\Gamma$
comes from Theorem \ref{thm2.5}, i.e.
\[
               \Gamma_{ij} = \hat\vs_{ij}+E\int_0^{\tau(\om)}\xi_j(s,\om) ds \int_0^s \xi_i(u,\om) du, \ \ i,j=1,...,d.
\]

To describe the limiting dynamics consider, for each $\ve \in (0,1)$, the unique solution to the (random) rough differential equation
\begin{equation}\label{2.11sec6+}
d\Xi^\ve(t)=\sig(\Xi^\ve(t))d \mathbf{W}^\ve (t)+ b(\Xi^\ve(t))\bar\tau dt = \sig(\Xi^\ve(t))d \mathbf{W}^{\Ito,\ve} (t)+\tilde b(\Xi^\ve(t))dt
\end{equation}
with $\tilde b = b\bar\tau + c$ where, thanks to \eqref{equ:star}, 
$$ c_{i}(x)=\sum_{j,l=1}^e \sum_{k=1}^d \frac {\partial\sig_{ij}(x)}{\partial x_k}\big(\Gamma_{lj}\big )\sig_{kl}(x).
$$
The basic continuity result for RDEs, Theorem \ref{thm:RDEcont}, applies a fortiori to continuous $p$-variation rough paths, so that the arguments given in the c\`ad\`ag setting in the previous section, adapt immediately (cf. Remark \ref{rem:LILeps}) to the continuous setting.
 In particular, we see 
\begin{equation}\label{2.24sec6}
\sup_{0\leq t\leq T}|X^\ve(t)-\Xi^\ve(t/\bar\tau)|=O(\ve^\delta)\,\,\,\mbox{a.s.}
\end{equation}
We then remark that, by  basic consistency results of stochastic and rough integration (\cite[Ch. 5]{FH}), each process $\Xi^\ve$ is also the (unique) solution to the classical It\^o stochastic differential equation
$$
      d\Xi^\ve(t)= \sig(\Xi^\ve(t))d W^\ve (t)+\tilde b(\Xi^\ve(t))dt
$$
and we obtain (\ref{2.24}). \qed

\subsection{Euler--Maruyama approximation of the It\^o Brownian rough paths.}\label{subsec6.5} 

We recall the setup.   Let $\{ W_N : N \in \mathbb{N} \}$ be a family of Brownian motions defined
  on the same probability space and $\mathbf{W}_N = (W_N, \mathbb{W}_N)$ be
  the corresponding It\^o Brownian rough paths. Set
  \[ \hat{W}_N (t) = W_N ([t N] / N), \quad \hat{\mathbb{W}}_N (s, t) =
     \int_{(s, t]} (\hat{W}_N (r -) - \hat{W}_N (s)) \otimes d \hat{W}_N
     (r) \]
  so that $(\hat{W}_N, \hat{\mathbb{W}}_N)$ is the canonical (c\`adl\`ag) rough
  path lift of piecewise constant approximations to $W_N$. We now prove Lemma \ref{lem3.10},
  restated here for the reader's convenience.
\begin{lemma}
  For any $T>0$ and $p>2$,
there exists $\delta > 0$ such that, almost surely,
  \[ \| W_N - \hat{W}_N \|_{p, [0, T]} = O (N^{- \delta}), \qquad \|
     \mathbb{W}_N - \hat{\mathbb{W}}_N \|_{\frac{p}{2}, [0, T]} = O (N^{-
     \delta}) . \]
\end{lemma}

\begin{proof}
  We proceed in four steps. (1) We first get an $L^q$-version, any $q <
  \infty$, of these estimates in case $p = \infty$ where we recall $\| X
  \|_{\infty, [0, T]} = \sup | X (t) - X (s) |$ with sup taken over all $(s,
  t) \in \Delta_T := \{ (s, t) : 0 \le s \le t \le T
  \}$. (2) Uniform (in $N$) $p$-variation estimates, any $p > 2$. (3) An
  interpolation argument gives us $L^q$-estimates
     in $p'$-variation, any $p' >
  2$. (4) At last, the Borel-Cantelli lemma allows us to switch to a.s. convergence. (We
  insist that all $(W_N, \mathbb{W}_N)$ had identical law, so that in the steps
  (1)-(3) we could have written $(W, \mathbb{W})$. In the step (4) however, this
  notation is fully justified.)

  {\bf Step(1)} Write $t^- := [N t] / N, t^+ = t^- + 1 / N$ so
  that $\hat{W}_N (t) = W_N (t^-) .$ Trivially, $\hat{W}_N = W_N$ at times
  $t \in D_N := \{ t_i \equiv i / N : 0 \le i \le N T \}$.
  For arbitrary times $(s, t) \in \Delta_T$ we use the H\"older modulus, with
  exponent $\alpha < 1 / 2$, to see
  \begin{eqnarray*}
   \sup_{(s, t) \in \Delta_T} | (W_N (t) - W_N (s)) - (\hat{W}_N (t) -
     \hat{W}_N (s)) | &\le & 2 \sup_{t \in [0, T]} | W_N (t) - W_N (t^-)
     | \\
     & \le & 2 \| W_N \|_{\alpha ; [0, T]} (1 / N)^{\alpha} . \\
     \end{eqnarray*}
  By a classical result of Fernique (cf. below for a more general result with precise reference) the law of $\| W_N \|_{\alpha ; [0, T]}$
  (independent of $N$) enjoys Gaussian concentration, in the sense
  that $\mathbb{E} (e^{c \| W_N \|_\alpha^2}) < \infty$, for some $c = c (\alpha, T) > 0$.
  By expanding $\exp (.)$
  we see that the $L^q$-norm of any r.v. with Gaussian concentration is
  finite, and in fact, for a constant $C = C (\alpha, T)$,\footnote{Here and below we dependencies of constants w.r.t. $q$ are made explicit when easy to do so, 
 although this is not required for this proof.}
  \[ \| \| W_N - \hat{W}_N \|_{\infty, [0, T]} \|_{L^q (\Omega)} \le C
     \sqrt{q} (1 / N)^{\alpha} . \]
  For second level estimates, we first consider the case of partition points,
  i.e. $(s, t) \in \Delta_{T, N} := \Delta_T \cap D^2_N$. In this case,
  \[ \mathbb{W}_N (s, t) - \hat{\mathbb{W}}_N (s, t) = \sum_{i = m}^{n - 1}
     (\mathbb{W}_N) (t_i, t_{i + 1}) =: S (n) - S (m) \]
  noting that $S (n) \equiv \sum_{i = 0}^{n - 1} (\mathbb{W}_N) (t_i, t_{i +
  1})$ defines a random walk with centred independent increments, hence a
  discrete martingale. We then have
  \[ \sup_{(s, t) \in \Delta_{T, N}} | \mathbb{W}_N (s, t) -
     \hat{\mathbb{W}}_N (s, t) | \le \max_{1 \le n < [N T]} | S
     (n) | =: S^{\star} ([N T]) . \]
  To deal with non-partition points we focus on $s^- \le s < s^+
  \le t = t^-$. (The general case, with $t^- \le t < t^+$, is
  treated in the same fashion.) With Chen's relation, we have
  \[ (i) := \mathbb{W}_N (s, t) -\mathbb{W}_N (s^+, t) =\mathbb{W}_N (s,
     s^+) +  W_N (s, s^+) \otimes  W_N (s^+, t) + \]
  and hence (cf. footnote in Section \ref{sec:RP} on compatibility of norms on $\mathbb{R}^d \otimes
  \mathbb{R}^d$ and $\mathbb{R}^d$),
  \[ | (i) | \le \| \mathbb{W}_N \|_{2 \alpha, [0, T]} (1 / N)^{2
     \alpha} + \| W_N \|^2_{\alpha, [0, T]} (1 / N)^{\alpha} | t - s
     |^{\alpha} \]
  Similarly, using the very defininition of $(\hat{W}_N, \hat{\mathbb{W}}_N)$,
  we have
  \begin{eqnarray*}
    | (i i) | := | \hat{\mathbb{W}}_N (s, t) - \hat{\mathbb{W}}_N (s^+,
    t) | & = & | \hat{W}_N (s, s^+) \otimes \hat{W}_N (s^+,
    t) + \hat{\mathbb{W}}_N (s, s^+) |\\
    & = & |  W_N (s^-, s^+) \otimes \hat{W}_N (s^+, t) |\\
    & \le & \| W_N \|^2_{\alpha, [0, T]} (1 / N)^{\alpha} | t - s
    |^{\alpha}
  \end{eqnarray*}
  The terms $(i), (i i)$ account for the difference between $s$ and $s^+ \in
  D_N$. Similarly, one accounts for the difference between $t^- \in D_N$ and
  $t$ with terms $\widetilde{(i)}, \widetilde{(i i)}$, with identical
  estimate, so that \
  \begin{eqnarray*}
    | \mathbb{W}_N (s, t) - \hat{\mathbb{W}}_N (s, t) | & \le & |
    \mathbb{W}_N (s^+, t) - \hat{\mathbb{W}}_N (s^+, t) | + | (i) | + | (i i)
    | + | \widetilde{(i)} | + | \widetilde{(i i)} | .\\
    & \le & S^{\star} ([N T]) + 2 (\| \mathbb{W}_N \|_{2 \alpha, [0,
    T]} \frac{1}{N^{2 \alpha}} + 2 \| W_N \|^2_{\alpha, [0, T]} \frac{T^{\alpha}}{N^{\alpha}}
    ) .
  \end{eqnarray*}
  Using Doob's inequality, the fact that $S ([N T])$ is an element in
  the second Wiener-Ito chaos (and integrability properties thereof, see e.g.
  Theorem D.8 in \cite{FV}, 
  and finally independence of the
  $\mathbb{W}_N (t_i, t_{i + 1})$, we can bound
   \begin{eqnarray*}
  \| S^{\star} ([N T])\|_{L^q} & \le &
   \frac{q}{q - 1} \| S ([N T]) \|_{L^q} \lesssim q \| S ([N T])
     \|_{L^2} \\
     &=& q \left\| \sum_{i = m}^{[N T] - 1} \mathbb{W}_N (t_i, t_{i +
     1}) \right\|_{L^2} = q \sqrt{\sum_{i = 0}^{[N T] - 1} (t_i, t_{i + 1})^2}
     \le q \sqrt{\frac{T + 1}{N}}
    \end{eqnarray*}
  On the other hand, Fernique estimate for Brownian rough paths, Corollary
  13.14 in \cite{FV}, gives Gaussian concentration of $|||
  \mathbf{W}_N |||_{\alpha, [0, t]} = \| W_N \|_{\alpha, [0, T]}
  + \| \mathbb{W}_N \|^{1 / 2}_{2 \alpha, [0, T]}$ which implies a $O
  (q)$-bound for the $L^q$-norm of both $\| \mathbb{W}_N \|_{2 \alpha, [0,
  T]}$ and$\| W \|^2_{\alpha, [0, T]}$. Putting it all together, we see for
  some constant $C$ which does not depend on $N$ or $q$,
  \[ \| \sup_{(s, t) \in \Delta_T} | \mathbb{W}_N (s, t) - \hat{\mathbb{W}}_N
     (s, t) | \|_{L^q} \le C \frac{q}{N^{\alpha}} . \]
  {\bf Step(2)} From Proposition 6.17 in \cite{FZ} we can see
  that for some constant $C = C (p, T)$ but not dependent on $q$,
  \[ \sup_N \| \| \hat{\mathbb{W}}_N \|_{p / 2, [0, T]} \|_{L^q} \le C q
     < \infty . \]
  (The corresponding first level estimate is trivial in view of the
  $\omega$-wise estimate $\sup_N \|\hat W_N \|_{p, [0, T]} \le \|
  W \|_{p, [0, T]}$ and Gaussian concentration of the right-hand side.)

  {\bf Step(3)} We proceed by interpolation, as e.g. in Lemma 5.2 in \cite{FZ}.  
   Let $2 < p < p'$. In what follows, we spell out the second level
  estimates (the first level estimates are similar but easier),
   \begin{eqnarray*}
   \| \mathbb{W}_N - \hat{\mathbb{W}}_N \|_{p' / 2, [0, T]}
   & \le &
   \| \mathbb{W}_N - \hat{\mathbb{W}}_N \|^{1 - p / p'}_{\infty, [0, T]} \|
     \mathbb{W}_N - \hat{\mathbb{W}}_N \|^{p / p'}_{p / 2, [0, T]} \\
     & \le &
     \| \mathbb{W}_N - \hat{\mathbb{W}}_N \|^{1 - p / p'}_{\infty, [0, T]}
     \times K^{p / p'}_N (\omega)
     \end{eqnarray*}
  with $K^{p / p'}_N (\omega)$ can be taken as $2^{p / p'}$ times $\|
  \mathbb{W}_N \|^{p / p'}_{p / 2, [0, T]} + \| \hat{\mathbb{W}}_N \|^{p /
  p'}_{p / 2, [0, T]}$. Let $1 / q = 1 / q' + 1 / q''$ and apply H\"older's
  inequality to see, also setting $r' = q' (1 - p / p'), r'' = q'' p / p'$
  \begin{eqnarray*}
    \| \| \mathbb{W}_N - \hat{\mathbb{W}}_N \|_{p' / 2, [0, T]} \|_{L^q
    (\Omega)} & \le & \| \| \mathbb{W}_N - \hat{\mathbb{W}}_N \|^{1 - p
    / p'}_{\infty, [0, T]} \|_{L^{q'}} \| K^{p / p'} \|_{L^{q''}}\\
    & = & \| \| \mathbb{W}_N - \hat{\mathbb{W}}_N \|_{\infty, [0, T]}
    \|_{L^{r'}}^{1 - p / p'} \| K_N \|_{L^{r''}}^{p / p'} .
  \end{eqnarray*}
  Thanks to step (2), we have bounds on $\| K_N \|_{L^{r''}}$ which are uniform
  in $N$ and in fact such that $\sup_N \| K_N \|_{L^{r''}}\leq Cr''$ for some constant
  $C>0$ which does not depend on $N$ and $r''$. But then
  \[ \| \| \mathbb{W}_N - \hat{\mathbb{W}}_N \|_{p' / 2, [0, T]} \|_{L^q
     (\Omega)}\leq\tilde C(r' (1 / N)^{\eta})^{1 - p / p'} (r'')^{p / p'}
   \]
  for another constant $\tilde C>0$ which does not depend on $N$, $r'$ and $r''$.
  The precise choices are not that important but we can take $p = (p' + 2) /
  2$, $q' = q'' = q / 2$. In the end,
  \[ \| \| \mathbb{W}_N - \hat{\mathbb{W}}_N \|_{p' / 2, [0, T]} \|_{L^q
     (\Omega)} \le C (p') q (1 / N)^{\eta'} \]
  with $\eta' = \eta (1 - p / p') > 0.$

  {\bf Step(4)} A Borel-Cantelli argument then leads to the a.s.
  estimates. Indeed, for any $\varepsilon \in (0, \eta')$ pick $q > 1 /
  \varepsilon$ so that, from Chebyshev-Markov's inequality,
  \[ P \{ N^{\eta' - \varepsilon} \| \mathbb{W}_N - \hat{\mathbb{W}}_N \|_{p' /
     2, [0, T]} > 1\} \le N^{q (\eta' - \varepsilon)} E (\| \mathbb{W}_N
     - \hat{\mathbb{W}}_N \|^q_{p' / 2, [0, T]}) = O (N^{- q \varepsilon}) .
  \]
  Since $q \varepsilon > 1$, this bound is summable in $N$ and the
  Borel-Cantelli lemma tells us that $\| \mathbb{W}_N - \hat{\mathbb{W}}_N
  \|_{p' / 2, [0, T]} \le 1/N^{\eta' - \varepsilon}$ for all $N
  \geqslant N_0$ for some $N_0 = N_0 (\omega)$. The almost convergence thus
  holds with rate $\delta = \eta' - \varepsilon > 0$. (The first level
  estimates are similar.)
\end{proof}

\section{Rough paths and law of iterated logarithm for iterated sums and integrals}\label{roughsec7}
The present section is devoted to higher order extensions of Theorem \ref{thm2.2}. In conjunction
with a higher-order Strassen law for Brownian rough paths, also shown below, we arrive at
a (functional) law of iterated logarithm for iterated sums. We rely here on some
 higher-order concepts of rough paths. (Detailed references are given but without a
systematic review.)

\subsection{Lyons' extension for c\`adl\`ag rough paths}\label{subsec7.1}

Let $\mathbf{X} = (X, \mathbb{X})$ be c\`adl\`ag $p$-rough path, $p \in [2, 3)$.
One defines inductively iterated rough integrals
\[ \mathbf{\bar{X}}^\ell (s, t) = \int_{(s, t]} \mathbf{\bar{X}}^{\ell - 1}
   (s, r -) \otimes d \mathbf{X}(r-) \quad \in (\mathbb{R}^e)^{\otimes \ell} . \]
The entire stack $\mathrm{Ext}(\mathbf{X}) (s,t) := \mathbf{\bar{X}} (s, t) = (1, \mathbf{\bar{X}}^1 (s, t),
\ldots, \mathbf{\bar{X}}^{\ell} (s, t), \ldots)$, with values in the tensor series over $\R^e$, is known as {\em Lyons'
extension} of $\mathbf{X}$. It is equivalently given as $\mathbf{\bar{X}}
(s, t) = \mathbf{\bar{X}}^{- 1} (s) \otimes \mathbf{\bar{X}} (t)$, in
terms of a linear rough differential equation
\[ d \mathbf{\bar{X}} (t) = \mathbf{\bar{X}} (t -) \otimes d \mathbf{X}
   (t),\,\, \mathbf{\bar{X}} (0) = 1. \]
\begin{example}  \label{exam:sum}
  If $\mathbb{X} (s, t) = \int_s^t (X (r -) - X (s)) \otimes d X (r)$ for some
  c\`adl\`ag bounded variation path $X$, then, for all levels $\ell \geqslant 1$,
  \[ \mathbf{\bar{X}}^{\ell} (s, t) = \int_{\{ s \le r_1 \le
     \cdots \le r_{\ell} \le t \}} d X (r_1 -) \otimes \cdots
     \otimes d X (r_{\ell} -) = : \int_{\Delta_{s, t}^{\ell}} d X \otimes.
     \cdots \otimes d X \]
  (In case of a piecewise constant c\`adl\`ag path $X$, this becomes an iterated
  sum.)
\end{example}

\begin{example} In case of (It\^o, resp. Stratonovich) Brownian
  rough path, we have
  \[ \mathbf{\bar{W}}^{\Ito;\ell} (s, t) = \int_{\Delta_{s, t}^{\ell}} d W
     \otimes \cdots \otimes d W, \quad \mathbf{\bar{W}}^{\Strato;\ell} (s, t) =\quad \int_{\Delta_{s,
     t}^{\ell}} \circ d W \otimes \cdots \otimes \circ d W \]
  in the usual It\^o- resp. Stratonovich sense.
\end{example}

\begin{example} \label{exam:gBRP}
In case of a general Brownian rough path $\mathbf{W} = (W, \mathbb{W})$
  with parameters $(\Sigma, \Gamma),$ we can understand its Lyons extension
  $\mathbf{\bar{W}}  = \mathrm{Ext} (\mathbf{W})$ elegantly as solution
  to the It\^o linear rough differential equation
  \[ d \mathbf{\bar{W}} (t) = \mathbf{\bar{W}} (t) \otimes d W (t) +
     \mathbf{\bar W} (t) \otimes \Gamma d t, \quad \mathbf{\bar{X}} (0) = 1, \]
     followed by setting $ \mathbf{\bar{W}} (s,t)  =  \mathbf{\bar{W}} (s) ^{-1}\otimes  \mathbf{\bar{W}} (t)$.
    Given any word $w =(i_1 \cdots i_\ell) \in \{1,...,e\}^\ell$, of length $|w|=\ell$, and writing $e_w = e_{i_1 \cdots i_\ell}  = e_{i_1}  \otimes \cdots \otimes e_{i_\ell}$, the components
 of $ \mathbf{\bar W}^\ell \in (\R^e)^{\otimes \ell}$ also admit explicit combinatorial expressions, namely
  $$
                    \langle \mathbf{\bar{W}} (s,t) , e_w  \rangle  = \langle \mathbf{\bar{W}}^{\Ito} (s, t), e_w  \rangle    +  \sum_v c_v \langle \mathbf{\bar{W}}^{\Ito} (s, t), e_v \rangle
  $$
  with summation over all words $v$ obtained from $w$ by contracting one or more neighbouring pairs $(i_j,i_{j+1}) \in \{1,...,e\}^2$ to a single letter $0$, with the additional
  convention that $W_0 (t) = t$. (That is, $\mathbf{\bar{W}}^{\Ito}$ here should really be understood as the stack of iterated It\^o integrals of $(1+e)$-dimensional time-space Brownian motion $(W_0,W)$.) The constants $c_v$ are multiplicative functions of $\Gamma$. For instance, if $v$ is obtained by contracting, say, two pairs, $(i_j,i_{j+1}), (i_k,i_{k+1})$, with $1 < j+1 < k < \ell$, then $c_v = \Gamma_{i_j,i_j+1} \Gamma_{i_k,i_k+1}$.
This follows in exactly the same way as \cite[Prop. 22]{BCFP} and can be seen as algebraic renormalization procedure for rough paths.

  \end{example}

\begin{theorem} \label{thm:LyonsLiftLip}
  Let $\mathbf{X} = (X, \mathbb{X}), \tilde{\mathbf{X}} = (\tilde X, \tilde{\mathbb{X}})$ be
  c\`adl\`ag $p$-rough paths, $p \in [2, 3)$,
  with $$||| \mathbf{X} |||_{p, [0, T]} \vee ||| \tilde{\mathbf{X}} |||_{p, [0, T]}
  \le R \in [1,\infty).$$ Then, for every $\ell \in \mathbb{N}$ there exists $c = O(R^\ell)$, as $R \to \infty$,  
  such that
  \[ \| \mathrm{Ext}(\mathbf{X})^{\ell} -  \mathrm{Ext}(\tilde{\mathbf{X}})^{\ell}\|_{p / \ell,    [0, T]}
      \le c (\| X - \tilde X \|_{p, [0, T]} + \| \mathbb{X}-  \tilde{\mathbb{X}} \|_{p / 2, [0, T]}) . \]
\end{theorem}

\begin{proof}
  This is a variation of 
\cite[Thm 2.2.2]{Lyo}, see also \cite[Ex 4.6]{FH},  
  though what we need is not a direct consequence of these statements (which are
  given in terms of continuous control functions $\omega (s, t)$ resp. in a
  H\"older setting with $\omega (s, t) = t - s$).   We only illustrate the case
  $\ell = 3$, the general case being similar, giving a new argument based on
  local Lipschitz of higher-oder rough integration. (The case $\ell > 3$ goes
  along the same lines, cf. \cite[Sec 4.5]{FH}.)
  Recall that the space of
  (first order) controlled rough paths, $\mathcal{V}= (V, V') \in \mathscr{D}^{p /
  2}_X,$ is Banach with norm\footnote{Here and below, $(\delta X) (s, t) = X (t) - X (s)$ denotes the increments of
paths in a linear space. We also write, accordingly, $(\delta V - V' \delta X)
(s, t) = V (t) - V (s)  - V' (s) (X (t) - X (s))$, $(Z'' \mathbb{X}) (s, t) =
Z''  (s) \mathbb{X} (s, t)$ and so on.}
  $\| \mathcal{V} \|_{X ; p / 2} \equiv \| \delta V - V'
  \delta X \|_{p / 2} + \| V' \|_p .$ For a $p$-rough path $\mathbf{X} = (X,
  \mathbb{X})$ we have
  \[ \mathcal{V} \times \mathbf{X} \mapsto \left( \int (V, V')^- d (X,
     \mathbb{X}), V, V' \right) =: (Z, Z', Z'') =:
     \mathcal{Z} \in D_{\mathbf{X}}^{p / 3} \]
  where $\mathscr{D}_{\mathbf{X}}^{p / 3}$, the space of second order controlled rough
  paths, is Banach with norm
  \[ \| \mathcal{Z} \|_{\mathbf{X} ; p / 3} : = \| \delta  Z - Z' \delta X - Z''
     \mathbb{X} \|_{p / 3} + \| \delta  Z' - Z'' \delta X \|_{p / 2} + \| Z'' \|_p .
  \]
  Given another rough path $\widetilde{\mathbf{X}}$,
  the generic (local) Lipschitz
  estimate for rough integration gives
  \begin{eqnarray*}
    \| \mathcal{Z}; \tilde{\mathcal{Z}} \|_{\mathbf{X},
    \mathbf{\widetilde{\mathbf{X}}} ; p / 3} & \equiv & \| \delta Z - Z' \delta X
    - Z'' \mathbb{X}- (\tilde{Z} - \tilde{Z}' \delta \tilde{X} - \tilde{Z}''
    \tilde{\mathbb{X}}) \|_{p / 3}\\
    & + & \| \delta  Z' - Z'' \delta X - (\tilde{Z}' - \tilde{Z}'' \delta \tilde{X})
    \|_{p / 2}\\
    & + & \| Z'' - \tilde{Z}'' \|_p\leq c(\| X - \tilde{X} \|_p + \| \mathbb{X}-
     \tilde{\mathbb{X}} \|_{p / 2} + \| \mathcal{V}; \tilde{\mathcal{V}} \|_{X, \tilde{X} ; p / 2})
  \end{eqnarray*}
  where a constant $c$ can be taken uniformly provided $\mathbf{X},
  \widetilde{\mathbf{X}}$ and $\mathcal{V}, \tilde{\mathcal{V}}$ remain
  bounded in their approriate (rough resp. controlled rough path) spaces, and
  $\| \mathcal{V}; \tilde{\mathcal{V}} \|_{X, \bar{X} ; p / 2} \equiv \|
  V' - \tilde{V}' \|_p + \| \delta  V - V' \delta X - (\tilde{V} - \tilde{V}' \delta
  \tilde{X}) \|_{p / 2}$. We can now apply this with $(V, V') = (\mathbb{X},
  X)$, and $(\tilde{V}, \tilde{V}') = (\tilde{\mathbb{X}}, \tilde{X})$, the
  crucial remark being that in this case
  \[ \| \mathcal{V}; \tilde{\mathcal{V}} \|_{X, \tilde{X} ; p / 2} = \| X -
     \tilde{X} \|_p + \| \mathbb{X}- \tilde{\mathbb{X}} \|_{p / 2} . \]
  A moment of reflection (and Chen's relation) reveals that the $p / 3$-variation
  component of $\| \mathcal{Z}; \tilde{\mathcal{Z}} \|_{\mathbf{X},
  \mathbf{\widetilde{\mathbf{X}}} ; p / 3}$ is then nothing but the $p /
  3$-variation of the map
  \[ (s, t) \mapsto \int_s^t \mathbb{X} (s, r -) \otimes d X (r) - \int_s^t
     \tilde{\mathbb{X}} (s, r -) \otimes d \tilde{X} (r) ;
     \]
  we thus see that, for $\ell = 3$,
  \[
  \| \mathrm{Ext}(\mathbf{X})^{\ell} -  \mathrm{Ext}(\tilde{\mathbf{X}})^{\ell}\|_{p / \ell,    [0, T]}
 \le c
     (\| X - \tilde{X} \|_p + \| \mathbb{X}- \tilde{\mathbb{X}} \|_{p / 2}) .
  \]
The argument shows that $c$ can be taken uniformly as $\mathbf{X},\tilde{\mathbf{X}}$ remain in a bounded set, such
as a ball of radius $R$.
To make the dependence on $R$ explicit, we use scaling. Note that $\mathbf{Y} := \delta_{1 / R} \mathbf{X}$, and similar for $\tilde{\mathbf{Y}}$, are of (at most) unit size in the norm
$||| \cdot |||_{p, [0, T]}$. Application of the above estimate,
noting
$ \mathrm{Ext}(\mathbf{Y})^{\ell} = (1/R)^\ell  \mathrm{Ext}(\mathbf{X})^{\ell}$,
similarly for $\tilde{\mathbf{Y}}$, and
$$
      \| Y - \tilde{Y} \|_p  + \| \mathbb{Y}- \tilde{\mathbb{Y}} \|_{p / 2} = R^{- 1} \| X -
   \tilde{X} \|_p + R^{- 2} \| \mathbb{X}- \tilde{\mathbb{X}} \|_{p / 2}
   \le
   \| X - \tilde{X} \|_p +  \| \mathbb{X}- \tilde{\mathbb{X}} \|_{p / 2},
$$
using $R \ge 1$,
shows that $c$ can be taken as $O(R^\ell)$.

  \end{proof}

\subsection{The law of iterated logarithm for iterated sums}

\subsubsection{Almost sure invariance principle in rough paths metrics beyond level-$2$}

\begin{theorem} \label{thm:asIVPell}
  The conclusion of Theorem \ref{thm2.2}, with fixed $p \in (2, 3)$, can be extended to
  any level $\ell \in \mathbb{N}$. That is,
  \[ \| \overline{\mathbf{S}}_N^{\ell} \mathbf{} -
     \mathbf{\overline{W}}_N^{\ell} \|_{p / \ell, [0, T]} = O(N^{-\del})\quad\mbox{a.s.}
     \]
  where $\overline{\mathbf{S}}_N^{\ell} (s, t)$ is given by the rescaled
  $\ell$-fold iterated summation
  \[ \overline{\mathbf{S}}_N^{\ell} (s, t) = N^{- \ell / 2} \sum_{[N s]
     \le k_1 < \cdots < k_{\ell} < [N t]} \xi (k_1) \otimes \cdots
     \otimes \xi (k_{\ell}) \quad \in (\mathbb{R}^e)^{\otimes \ell} \]
  and $\mathbf{\overline{W}}_N = (1,
  \mathbf{\overline{W}}_N^1, \mathbf{\overline{W}}_N^2,
  \ldots)$ can be given as the solution of a ``drift-corrected'' It\^o stochastic
  differential equation
  \[ d \mathbf{\overline{W}}_N = \mathbf{\overline{W}}_N
     \otimes d W_N   + \mathbf{\overline{W}}_N \otimes \Gamma d t,
     \quad \mathbf{\overline{W}}_N (0) = 1, \]
with associated increments $ \mathbf{\overline{W}}_N (s, t) =
  \mathbf{\overline{W}}_N (s)^{- 1} \otimes
 \mathbf{\overline{W}}_N (t)$, and  driving Brownian motion $W_N (t) = N^{- 1 / 2} \mathcal{W} (N t)$.
  \end{theorem}
\begin{remark} Decomposing $\mathbf{\bar{W}}^\ell (s,t) = \sum \langle \mathbf{\bar{W}} (s,t) , e_w  \rangle e_w$ with sum over all words of length $|w| = \ell$, we note that an explicit combinatorial expression of these coefficients, as linear combinations of iterated It\^o-integrals of time-space Brownian motion, was given in Example \ref{exam:gBRP}.

\end{remark}

\begin{proof}
  By Theorem \ref{thm2.2} the claimed estimate holds true for $\ell = 1, 2$ so that a.s.
  \[ \| S_N - W_N \|_{p, [0, T]} + \| \mathbb{S}_N -\mathbb{W}_N \|_{p / 2,
     [0, T]} = O (N^{- \delta}). \]
     Thanks to $p \in (2,3)$, we can appeal to Theorem \ref{thm:LyonsLiftLip}, noting that the polynomial growth $c=c(R)$ therein is more than enough to allow us to proceed
     similarly to Section \ref{subsec6.4} (Proposition \ref{prop:lil} and Corollary \ref{cor:convwithrates}). It then suffices to recall that the Lyons extension
   $\mathrm{Ext} (S_N, \mathbb{S}_N)$ is precisely given by  $\overline{\mathbf{S}}_N$,  as was pointed out in Example \ref{exam:sum},
     and that the description of $\mathbf{\overline{W}}_N = \mathrm{Ext} (W_N, \mathbb{W}_N)$ is given in Example \ref{exam:gBRP}.
\end{proof}

\subsubsection{Strassen's functional LIL for Brownian rough paths}

A (possibly degenerate) covariance matrix gives rise to a (possibly
degenerate) inner product structure, $\langle \tilde{v}, \tilde{v}
\rangle_{\Sigma^{- 1}} : = \langle v, v \rangle = \sum_{i = 1}^e v_i^2$
when $\tilde{v} = \sqrt{\Sigma} v$ and $+ \infty$ else. Note that $\langle v, v
\rangle_{\Sigma^{- 1}} = \langle \Sigma^{- 1} v, v \rangle$ in the
non-degenerate case. An absolutely continuous path $H : [0, T] \rightarrow
\mathbb{R}^e$, with $H (0) = 0$, is called {\em Cameron--Martin path} if
$\| H \|^2_{\mathcal{H} ; [0, T]} := \int_0^T \langle \dot{H} (t), \dot{H} (t)
\rangle_{\Sigma^{- 1}} d t < \infty$. Every such path lifts canonically to a
rough path $(H, \mathbb{H})$ with $\mathbb{H}= \int \delta H \otimes d H$,
i.e. $\mathbb{H} (s, t) = \int_s^t (H (r) - H (s)) \otimes \dot{H} (r) d r$. We take $T=1$
in what follows and also need the Cameron--Martin unit ball, 
$$
 \mathcal{K} := \{ H \in \mathcal{H}: \| H \|_{\mathcal{H} ; [0, 1]} < \infty \}.
$$

Let now $\mathbf{W} = (W, \mathbb{W})$ be a Brownian rough path with parameters
$(\Sigma, \Gamma)$, with increments of its Lyons extension of the form
\[ \mathbf{\bar{W}} (s, t) = (1, \mathbf{\bar{W}}^1 (s, t), \ldots,
   \mathbf{\bar{W}}^{\ell} (s, t), \ldots). \]
   Set also
 \begin{equation} \label{equ:Kl}
     \mathbf{\mathcal{K}}^{\ell} := \left\{ 
     \mathbf{H}^{\ell} : H  \in {\mathcal{K}}
     \right\},
      \quad \mathbf{H}^{\ell} (s, t) = \int_{\Delta^{\ell}_{s,t}} d H \otimes \ldots \otimes d H.
  \end{equation}

    \begin{proposition} \label{prop:LILBRP}
  For any $\ell \in \mathbb{N}$ and $p > 2$, as $n \to \infty$, a.s.
  \[ \inf_{H \in {\mathcal{K}}} \left\| (2 n \log \log
     n)^{- \frac{\ell}{2}} \mathbf{\bar{W}}^{\ell} (n \cdot, n \cdot) 
     -
     \mathbf{H}^{\ell} 
     \right\|_{p / \ell
     ; [0, 1]} \to 0,\] 
     and the set of limit  points of the above sequence equals $\mathbf{\mathcal{K}}^{\ell}$. 
  \end{proposition}

\begin{remark}
  The same proof yields the same statement in stronger $\alpha \ell$-H\"older
  sense, $\alpha = 1 / p \in (1 / 3, 1 / 2)$.
\end{remark}

\begin{proof}
  Strassen's functional LIL for Brownian motion is a well-known consequence of
  Schilder's theorem. Typically formulated in $\infty$-topology (e.g. \cite{DeuStr}),
   extensions to iterated stochastic integrals \cite{Bal} and to $\alpha$-H\"older topology \cite{BBAK} have appeared in the
  literature.
  Similarly, a Schilder theorem for the Brownian rough
  path gives Strassen's law in rough path topology, as was first seen in the
  $p$-variation, then $\alpha$-H\"older rough path topology, see \cite{LQZ, FV} and references therein.

  Since the afore-mentioned results only deal with standard Brownian motion
  $B$ with $\mathbb{B} = \int \delta B \otimes \circ d B$ we quickly treat
  the case of a general $\mathbf{W} = (W, \mathbb{W})$, a Brownian rough
  path with parameters $(\Sigma, \Gamma)$. To this end, let $\sigma =
  \sqrt{\Sigma}$ and note that
  \[ W = \langle \sigma, B (t) \rangle,\, \mathbb{W} (s, t) = \left\langle
     \sigma \otimes \sigma, \int \delta B \otimes d B \right\rangle +
     \Gamma (t - s) = \langle \sigma \otimes \sigma, \mathbb{B} \rangle +
     \tilde{\Gamma} (t - s), \]
  with It\^o-Stratonovich corrected $\tilde{\Gamma} = {\Gamma} - \Sigma /
  2$. The map
  \[ (B, \mathbb{B}) \mapsto (W, \mathbb{W}) = \mathbf{W} \mapsto
     (\mathbf{\bar{W}}^1, \cdots, \mathbf{\bar{W}}^{\ell}) \equiv
     \mathbf{\bar{W}}^{\le \ell} \]
  is continuous between the appropriate rough path spaces, the contraction
  principle then shows that
  \[ \{ (\varepsilon \mathbf{\bar{W}}^1, \ldots, \varepsilon^{\ell}
     \mathbf{\bar{W}}^{\ell}) : \varepsilon > 0 \} \]
  satisfies a LDP with speed $\varepsilon^2$ in $p$-variation (or
  $1/p$-H\"older) rough path topology, with good rate function
  \[ I (\mathbf{H}) = \frac{1}{2} \| H \|^2_{\mathcal{H} ; [0, 1]} \]
  whenever $\mathbf{H} = \left( \mathbf{H}^1, \ldots, \mathbf{H}^\ell \right)$
is the canonical lift of $H$, and $+
  \infty$ else. Note that this rate function depends on $\Sigma$ but not on
  $\Gamma$. As in  \cite{LQZ, FV}, it then follows that
  Strassen's functional LIL holds in the stated higher order generality,
  \[ \left\| (2 n \log \log n)^{- \frac{\ell}{2}} \mathbf{\bar{W}}^{\ell} (n
     \cdot, n \cdot) ; \mathbf{\mathcal{K}}^{\ell} \right\|_{p / \ell ; [0,
     1]} \rightarrow_{n \rightarrow \infty} 0, \]
     together with the stated characterization of the limit points of this sequence. 
  (We used notation $\| x ; A \| = \inf_{y \in A} \| x
  - y \| .$)
\end{proof}

\subsubsection{Strassen's theorem for iterated sums} Consider
\[ \mathbf{\bar{\xi}}^{\ell} (m, n) := \sum_{m \le k_1 < \cdots
   < k_{\ell} < n} \xi (k_1) \otimes \cdots \otimes \xi (k_{\ell}) \]
and recall that $\mathcal{K} \subset \mathcal{H}$ defines the unit Cameron--Martin ball. 
%

\begin{theorem} For any $\ell \in \mathbb{N}$ and $p > 2$, as $N \to \infty$, a.s.
  \begin{equation} \label{equ:Str1} 
  \inf_{H \in \mathcal{K}} \left\| (2 N \log \log
     N)^{- \frac{\ell}{2}} \mathbf{\bar{\xi}}^{\ell} ([N \cdot], [N \cdot]) 
     -
    \mathbf{H}^\ell
    \right\|_{p / \ell
     ; [0, 1]} \rightarrow 0 
     \end{equation} 
    and the set of limit points given by $\eqref{equ:Kl}$. 
   In particular, 
   \begin{equation} \label{equ:Str2} 
    \inf_{H \in \mathcal{K}} \left| (2 N \log \log
     N)^{- \frac{\ell}{2}} \sum_{0 \le k_1 < \cdots < k_{\ell} < N} \xi
     (k_1) \otimes \cdots \otimes \xi (k_{\ell}) 
      -
      \mathbf{H}^\ell(0,1)
    \right| \to 0,
     \end{equation} 
     with  set of limit points given by 
     $ \left\{ 
     \mathbf{H}^{\ell}(0,1) : H  \in {\mathcal{K}} \right\}$.
     \end{theorem}

\begin{proof}
  (i) Recall $\overline{\mathbf{S}}_N^{\ell} (s, t) = N^{- \ell / 2}
  \mathbf{\bar{\xi}}^{\ell} ([N s], [N t])$ is precisely the Lyons lift of
  $(S_N, \mathbb{S}_N) .$ Recall also $\mathbf{\overline{W}}_N^{\ell}
  (s, t) = N^{- \ell / 2} \mathbf{\overline{W}}^{\ell} (N s, N t)$.
  Then $\overline{\mathbf{S}}_N^{\ell} \mathbf{} (\cdot, \cdot) -
  \mathbf{\overline{W}}_N^{\ell} (\cdot, \cdot) = O (N^{- \delta})$
  from Theorem \ref{thm:asIVPell} above shows
  \[ \mathbf{\bar{\xi}}^{\ell} ([N \cdot], [N \cdot]) -
     \mathbf{\overline{W}}^{\ell} (N \cdot, N \cdot) = O (N^{\ell /
     2 - \delta}), \]
  always in $p / \ell$-variation sense on $[0, 1]$, hence
  \[ (2 N \log \log N)^{- \frac{\ell}{2}} \mathbf{\bar{\xi}}^{\ell} ([N
     \cdot], [N \cdot]) - (2 N \log \log N)^{- \frac{\ell}{2}}
     \mathbf{\overline{W}}^{\ell} (N \cdot, N \cdot) = o (1) \]
  and so the functional LIL for $\mathbf{\bar{\xi}}^{\ell}$,
  as stated in \eqref{equ:Str1}, follows directly from the one for
  $\mathbf{\overline{W}}^{\ell}$
  in Proposition \ref{prop:LILBRP} above. 
As for \eqref{equ:Str2} it suffices to note that any variation norm on $[0, 1]$
  dominates the increment over the unit interval.
\end{proof}

The following corollary can be seen as generalization of \cite[Cor. 3.2]{Bal} 
which dealt with iterated Brownian integrals.

\begin{corollary} \label{cor:BaldiGen}
  Let $A \in (\mathbb{R}^e)^{\otimes \ell}$ and define the tensor contraction
  $\mathcal{X}^{\ell}_N := \left\langle A, \sum_{0 \le k_1 < \cdots
  < k_{\ell} < N} \xi (k_1) \otimes \cdots \otimes \xi (k_{\ell})
  \right\rangle$ with values in the reals. Then
  \[ P \left( \limsup_{N \rightarrow \infty} \frac{\mathcal{X}^{\ell}_N}{(2 N
     \log \log N)^{\frac{\ell}{2}}} = M \right) = 1 \]
  with
  \begin{equation} \label{equ:M} 
  M = \sup \left\{ \left\langle A, \int_{\Delta_{0, 1}^{\ell}} d H
     \otimes \ldots \otimes d H \right\rangle : \| H \|_{\mathcal{H} ; [0, T]}
     \le 1 \right\} . 
 \end{equation}
\end{corollary}

\begin{proof} Immediate from \eqref{equ:Str2} and accompanying description of the limit set. 
\end{proof} 


\subsection{The law of iterated logarithm for iterated integrals}

The arguments of the last section immediately extend to the case of iterated integrals, and in particular lead to
a proof of Corollary \ref{cor2.6}. The analogue of Theorem \ref{thm:asIVPell} reads
\begin{theorem} \label{thm:asIVPellint}
  The conclusion of Theorem \ref{thm2.5}, with fixed $p \in (2, 3)$, can be extended to
  any level $\ell \in \mathbb{N}$. That is,
  \[ \| \overline{\mathbf{V}}_{\ell}^\eps  -
     \mathbf{\overline{W}}_{\ell}^\eps \|_{p / \ell, [0, T]} = O(N^{-\del})\quad\mbox{a.s.}
     \]
  where $ \overline{\mathbf{V}}_{\ell}^\eps$ is given by the rescaled  $\ell$-fold iterated integrals,
  $$ \overline{\mathbf{V}}_{\ell}^\eps (s,t) = \ve^\ell 
    \int_{\{ s\bar\tau\ve^{-2} \le r_1 \le
     \cdots \le r_{\ell} \le t\bar\tau\ve^{-2} \}} \xi (r_1) \otimes \cdots
     \otimes \xi (r_{\ell}) d r_1 \cdots d r_\ell
  $$
  and $\mathbf{\overline{W}}_N = (1,
  \mathbf{\overline{W}}_N^1, \mathbf{\overline{W}}_N^2,
  \ldots)$ exactly as in Theorem \ref{thm:asIVPell}, just with updated covariance for the Brownian motion $\mathbf{\overline{W}}_N^1 = W_N$, namely the covariance given in Theorem \ref{thm2.5}, and $\Gamma = \{ \Gamma_{ij}: 1 \le i,j \le e \}$ given by
$$
   \Gamma_{ij} = \sum_{l=1}^\infty E(\eta_i(0)\eta_j(l))+
 E \Big( \int_0^{\tau(\om)}\xi_j(s,\om)ds\int_0^s\xi_i(u,\om)du \Big).
$$
  \end{theorem}
\begin{proof} Similar to Theorem \ref{thm:asIVPell}: by Theorem \ref{thm2.5} the claimed estimate holds true for $\ell = 1, 2$ and thanks to $p \in (2,3)$, we can appeal to Theorem \ref{thm:LyonsLiftLip}, noting that the Lyons extension of $V^\eps$ is precisely given by $\overline{\mathbf{V}}^\eps = (V^\eps, \mathbb{V}^\eps)$. %
\end{proof}

We can now deduce, as in the discrete case, a functional LIL for iterated integrals from the corresponding statement for Brownian rough paths; Proposition \ref{prop:LILBRP}. We set
  $$
     \mathbf{\bar{\xi}}^{\ell} (s, t) :=  \int_{\{ s \le r_1 \le
     \cdots \le r_{\ell} \le t \}} \xi (r_1) \otimes \cdots
     \otimes \xi (r_{\ell}) d r_1 \cdots d r_\ell \quad \in (\mathbb{R}^e)^{\otimes \ell}.
  $$

\begin{theorem} For any $\ell \in \mathbb{N}$ and $p > 2$, as $N \to \infty$, a.s.
  \[ \inf_{H \in \mathcal{K}} \left\| (2 N \log \log
     N)^{- \frac{\ell}{2}} \mathbf{\bar{\xi}}^{\ell} (\bar\tau N \cdot, \bar\tau N \cdot) -
     \mathbf{H}^{\ell}
     \right\|_{p / \ell
     ; [0, 1]} \rightarrow 0 \]
     and the set of limit points given by $\eqref{equ:Kl}$. 
  In particular, 
  \[  \inf_{H \in \mathcal{K}} \left| (2 N \log \log
     N)^{- \frac{\ell}{2}}
      \mathbf{\bar{\xi}}^{\ell} (0, \bar\tau N )
      -
     \mathbf{H}^{\ell}(0,1)
 \right| \to 0.
     \]
     with  set of limit points given by 
     $ \left\{ 
     \mathbf{H}^{\ell}(0,1) : H  \in {\mathcal{K}} \right\}$.
     \end{theorem}


We have, as before, 

\begin{corollary} \label{cor:BaldiGenInt}
  Let $A \in (\mathbb{R}^e)^{\otimes \ell}$ and define the real-valued tensor contraction
  $\mathcal{X}^{\ell}_N := \left\langle A,
  \mathbf{\bar{\xi}}^{\ell} (0, N )
  \right\rangle . $ 
  Then, with $M$ given in \eqref{equ:M}, 
  \[ P \left( \limsup_{N \rightarrow \infty} \frac{\mathcal{X}^{\ell}_N}{(2 N
     \log \log N)^{\frac{\ell}{2}}} = M / \bar\tau \right) = 1. \]
\end{corollary}

\subsection{Remarks on iterated sums and integrals}

Iterated integrals and sums of the type considered in Corollaries \ref{cor:BaldiGen} and \ref{cor:BaldiGenInt} are of interest
in data science. Specifically, iterated integrals have given rise to a popular feature set of machine learning applications, the lectures notes \cite{CK} constitute an excellent source of information. Iterated sums, a.k.a. {\em iterated-sums signatures} are a natural variant, specifically for feature extraction of time-series, see e.g. 
 \cite{BKPASL, DEFT, CGGOT}. Strictly speaking, they allow for additional integer
powers of the $\xi$'s. An extension of Corollary \ref{cor:BaldiGen} in this direction is not difficult, e.g. using results of  \cite{FZ}, but this would require an algebraic setup in terms of quasi-shuffle or Grossmann--Larson Hopf algebras
that would lead us too far astray. 


\end{document}